%% file: main.tex
\documentclass{article}
\usepackage[utf8]{inputenc}
\usepackage[english]{babel}
\usepackage{quiver}
\usepackage{bm}
\usepackage[margin=3.5cm]{geometry}
\usepackage{tikz-cd,  xfrac}
\usepackage[outline]{contour}
\usetikzlibrary{decorations.text}

\usepackage[hidelinks, colorlinks=true, citecolor=red, linkcolor=black, urlcolor=blue!50]{hyperref}  
\usepackage{amsmath,mathtools, amsfonts, amssymb, amsthm, stmaryrd, extarrows}
\usepackage{thmtools}
\usepackage{thm-restate} 
\usepackage{dsfont}
\usepackage{faktor} 
\usepackage{chngcntr}
\usepackage[toc, page]{appendix}

\usepackage{bbm}

\usepackage{datetime}
\newdateformat{monthyeardate}{\THEDAY\text{ }\monthname[\THEMONTH] \THEYEAR}

\usepackage{xcolor}
\usepackage{todonotes}

\newcommand{\myemph}[1]{\emph{#1}}

\usepackage[ruled,linesnumbered]{algorithm2e}
\usepackage{comment}
\SetKwComment{Comment}{// }{}
\DontPrintSemicolon
\rmfamily\mdseries\scshape

\usepackage{fontawesome}

\newtheoremstyle{exampstyle2}
  {13pt} 
  {13pt} 
  {\itshape} 
  {} 
  {\bfseries} 
  {.} 
  {.5em} 
  {} 
\theoremstyle{exampstyle2}

\newtheorem{proposition}{Proposition}[]
\newtheorem{theorem}[proposition]{Theorem}

\newtheorem{corollary}[proposition]{Corollary}
\newtheorem{lemma}[proposition]{Lemma}
\newtheorem{definition}[proposition]{Definition}
\newtheorem{example}[proposition]{Example}

\theoremstyle{remark}
\newtheorem*{remark}{Remark}
\newtheorem{observation}[proposition]{Observation}

\newcommand{\nc}{\newcommand}
\usepackage{bm}
\nc{\twoquote}{``}
\nc{\onequote}{`}
\nc{\cB}{\mathcal{B}}
\nc{\C}{\mathcal{C}}
\nc{\D}{\mathcal{D}}
\nc{\bF}{\mathbb{F}}
\nc{\Ff}{\mathbb{F}} 
\nc{\F}{\mathcal{F}}
\nc{\cF}{\mathcal{F}}
\nc{\G}{\mathcal{G}}
\nc{\gb}{\genfrac{[}{]}{0pt}{}}
\nc{\GL}{\mathsf{GL}}
\nc{\Gr}{\mathsf{Gr}}
\nc{\myfont}[1]{{\normalfont\sffamily\bfseries #1}}
\nc{\N}{{\bm{\mathsf{N}}}}
\nc{\Nn}{\mathbb{N}} 

\nc{\sS}{ \mathscr{S}}
\nc{\sC}{ \mathscr{C}}

\nc{\BF}{\cB^{\cF}}
\nc{\SF}{\cS^{\cF}}
\nc{\BH}{\cB^{H}}
\nc{\SH}{\cS^{H}}

\nc{\cH}{{\mathcal H}}
\nc{\cR}{{\mathcal R}}
\nc{\cA}{{\mathcal A}}
\nc{\cG}{{\mathcal G}}
\nc{\cC}{{\mathcal C}}
\nc{\cD}{{\mathcal D}}
\nc{\cO}{{\mathcal O}}
\nc{\cI}{{\mathcal I}}
\nc{\cY}{{\mathcal Y}}
\nc{\cK}{{\mathcal K}} 
\nc{\cX}{{\mathcal X}}
\nc{\cS}{{\mathcal S}}
\nc{\cE}{{\mathcal E}}
\nc{\cZ}{{\mathcal Z}}
\nc{\cQ}{{\mathcal Q}}
\nc{\cP}{{\mathcal P}}
\nc{\cL}{{\mathcal L}}
\nc{\cM}{{\mathcal M}}
\nc{\cT}{{\mathcal T}}
\nc{\cW}{{\mathcal W}}
\nc{\cU}{{\mathcal U}}
\nc{\cJ}{{\mathcal J}}
\nc{\cV}{{\mathcal V}}
\nc{\bH}{{\mathbb H}}
\nc{\bA}{{\mathbb A}}
\nc{\bG}{{\mathbb G}}
\nc{\bC}{{\mathbb C}}
\nc{\bO}{{\mathbb O}}
\nc{\bI}{{\mathbb I}}
\nc{\bB}{{\mathbb B}}
\nc{\bY}{{\mathbb Y}}
\nc{\bK}{{\mathbb K}} 
\nc{\bX}{{\mathbb X}}
\nc{\bS}{{\mathbb S}}
\nc{\bE}{{\mathbb E}}
\renewcommand{\bF}{{\mathbb F}}
\nc{\bZ}{{\mathbb Z}}
\nc{\bQ}{{\mathbb Q}}
\nc{\bN}{{\mathbb N}}
\nc{\bP}{{\mathbb P}}
\nc{\bL}{{\mathbb L}}
\nc{\bM}{{\mathbb M}}
\nc{\bT}{{\mathbb T}}
\nc{\bW}{{\mathbb W}}
\nc{\bU}{{\mathbb U}}
\nc{\bD}{{\mathbb D}}
\nc{\bJ}{{\mathbb J}}
\nc{\bV}{{\mathbb V}}
\nc{\bR}{{\mathbb R}}

\nc{\bfone}{\mathbf{1}}

\nc{\ppf}{\mathsf{ppf}}
\nc{\PR}{\mathbb{P}}
\nc{\pc}{\scalebox{1.1}{\(\wp\)}}

\newcommand{\spn}[1]{\left\langle #1 \right\rangle}

\newcommand{\withoutzero}{\setminus\{0\}}
\newcommand{\wt}{\operatorname{wt}}
\newcommand{\wtH}{\operatorname{wt}_{\operatorname{H}}}
\newcommand{\dist}{d}
\newcommand{\dH}{\dist_{\operatorname{H}}}
\newcommand{\dF}{\dist_{\cF}}
\newcommand{\stab}{\operatorname{Stab}}
\newcommand{\Span}{\operatorname{Span}}

\newcommand{\Vdim}{N}

\newcommand{\isomfun}{L}
\newcommand{\isomFun}{T}
\newcommand{\lift}[1]{\hat{#1}}
\newcommand{\Qwt}[1]{\wt_{\operatorname{quot},#1}}
\newcommand{\Hsubspace}{C}
\newcommand{\HSubspace}{\Hsubspace}
\newcommand{\HSubspaceQuotient}{ \left(\varphi,\Hsubspace, \rho, \bF_q^r \right)}
\newcommand{\HSubspaceQuotientNum}[1]{ \left(\varphi_{#1},\Hsubspace_{#1}, \rho_{#1}, \bF_q^{r_{#1}} \right)}
\newcommand{\PSubspace}{ \left(\iota,  W, \psi, \Ff_q^r \right)}
\newcommand{\PSubspaceNum}[1]{ \left(\iota_{#1}, W_{#1}, \psi_{#1},\Ff_q^{r_{#1}}\right)}
\nc{\proj}{\operatorname{\psi}}
\newcommand{\eqclass}[1]{\left[#1\right]}

\newcommand{\isom}{\operatorname{Isom}}
\newcommand{\aut}{\operatorname{Aut}}
\newcommand{\Fisom}[3]{\aut_{\cF}(#3:#1,#2)}
\newcommand{\Hisom}[3]{\aut_{\operatorname{H}}(#3:#1,#2)}

\newcommand{\Fixisom}[4]{\isom_{\cF}(#3,#4:#1,#2)}
\newcommand{\Hixisom}[4]{\isom_{\operatorname{H}}(#1,#2:#3,#4)}
\renewcommand{\ker}{\operatorname{Ker}}

\nc{\miso}{F}
\newcommand{\cfgcap}[3]{\langle #1 \cap #2 \rangle \cap \langle #1 \cap #3 \rangle}
\newcommand{\fgcap}[2]{\langle #1 \rangle \cap \langle #2 \rangle}

\title{\vspace{-2cm}Fundamental Notions of Projective and\\ Scale-Translation-Invariant Metrics in Coding Theory\vspace{-0.0cm}}
\date{}

\newcommand{\printthis}[2][false]{%
	\ifbool{#1}{%
		#2%
	}{%
	}%
}
\makeatletter

\renewcommand{\author}[1]{\renewcommand{\@author}{\color{\@authorcolor}#1}}
\newcommand{\@authorcolor}{black}
\newcommand{\authorcolor}[1]{\renewcommand{\@authorcolor}{#1}}
\makeatother

\author{  
Gabor Riccardi$\text{}^1$, Hugo Sauerbier Couvée$\text{}^2$
\text{}\vspace{0.3cm}
\\
\normalsize{ $\text{}^1$University of Pavia}\\
\normalsize{$\text{}^2$Technical University of Munich}\\
\\
\normalsize{Emails: gabor.riccardi01@universitadipavia.it, hugo.sauerbier-couvee@tum.de}
}

\begin{document}

\maketitle

\input{Sections/1-ArXiv}
\input{Sections/appendix}

\newpage
\nocite{*}
\bibliographystyle{plain}
\bibliography{biblio}

\end{document}

%% file: Sections/1-ArXiv.tex
\begin{abstract}

Projective metrics on vector spaces over finite fields, introduced by Gabidulin and \mbox{Simonis} in 1997, 
generalize classical metrics in coding theory like the Hamming metric, rank metric, and combinatorial metrics. 
While these specific  metrics have been thoroughly investigated, the overarching theory of projective metrics has remained underdeveloped since their introduction.
In this paper, we present and develop the foundational theory of projective metrics, establishing several elementary key results on their characterizing properties, equivalence classes, isometries, constructions, connections with the Hamming metric, associated matroids, sphere sizes and Singleton-like bounds. Furthermore, some general aspects of scale-translation-invariant metrics are examined, with particular focus on their embeddings into larger projective metric spaces.
%
%
%
%
%
\end{abstract}

\section{Introduction}

Projective metrics
in a vector space \(V\) are defined by a subset of vectors \(\cF\) in \(V\). They were first introduced by Gabidulin and Simonis in \cite{gabidulin1997metrics, gabidulin1998metrics} as a generalization of some already well-studied metrics, like the \emph{Hamming metric}, the \emph{rank-metric}, \emph{combinatorial metrics} and \emph{tensor rank metric}  \cite{roth1996tensor, byrne2019tensor, byrne2021tensor}. 
An overview of the current state of research on the topic can be found in general surveys on metrics, such as \cite{gabidulin2012brief, firer2021alternative}.
Part of the work done on projective metrics focuses on the relation between the projective metric defined by \(\cF\) and the Hamming metric considered on the vector space of dimension \(|\cF|\). 
This gives rise to a code \(\pc\), called the \emph{parent code} of the projective metric.
Conversely, given any linear code \(\pc\) in \(V\), there is a corresponding projective metric whose parent code is \(\pc\). 
In \cite{gabidulin1997metrics}, it is shown that the projective weight of an element can be seen as the Hamming weight of a corresponding coset of the code \(\pc\), and that the projective weight distribution is the same as the coset weight distribution of \(\pc\), although these have been known to be hard to compute in general  \cite{helleseth1979weight, jurrius2009extended}. 
In \cite{gabidulin1998metrics}, a characterization of perfect codes with respect to the projective metric is given for the case where the minimum Hamming distance of $\pc$ is larger than its covering radius.
In \cite{gabidulin1997metrics, gabidulin1998metrics}, a class of projective metrics is investigated with parent codes that asymptotically attain the Gilbert-Varshamov bound. 
%
%
Several other metrics were later also studied for the first time in the context of projective metric, e.g. the \emph{term-rank} or \emph{cover metric} in \cite{gabidulin1998metrics}, the \emph{phase rotation metric} in \cite{gabidulin1998hard}, metrics based on BCH codes in \cite{gabidulin1998perfect}, and the \emph{Vandermonde $\cF$-metric} based on GRS codes in \cite{gabidulin2003codes, catterall2006public}, in which also fast decoders and cryptographic applications have been proposed.

\paragraph{Our contribution} Although a few specific projective metrics and the subfamily of combinatorial metrics have been studied in-depth, very little literature has been written on the general theory of projective metrics after its introduction. Our contributions are the following.

First, we introduce the notion of quotient spaces in the context of translation-invariant \newpage
\noindent
metrics or weights and describe some category-theoretical properties. 
Using quotients we extend the groundwork of the theory, show additional fundamental characteristics of projective metrics, and demonstrate their equivalence to metrics defined by sets of subspaces and to integral, convex, scale-translation-invariant metrics. We further show that any scale-translation-invariant metric can be embedded into a suitably larger projective metric space.  


Second, we make the relation between projective weights and their parent codes (already briefly noted in \cite{gabidulin1998metrics}) more precise, showing a bijective correspondence between equivalence classes of projective metrics and classes of Hamming-isometric subspaces of suitably large vector spaces. 
Additionally, we give a characterization of the isometries of projective metrics as the stabilizer of their parent codes.

Third, we explore ways of constructing new metrics from old ones and consider bounds on codes for projective metrics. In particular, we provide a  Singleton-like bound, and with the focus on sphere-packing bounds we investigate sphere sizes. We demonstrate that these are fully determined by an underlying matroid structure regarding \emph{extended spanning families}, i.e. if two projective metrics have matroid equivalent extended spanning families, then they have equal sphere sizes.

Throughout the paper we provide numerous examples, counterexamples, and a Sage script to facilitate these constructions. We also consider various connections to other fields of mathematics such as graph theory and matroid theory, highlighting potential directions for future research.

\paragraph{Structure of the paper}
\begin{enumerate}
\setcounter{enumi}{1}
    \item[Section  \ref{sec: Preliminaries on metrics}] In this section we present preliminary material on metrics. Subsection \ref{subsec: metrics and weight function} recalls the basic definitions of weights and metrics, while Subsection \ref{subsec: weighted vector spaces and contractions} introduces contractions and quotient metrics derived from other weight functions.
    
    \item[Section \ref{sec: introduction to projective metics}] In this section, we provide an in-depth overview of the definitions and characterizing properties of projective metrics. Subsection \ref{subsec: Why} motivates their study and contextualizes them among other translation- and scale-invariant metrics. Subsection \ref{subsec: projective weight basic proprieties} gives their formal definition and outlines key properties. In Subsection \ref{sec: examples}, we supply an extensive list of  examples.

    \item[Section \ref{sec: Projective metrics as quotients}] In this section, we connect projective metrics to quotient weights. In Subsection \ref{subsec: parent codes and parent function} we make explicit the bijective correspondence between equivalence classes of projective metrics and equivalence classes of linear codes under Hamming isometries. 
    In Subsection \ref{subsec: Isometries}, we characterize the isometry group of a projective metric as the subgroup of Hamming isometries that fix the parent code. We also prove that every quotient of the Hamming metric yields a projective metric. In Subsection \ref{subsec: decoding algorithm}, we show that a decoding algorithm for the parent code gives an algorithm for computing the projective weight.

    \item[Section \ref{section: embedding in projective metrics}] In this section, we show that any scale-invariant metric can be isometrically embedded into a projective metric defined on a larger vector space.

    \item[Section \ref{sec: Sphere sizes}] In this section, we study sphere sizes in projective metrics. In Subsection \ref{subsec: some obs} we observe that the parent function behaves like an isometry for small enough distances, we analyze how sphere sizes behave under unions  and we connect them to coset weight distributions. In Subsection \ref{subsectiond: perfect codes}, we derive sufficient conditions for codes to be perfect with respect to a projective metric. Finally, in Subsection \ref{subsec: matroids}, we introduce extended spanning families and prove that the matroid associated to such a family fully determines the sphere sizes.

    \item[Section \ref{sec: Singleton Bound}] In this section, we establish a Singleton-type bound for projective metrics and provide several properties and examples.

\newpage

\end{enumerate}

\text{}

\subsection*{Notation}
\noindent
Throughout this paper will use the following notation:\vspace{-0.5em}\\ 
\\
\noindent
\renewcommand{\arraystretch}{1.5}
\begin{tabular}{l l}
$\Nn$ & The set of the natural numbers $\{0,1,2,\ldots\}$;\\
$\Nn^+$ & The set of the positive natural numbers $\{1,2,3,\ldots\}$;\\
$[n]$ &  The set of positive natural numbers $\{1,2,\ldots,n\}$;\\
$|X|$ & The cardinality of a finite set $X$;\\
$\mathcal{P}(X)$ & The power set of $X$, i.e. the set of all subsets of $X$;\\
$\Ff$ & A finite field;\\
$\Ff_q$ & The (up to isomorphism unique) finite field of order $q$, with $q$ a prime power.\vspace{-0.2cm}\\ & For prime $q$, the elements of $\Ff_q$ will be represented by the integers $0,\ldots,q-1$;\\
$V$ & A finite dimensional vector space over a finite field. Usually we will consider $V = \Ff_q^N$;\\
$U \leq V$ & $U$ is a linear subspace of $V$;\\
$\dim_\Ff(U)$ & The dimension of $U$ as vector space over $\Ff$;\\
$\spn{X}$ & The linear span of a subset of vectors $X \subseteq V$. The span of a singleton $\{x\} \subseteq V$ is\vspace{-0.2cm}\\ & denoted by $\spn{x}$. If the span is taken with scalars from a specific field $\Ff$, we write $\spn{\cdot}_{\Ff}$.\vspace{-0.2cm}\\ & By definition $\spn{\emptyset} = \spn{0} = \{0\}$;\\
$\Gr_k(V)$ & The set of all $k$-dimensional linear subspaces of $V$, called a \emph{Grassmannian}.\\
$\GL(V)$ & The \emph{general linear group} of $V$, consisting of all invertible linear maps $L: V \to  V$. \\
$\cI(N,n)$ & Is the set of ordered indices \(\{(i_1,\ldots, i_n) \in \{1,\ldots,n\} |\; i_1 < \ldots < i_n\}\)
\end{tabular}

\text{}\\
\\
The word `linear' is often omitted when this is clear from the context.\\

\newpage

\newpage
\section{Preliminaries on metrics}
\label{sec: Preliminaries on metrics}

In this section, we consider any finite abelian group $(A,+)$. Later we will be particularly interested in finite dimensional vector spaces over a finite field $\Ff$.

The set $\bR \cup \{\infty\}$ will be endowed with the usual addition $+$ and order $\leq$ where as convention we have for all $r \in \bR \cup \{\infty\}$:
\begin{itemize}
    \item $r + \infty = \infty$ 
    \item $r \leq \infty$
\end{itemize}

\subsection{Metrics and weight functions}\label{subsec: metrics and weight function}

\begin{definition} \label{def: metric}
A \myfont{metric} or \myfont{distance function} on $A$ is a function $d(\cdot,\cdot) : A \times A \to \bR \cup \{\infty\}$  satisfying the following properties for all $x,y,z \in A$:\\
\\
\renewcommand{\arraystretch}{1.5}
\begin{tabular}{lll}
    (M0) & Non-negativity: & $d(x,y) \geq 0$   \\
    (M1) & Identity of indiscernibles: &  $d(x,y) = 0 \; \Leftrightarrow\;  x = y$\\
    (M2) & Symmetry: & $d(x,y) = d(y,x)$ \\
    (M3) & Triangle inequality: & $d(x,y) \leq d(x,z) + d(z,y)$
\end{tabular}
\noindent
\text{}\\
\vspace{0.2cm}
\text{}\\
In the case that $A$ is a module over a commutative ring $R$ (e.g. a vector space over a field $\Ff$), a \myfont{scale-invariant metric} on $A$ moreover satisfies the additional property\\~\\
\begin{tabular}{lll} \label{def: scale invariance}
    (M4) & Scale invariance: &  \quad \quad \quad \,\, $d(\alpha x, \alpha y) = d(x,y)$ \quad  for all $\alpha \in R\withoutzero$
\end{tabular}
\text{}\\
\vspace{0.2cm}
\text{}\\
Likewise, a \myfont{translation-invariant metric} is a metric satisfying \textit{(M0)-(M3)} and additionally\\~\\
\begin{tabular}{lll} \label{def: translation invariance}
    (M5) & Translation invariance: &  \quad $d(x + z, y + z) = d(x,y)$
\end{tabular}
\end{definition} 

\noindent In coding theory, it is regularly needed to measure the distance from $x$ to $x+e$ for given $x,e \in A$. This value $d(x,x+e)$ is called the \myemph{weight} of $e$ with respect to $x$. Since one often deals with a translation-invariant metric $d$, the distance $d(x,x+e)$ is in that case independent of $x$ and equals $d(0,e)$. We then call the function $\wt(\cdot) := d(0,\cdot)$ a \myemph{weight function}.
\begin{definition}
\label{def: weight function}
A \myfont{weight function} on $A$ is a function $\wt(\cdot) : A \to \bR \cup \{\infty\}$  satisfying the following properties for all $x,y \in A$:\\
\\
\renewcommand{\arraystretch}{1.5}
\begin{tabular}{lll}
    (W0) & Non-negativity: & $\wt(x) \geq 0$   \\
    (W1) & Positive definiteness: &  $\wt(x) = 0 \; \Leftrightarrow\;  x = 0$\\
    (W2) & Symmetry: &  $\wt(x) = \wt(-x)$ \\
    (W3) & Triangle inequality: & $\wt(x+y) \leq \wt(x) + \wt(y)$
\end{tabular}
\noindent
\text{}\\
\vspace{0.2cm}
\text{}\\
In the case that $A$ is a module over a commutative ring $R$ (e.g. a vector space over a field $\Ff$), a \myfont{scale-invariant weight function} on $A$ moreover satisfies this additional property:\\~\\
\begin{tabular}{lll}
    (W4) & Scale invariance: &  \quad \,\, $\wt(\alpha x) = \wt(x)$ \quad for all $\alpha \in R\withoutzero$
\end{tabular}
\end{definition} 

\newpage
\noindent
Note that \textit{(W0)} is automatically implied by \textit{(W1)-(W3)} since $0 = \wt(0) \leq \wt(x) + \wt(-x) = 2 \wt(x)$, and if \textit{(W4)} is satisfied, \textit{(W2)} is as well.
We can now easily describe a bijective relation between (scale-invariant) weight functions and (scale-invariant) translation-invariant metrics: 
\begin{itemize}
    \item Given a translation-invariant metric $d(\cdot,\cdot)$, its corresponding weight function is defined as \[\wt(x) := d(0,x).\]
     \item Given a weight function $\wt(\cdot)$, its corresponding translation-invariant metric is defined as \[d(x,y) := \wt(y-x).\]
\end{itemize}
Another way to view this correspondence is by the principle that we can describe a metric `up to its symmetries', in this case translations and/or scaling. 
When we introduce a metric by defining a weight function, we assume (unless stated otherwise) that the corresponding metric is the translation-invariant metric as defined above.

When a metric or weight function takes values in $\Nn \cup \{\infty\}$, it is called an \myfont{integral metric} or \myfont{integral weight function} respectively. Likewise, a metric or weight function is \myfont{finitely valued} if it takes values solely in $\bR$ (and not $\infty$).

\subsection{Weighted vector spaces and contractions}\label{subsec: weighted vector spaces and contractions}

Groups/modules/vector spaces/etc. endowed with a weight function are referred to as \myfont{weighted} groups/modules/vector spaces/etc. We refer the reader to \cite{Grandis2007} for an overview of these objects in a categorical context.
Of particular interest to us are \textit{finite dimensional} weighted vector spaces over \textit{finite fields}, although some of the theory in this section is also applicable to other vector spaces.

Let $(V, \wt_V)$, $(W,\wt_W)$ be finite weighted vector spaces over $\Ff_q$. A linear map $f: V \to W$ is called a \myfont{(weak) contraction} if $ \wt_V(v) \geq \wt_W(f(v))$ for all $v \in V$. In the special case that $\wt_V(v) = \wt_W(f(v))$ for all $v \in V$, then  $f$ must be injective and we call $f$ a \myfont{(linear) 
 embedding}. 
If a linear $f : V \to W$ is both an embedding and a bijection, then $f$ is a \myfont{(linear) isometry} or \myfont{isomorphism}, and the finite weighted vector spaces $(V, \wt_V)$ and $(W,\wt_W)$ are called \myfont{(linearly) isometric} or \myfont{isomorphic} indicated with the notation $(V, \wt_V) \cong (W,\wt_W)$. 
Note that linear isometries coincide with the usual category-theoretic notion of isomorphisms, i.e., contractions \(f : V \to W\) for which there exists a contraction \(g : W \to V\) such that \(g \circ f = \operatorname{id}_V\) and \(f \circ g = \operatorname{id}_W\), where \(\operatorname{id}_V\) and \(\operatorname{id}_W\) are the identity maps on \(V\) and \(W\), respectively.

Whenever $V \leq W$ is a linear subspace, we endow  $V$ (unless stated otherwise) with the weight $\wt_W|_V$ of $W$ restricted to $V$, and call $V$ simply a \myfont{subspace} (in the metric sense) of $W$.\\

We highlight an important elementary construction of weight functions on finite vector spaces that are induced by surjective linear maps:


\begin{definition}
\label{def: quotient weight}
Let \(\varphi: X \to Y\) be a \textit{surjective} linear map from a finite weighted vector space \((X, \wt_X)\) to a finite vector space \(Y\). The \textbf{quotient weight/metric} on \(Y\) (with respect to \(\varphi\)) is defined as
\[
\wt_{\operatorname{quot},\varphi}(y) := \min\{\wt_X(w) \ : \ w \in \varphi^{-1}(y)\}
\]
for all \(y \in Y\). 
\end{definition}
\medskip

\begin{remark}
We can naturally generalize to quotients with respect to any linear map $\varphi$ by setting $\Qwt{\varphi}(y) := \infty$ whenever $\varphi^{-1}(y)$ is empty. 
\end{remark}

\newpage

\begin{observation} \label{obs: quotient proprerties}
Note that the quotient weight satisfies the following easy-to-check properties:
\begin{enumerate}
\item $ \wt_{\operatorname{quot},\varphi}$ is a weight function satisfying  \textit{(W0) - (W3)}.
\item If $\wt_X$ is scale-invariant satisfying \textit{(W4)}, then so is $\wt_{\operatorname{quot},\varphi}$.
\item If $\wt_X$ is integral, then so is $\wt_{\operatorname{quot},\varphi}$.

    \item $\varphi$ is a contraction from $(X, \wt_X)$ to $(Y,\wt_{\operatorname{quot},\varphi})$.


\end{enumerate}
\end{observation}
\noindent
Moreover, Theorem \ref{thm: universal property} states that the quotient weight satisfies a universal property similar to quotients of other types of objects, e.g. groups, modules, vector spaces, topological spaces, etc.
\begin{restatable}[]{theorem}{quotientweightproperties}\label{thm: universal property}
The quotient weight \(\wt_{\operatorname{quot},\varphi}\) with respect to a surjective linear map \(\varphi: (X, \wt_X) \twoheadrightarrow Y\) satisfies the following:
\end{restatable}
{\itshape
\begin{enumerate}
    \item \textbf{(Maximality)} The quotient weight \(\wt_{\operatorname{quot},\varphi}\) is the largest weight on \(Y\) making \(\varphi\) a contraction: for any other weight \(\wt_Y\) such that \(\varphi: (X,\wt_X) \to (Y,\wt_Y)\) is a contraction, the identity map \(\operatorname{id}: (Y, \wt_{\operatorname{quot},\varphi}) \to (Y, \wt_Y)\) is a contraction.

    \item \textbf{(Universal Property)} If \(f: (X, \wt_X) \to (Z, \wt_Z)\) is a contraction with \(\ker(\varphi) \subseteq \ker(f)\), then there exists a unique contraction \(g: (Y, \wt_{\operatorname{quot},\varphi}) \to (Z, \wt_Z)\) such that \(f = g \circ \varphi\). Moreover,
    \[
    \ker(g) = \varphi(\ker(f)) \quad \text{and} \quad \wt_{\operatorname{quot},f} = \wt_{\operatorname{quot},g}.
    \]
    This property is summarized in the following commuting diagram:
    \[
    \begin{tikzcd}
    (X, \wt_X) \arrow[r, "\varphi", two heads] \arrow[rd, "f"'] & (Y, \wt_{\operatorname{quot},\varphi}) \arrow[d, dashed, "\exists! \, g"] \\
    & (Z, \wt_Z)
    \end{tikzcd}
    \]
    \item (\textbf{Isometry Property})\label{prop: Isometry Property} If \(\ker(\varphi) = \ker(f)\) and \(f\) is surjective, then \(g: (Y, \Qwt{\varphi}) \to (Z, \Qwt{f})\) as defined in (2) is an isometry. 
\end{enumerate}
}

The proof can be found in Appendix \ref{appendix: quotient weights}. 

\begin{corollary}\label{cor: composition of quotients}
Let  $\varphi : X \to Y$, $\psi: Y \to Z$ be surjective linear maps and consider the diagram\\
\[
\begin{tikzcd}
{(X,\wt_X)} \arrow[r, "\varphi", two heads] \arrow[rr, "\psi \, \circ \, \varphi", two heads, bend right, shift right] & {(Y, \wt_{\operatorname{quot}, \varphi})} \arrow[r, "\psi", two heads] & Z 
\end{tikzcd}
\]
of finite weighted vector spaces, with $ \wt_X$ any weight function on $X$. Then the quotient weights on $Z$ with respect to $\psi$ and $\psi \circ \varphi$ are equal, i.e.
\[
(Z, \wt_{\operatorname{quot}, \psi}) = (Z, \wt_{\operatorname{quot}, \psi \, \circ \, \varphi}).
\]
\end{corollary}
This follows directly from applying the universal property to $\psi$, $\varphi$ and $\psi \circ \varphi$, and deducing that the identity maps between $(Z, \wt_{\operatorname{quot}, \psi})$ and $(Z, \wt_{\operatorname{quot}, \psi \, \circ \, \varphi})$ are contractions in both directions and thus isomorphisms.

\newpage

\section{Introduction to Projective Metrics}\label{sec: introduction to projective metics}

In this section we will give a formal introduction to projective metrics: in subsection \ref{subsec: Why}, we explore the motivation behind projective metrics, anticipating some of the results in this paper.  In subsection \ref{subsec: projective weight basic proprieties}, we give the definition of projective metrics and some basic properties. In addition, we discuss here and in later sections other equivalent definitions: 
\medskip

\begin{center}
\renewcommand{\arraystretch}{2.0}
\begin{tabular}{|c|c|}
\hline
Description & Statement \\
\hline
\hline
Minimum size subsets of family of projective points   &  Definition \ref{def: proj weight}\\
\hline
Minimum size subsets of family of subspaces &  Proposition \ref{prop: projective point form}\\
\hline
Integral convex scale-translation-invariant metrics & Proposition \ref{prop: projective iff convex scale trans}\\
\hline
Distances on Cayley graphs on vector spaces & Proposition \ref{prop: proj iff cayley}\\
\hline
Quotients of Hamming metric spaces & Proposition \ref{prop: proj iff quotient 1} \vspace{-4mm}\\
by subspaces of minimum distance $\geq 3$ & \\
\hline
Quotients of Hamming metric spaces & Corollary \ref{cor: reduction to parent functions} \vspace{-4mm} \\
by subspaces of any minimum distance & \\
\hline
Coset leader weights & Proposition \ref{prop: decoding weight algorithms eq}\\
\hline
\end{tabular}
\end{center}
\text{}
  
Lastly, in subsection \ref{sec: examples} we provide extensive examples of projective metrics and of operations for constructing new projective metrics from old.

\subsection{Why projective metrics?} \label{subsec: Why}

As we will later see in the examples in subsection \ref{sec: examples}, many of the classical metrics such as Hamming metric, rank metric, discrete metric, combinatorial metric etc. are projective metrics. In this section we further justify the study of projective metrics by showing that a metric endowed with some \onequote natural' properties are projective metrics: while metrics as defined in Definition \ref{def: metric} are very general, the definition of metrics used in practice often arises from the type of errors which occur on a specific communication channel. 

For example, let $V$ be a finite dimensional vector space over a finite field $\Ff_q$, and consider the transmission of vectors from $V$ over some noisy channel that adds error vectors. For every vector \(v \in V\), let \(\cE(v) \subset V\) be the set of possible `elementary' errors (e.g. a one symbol change) which may occur at \(v \in V\). Any error vector added during transmission is a sum of these elementary errors, and we always get a metric on $V$ by defining the distance between two vectors \(v_1,\;v_2 \in V\) as the shortest path of elementary errors needed to transform one vector into the other. 

More precisely, a path of elementary errors from $v_1$ to $v_2$ is a sequence $(e_1,e_2,\ldots,e_n)$ such that $e_i \in \cE(v_1+\sum_{j=1}^{i-1}e_i)$ for all $i =1,\ldots,n$  and  $v_2 = v_1 + \sum_{i=1}^ne_i$.  A distance function or metric corresponding to these elementary errors is then defined as 
\[d(v_1,v_2) \coloneqq \min\{n \in \bN\mid \text{ there exists an elementary error path from $v_1$ to $v_2$ of length $n$} \}. \] Metrics defined this way will always be integral and, since any path of errors between \(v_1\) and \(v_2\) of length \(d(v_1,v_2)\) is minimal by definition of \(d\), we have that \(d\) is an integrally convex metric:

\newpage

\begin{restatable}[]{definition}{convexity}
    \label{def: convex}
    An integral metric $d(\cdot,\cdot)$ is \myfont{(integrally) convex} if one of the following equivalent conditions holds:
    \begin{enumerate}
        \item For all $v_1,v_2 \in V$ with $d(v_1,v_2) < \infty$:
        \begin{align*} 
        d(v_1,v_2) = &  \min\{n \in \Nn \ \mid  \text{ there exists a sequence } (v_1 = x_0,x_1,\ldots,x_{n}=v_2)   \text{ with $x_i \in V$ }  \\  &\text{}   \qquad\qquad\qquad\qquad\qquad \text{ such that } \ d(x_k, x_{k+1}) = 1  \ \text{ for all } \  k = 0, \ldots, n-1 \}.
        \end{align*}
        \item  For all $v_1,v_2 \in V$ with $d(v_1,v_2) < \infty$ and all $i \in \{0,1,\ldots,d(v_1,v_2)\}$, there exists an $x \in V$ such that
    \[
    d(v_1,x) = i \quad\text{ and }\quad d(x,v_2) = d(v_1,v_2)-i.
    \]
    \end{enumerate}
\end{restatable}
See Appendix \ref{appendix: correction, detection and convexity} for the proof of the equivalence of these two conditions.  
\medskip

An alternative approach (see also \cite{silva2008rank}) to defining a metric in practice is to capture how much error can occur while still allowing correct decoding of a received vector, or at least detecting that an error has occurred beyond a recoverable threshold. We will briefly discuss some concepts informally here, and refer to Appendix \ref{appendix: correction, detection and convexity} for all formal definitions, statements, and proofs.

The \myemph{correction capability} \(\tau(v_1, v_2)\) (see Definition \ref{def: correction-normal}) is a measure of the maximum error magnitude for which, given a received vector \(v\), we can always determine whether the original transmitted vector was \(v_1\) or \(v_2\). For many metrics $d$ there is a direct relationship between \(\tau(v_1, v_2)\)  and \(d(v_1, v_2)\): we say that \(d\) is \myemph{correction-normal} if for all $v_1,v_2 \in V$ 
\[
\tau(v_1, v_2) = \left\lfloor \frac{d(v_1, v_2) - 1}{2} \right\rfloor
\]
holds. Although well-known metrics are often correction-normal, there are useful metrics (e.g. \cite{bitzer2024}) for which the above equality does not hold.

We are also interested in the number of errors \(\sigma(v_1, v_2)\) for which we can detect that an error has occurred but cannot determine whether the original vector was \(v_1\) or \(v_2\). Specifically, given a metric \(d(v_1, v_2)\) associated with the error, \(\sigma(v_1, v_2)\) is exactly equal to the minimal \(d(v, v_1)\), where \(v \in V\) is such that \(d(v, v_1) = d(v, v_2)\), provided such a \(v\) exists. Analogous to the Euclidean case, if the points equidistant from \(v_1\) and \(v_2\) lie on the intersection of spheres centered at \(v_1\) and \(v_2\) with radius equal to half the distance between them (when the distance is even; otherwise, the intersection is empty since the distance is an integer), then we say that the metric \(d(\cdot,\cdot)\) is \emph{equal-detection normal}. For any greater number than \(\sigma(v_1,v_2)\) of errors, \(v\) may become closer to \(v_2\) even if the original vector was \(v_1\). Conversely, for any smaller number of errors, \(v\) would remain within the ball of radius \(\sigma(v_1, v_2)\) centered at \(v_1\), allowing for correct decoding. We refer to \(\sigma(v_1, v_2)\) as the \myemph{equal-detection threshold} (see Definition \ref{def: equal-detection threshold}).  Furthermore, a metric is \myemph{equal-detection normal} if the distance between vectors satisfies
\[
d(v_1, v_2) = 2\sigma(v_1, v_2)
\]
whenever $\sigma(v_1, v_2)$ is defined. \\

As a bridge between the first approach, where a metric measures shortest paths of elementary errors, leading to the notion of convexity, and the second approach, where a metric measures the correction and detection of errors, leading to correction- and equal-detection normality, the following proposition states that these two approaches in fact coincide (see Appendix \ref{appendix: correction, detection and convexity}):

 \begin{restatable}[]{proposition}{convexcorrectdetect}
\label{prop: convexity, correction, detect}%
   An integral metric is convex if and only if it is correction-normal and equal-detection normal.
 \end{restatable}
\noindent

While integrality and convexity follow naturally from a definition of a metric arising from elementary errors, the following two properties  only hold for a subset of all possible integral convex metrics, but are nevertheless natural in many applications:  \begin{itemize}
    \item Equality of the elementary error sets for all vector \(v_1,v_2 \in V\), i.e. \(\cE(v_1) = \cE(v_2)\). Then the corresponding metric is \emph{translation invariant}, \textit{(M4)};
    \item  The elementary error sets being closed under change of symbols, i.e., if \(e \in \cE(v_1)\), then \(\alpha e \in \cE(v_1)\) for any \(\alpha \in \bF\withoutzero\). The corresponding metric is \emph{scale invariant}, \textit{(M5)}.

\end{itemize}
The following proposition, proved in Appendix \ref{appendix: correction, detection and convexity}, comes directly from the equivalent definition \ref{def: convex} of  convexity and justifies the study of projective metrics:
 \begin{restatable}[]{proposition}{convexscaletranslation}\label{prop: projective iff convex scale trans}
    A metric is projective if and only if it is integral, convex, scale- and translation-invariant. 
\end{restatable}

Furthermore, in Section~\ref{section: embedding in projective metrics}, we show as a corollary of a more general result, that even metrics with weaker assumptions, namely those that are not convex, can be seen as submetrics of a projective metric on a larger ambient space.
\begin{restatable}[]{corollary}{scaleinvariantembedding}
    Any integral scale-translation-invariant metric space can be linearly embedded in a larger projective metric space.
\end{restatable}

We interpret this as follows: among all possible integral scale-translation-invariant metrics, the convex ones, i.e. projective metrics, can be used for good ambient spaces. The remaining metrics are then subspaces or codes in these ambient spaces. 

\subsection{Projective weight functions and projective metrics}
\label{subsec: projective weight basic proprieties}
In this subsection, we present the main definitions regarding projective metrics. For the remainder of this section, let \(V\) be a finite dimensional vector space over a finite field $\Ff_q$.

\begin{definition}\label{def: proj weight}
    Given a subset of pairwise linearly independent vectors \(
\mathcal{F} \subset V
\), the \myfont{projective weight} \(\wt_{\mathcal{F}}(x)\) of \(x \in V\) with respect to \(\mathcal{F}\) is defined as the minimum cardinality of a subset \(I \subset \mathcal{F}\) that satisfies $x \in \spn{I}$, where $\spn{\cdot}$ denotes the linear span in $V$.
Otherwise, if \(x \notin \spn{\mathcal{F}}\), we set \(\wt_{\mathcal{F}}(x) := \infty\) as convention. In short,
\begin{equation*}
\wt_\F(x) := \min\left(\{\left|I\right| \ : \ I \subset \F,\, x \in \spn{I}\} \cup \{\infty\}\right).
\end{equation*}
The function \(\wt_{\mathcal{F}}: V \to \mathbb{N} \cup \{\infty\}\) sending $x$ to \(\wt_{\mathcal{F}}(x)\), is called the \myfont{projective weight function}.
\end{definition}

\begin{remark}
Note that we only require \textit{pairwise} linear independency, i.e. $0 \notin \F$ and if $x,y \in\F$ are distinct, then $x \neq \lambda y$ for any $\lambda \in \Ff_q$. Yet any three or more vectors in $\F$ may be linear dependent.  
\end{remark}

\medskip

The term \onequote projective weight' arises from the fact that replacing any element \(f \in \mathcal{F}\) with \(\lambda f\), where \(\lambda \in \mathbb{F}_q\setminus\{0\}\), results in the same weight function. Consequently, a projective weight function can equivalently be defined by a subset of points $\F$ in the projective space
\[\mathbb{P}(V) = (V \setminus \{0\}) / \sim \quad \text{ with } \quad x \sim y \, \Leftrightarrow \, x = \lambda y \quad \text{for some } \lambda \in \Ff_q \setminus \{0\} 
\]
for \(x, y \in V \setminus \{0\}\). This projective space is equivalently described by the set \(\Gr_1(V)\) of 1-dimensional subspaces $\spn{x} \subset V$, since $ x \sim y  \Leftrightarrow  \spn{x} = \spn{y}$. For this reason we will sometimes refer to \(\cF\) as a subset of \(\Gr_1(V)\).\\

The set $\F$ is called the \myfont{projective point family} of $\wt_\F$.  In general we will assume that $\spn{\F} = V$, in which case we call $\F$ a \myfont{spanning family}. The reason for this assumption is that the weight function $\wt_\F$ is finitely valued whenever $\F$ is a spanning family. Moreover, the family $\F$ is uniquely determined by the projective weight function $\wt_\F$:

\medskip
\begin{observation}
Let $\cF,\G \subset \Gr_1(V)$ be such that $\wt_\cF(\cdot) = \wt_\G(\cdot)$, then $\cF= \G$. 
\end{observation}
\medskip
\noindent
This follows directly from the fact that for every $x \in V\setminus\{0\}$ we have \begin{center}
    $\spn{x} \in \cF\;\Leftrightarrow\; \wt_\cF(x) = 1 \;\Leftrightarrow\; \wt_\G(x)=1 \;\Leftrightarrow\; \spn{x} \in \G$.
\end{center}

Gabidulin and Simonis in \cite{gabidulin1998metrics} first introduced the concept of projective weight functions as a special case of weight functions generated by families of subspaces (referred to as \(\cF\)-weights).

\begin{definition}\label{def: F-weight}
Let $\cF= \{F_1, \ldots, F_\N \} \subset \cP(V)$ be a family of subsets of $V$. Then, if  \( x \in \sum_{F \in \cF}\spn{F}\), the \myfont{$\bm{\cF}$-weight} $\wt_\cF(x)$ of a vector $x \in V$ is defined as the minimum cardinality of a set $I \subseteq [\N]$ that satisfies
\[
x \in \sum_{i \in I}\spn{F_i}.
\]
Otherwise, if \( x \notin \sum_{F \in \cF}\spn{F}\), we define $\wt_\cF(x)$ as \(\infty\). The function $\wt_\cF(\cdot) : V \to  \Nn \cup \{\infty\}$ sending $x$ to $\wt_\cF(x)$ is called the $\bm{\cF}$\textbf{-weight function} or the weight function generated by $\cF$. 

\end{definition}

A projective weight associated to a family \(\cF \subset \Gr_1(V)\) is also trivially an \(\cF\)-weight. The converse is also true; in Appendix \ref{appendix: F-metrics} we show that projective weights and \(\cF\)-weights are,  in fact, equivalent. This is summarized in the following proposition.
 
\begin{restatable}[]{proposition}{propProjectivePointForm}\label{prop: projective point form} Let $\cF\subseteq \mathcal{P}(V)$ be a family of subsets of $V$. Then the \myfont{projective point form} of $\cF$, defined as
\[
\ppf(\cF) := \bigcup_{F \in \cF} \Gr_1(\spn{F})  \;\;\subseteq\; \Gr_1(V),
\]
spans the same subspace spanned by \(\cF\), that is: \(\sum_{G \in \ppf(\cF)}G= \sum_{F \in \cF}\spn{F}\) and generates the same weight function as $\cF$, i.e. 
\[
\wt_\cF(x) = \wt_{\ppf(\cF)}(x)
\]
for all $x \in V$.
\end{restatable}

\noindent

We observe that a projective weight $\wt_\F$ with spanning family $\F = \{f_1,f_2,\ldots,f_\N\}$ can also be seen as a quotient weight (see Definition \ref{def: quotient weight}) of the Hamming weight on the larger space \(\bF_q^{\N}\), given by the linear map \(\varphi: \bF_q^{\N} \to V \) sending the $i$-th basis vector \(e_i\) to \(f_i\). This aspect of projective weight functions will be further discussed in depth in section \ref{sec: Projective metrics as quotients}.


\text{}

The distance function $d_\F(x,y) := \wt_\F(y-x)$ corresponding to a projective weight $\wt_\F$ is called a \myfont{projective metric}. Since $\wt_\F$ satisfies all the axioms \textit{(W0)}-\textit{(W4)} of a scale-invariant weight function, any projective metric is a scale-translation-invariant metric. It turns out that for a scale-translation-invariant metric, integral convexity is necessary and sufficient for the metric to be projective.

\convexscaletranslation*

\begin{proof}
Consider a projective metric $d_\F(\cdot,\cdot)$ with spanning family $\F \subset V$, and let $x \in V$ with $\wt_\F(x) < \infty$.  Let $a$ be any integer between $0$ and $d := \wt_\F(x)$. Due to translation-invariance, it suffices to show the existence of a $y \in V$ with $\wt_\F(y) = a$ and $\wt_\F(x-y) = d-a$ to prove convexity.

By definition there is a minimum size subset \(I \subset \cF\) of size $d$ such that \( x = \sum_{f \in I}a_f f\) for some coefficients $a_f \in \Ff_q \setminus \{0\}$. Let \(J \subset I\) with \(|J| = a\). Let \( y = \sum_{f \in J} a_f f\) and \(y' = \sum_{f \in I \setminus J} a_f f\). Then \(\wt_{\cF}(y) \leq a\) and \(\wt_{\cF}(y') \leq |I \setminus J| = d-a\), and \(d = \wt_{\cF}(x) = \wt_{\cF}(y+y') \leq \wt_{\cF}(y) + \wt_{\cF}(y') \leq a + d - a = d  \). Thus \(\wt_{\cF}(y) = a\), \(\wt_{\cF}(y') = d-a\). From this we have \( \wt_{\cF}(y) = a\) and \( \wt_{\cF}(x-y) =  \wt_{\cF}(y') = d-a\).

For the converse, consider an integral convex, scale- and translation-invariant metric $d(\cdot,\cdot)$ with corresponding weight function \(\wt(\cdot)\). Let $v_1,v_2 \in V$ with $d(v_1,v_2) < \infty$. Then \(\wt(\cdot)\) is projective with spanning family equal to \(\cF \coloneqq \{v \in V \mid \wt(v)=1\}\), since by definition of integral convexity, there is a minimal length path  \(S = (v_1=x_0,\ldots, x_{d(x_1,x_2)}=v_2)\) with \(1=d(x_{i},x_{i+1}) = \wt(x_{i+1}-x_i)\) for all \(i = 0,\ldots,d(x_1,x_2)-1\). Thus \(x_{i+1}-x_i \in \cF\) and a minimal path \(S\) is equivalent to a minimal linear combination of elements in \(\cF\) adding to $v_2-v_1$.
\end{proof}

Recall that two weight functions $\wt_1(\cdot)$, $\wt_2(\cdot)$ on  $V$ are linearly isometric or isomorphic, denoted $\wt_1(\cdot) \cong \wt_2(\cdot)$, if there exists an invertible linear map $L \in \GL(V)$ such that $\wt_1(x) = \wt_2(L(x))$ for all $x \in V$. As convexity and scale/translation invariance are defined solely in terms of the weight function, addition, and scaling, all of which are preserved by linear isometries, it is not surprising that the property of `being a projective metric space' is well-defined with respect to isometries, a simple yet necessary requirement for this property to be natural. This is summarized in the following proposition, with a short proof provided in Appendix \ref{appendix: linear equivalence}:

\begin{restatable}[]{proposition}{propProjectiveMetricsEquivalence}
\label{prop projective metrics equivalence}
Let $\wt(\cdot)$ be a weight function and $\wt_\cF(\cdot)$ a projective weight function on $V$ with $\cF$ its projective point family. Then
\[
\wt(\cdot) \simeq \wt_\cF(\cdot)
\]
if and only if $\wt(\cdot)$ is also a projective weight function with projective point family
\[
L(\cF) := \left\{ L(\spn{f}) \;\,|\; \spn{f} \in \F\right\}
\]
for some invertible linear map $L \in \GL(V)$.
\end{restatable}

\begin{observation}
    Seeing $\cF_1,\cF_2 \subset \Gr_1(V)$ as sets of points in the projective space, then $w_{\cF_1} \cong w_{\cF_2}$ if and only if there exists a homography mapping $\cF_1$ to $\cF_2$.
\end{observation}
 
\medskip
 
Lastly, to any projective metric $d_\F$ on $V$ with family $\F$ we can naturally define the \myfont{associated Cayley graph} $\Gamma(V,S_\F)$ with vertex set $V$ and generating set $S_\F := \{\lambda f \mid \lambda \in \Ff_q\setminus\{0\}, \, f \in \F\}$. We can then view projective metrics in the context of distances on Cayley graphs:

\begin{proposition}\label{prop: proj iff cayley}
The projective metric $d_\F$ on a vector space $V$ is equal to the graph distance on the associated Cayley graph $\Gamma(V,S_\F)$.
\end{proposition}


\subsection{Examples}\label{sec: examples}

In this section, we provide several examples of projective metrics on the vector space \(V = \mathbb{F}_q^N\) and introduce various operations on the spanning family to observe how more complex metrics can be obtained from simpler ones. Most of these examples were already noted in \cite{gabidulin1997metrics}; we report them here for ease of reference. Let \( \mathcal{B} \coloneqq \{e_1, \ldots, e_N\}\) be a basis of \(V\). 
\par 
In many cases, the set \(\mathcal{F}\) defining a projective weight \(\wt_{\mathcal{F}}\) is derived from a given set of vectors \(F \subset V\), which may not necessarily be pairwise linearly independent. The set \(\mathcal{F}\) is then chosen as a maximal subset of \(F\) such that all its vectors are pairwise independent.

For simplicity, we adopt a slight abuse of notation by identifying \(\mathcal{F}\) with \(F\). However, it is important to note that the size of the weight obtained from \(F\) can be strictly smaller than \(|F|\), as non-independent vectors are excluded to ensure pairwise independence.

\begin{example}
    \myfont{Hamming Metric.} The Hamming metric \(\wtH\) can be shown to be a projective metric by simply taking the spanning family to be the basis \(\cB\) itself.
    Here, the weight of a vector \(x \in V\) is the number of nonzero coordinates respect to the basis \(\cB\).
\end{example}

\begin{example}
    \myfont{Discrete Metric.} The discrete metric is defined as:
    \[
    \wt_{\operatorname{Dis}}(x) \coloneqq 
    \begin{cases} 
        0 & \text{if } x = 0, \\
        1 & \text{otherwise}.
    \end{cases}
    \]
    This is a projective metric with the spanning family equal to \(\Gr_1(V)\), the set of all one-dimensional subspaces of \(V\). The discrete metric assigns weight \(1\) to any nonzero vector, satisfying the projective metric definition.
\end{example}

\begin{example}\label{example: combinatorial metrics}
    \myfont{Combinatorial Metrics (over finite fields).} (See \cite{pinheiro2019combinatorial}). A combinatorial weight  \(\wt_{\operatorname{Combinatorial}}\) is an \(\mathcal{F}\)-weight (see Definition \ref{def: F-weight}), where \(\mathcal{F}\) is a family of subspaces \(\{F_1, \ldots, F_k\}\)  with each subspace \(F_j\) of the form
    \[
    F_j = \langle I_j \rangle \leq V
    \]
    for some \(I_j \subset \mathcal{B}\), where \(\mathcal{B}\) is a basis of \(V\). Since by Proposition \ref{prop: projective point form} all \(\cF\)-metrics are projective metrics,  combinatorial metric are as well.
\end{example}

\begin{example}
\label{example: Phase Rotation Metric}
    \myfont{Phase Rotation Metric.} (See \cite{gabidulin1998hard}). The phase rotation metric is a projective metric with spanning family:
    \[
    \mathcal{F} \coloneqq \cB \cup \{w\},
    \]
    where \(w \coloneqq \sum_{v \in \cB}v\).
    In figure \ref{fig: phase cayley}, we provide a small visualization for the Cayley graph \(\Gamma(V,S_{\cF})\) with \(V =\bF_2^4\), showcasing a shortest path in the phase rotation metric compared to the Hamming metric.
    
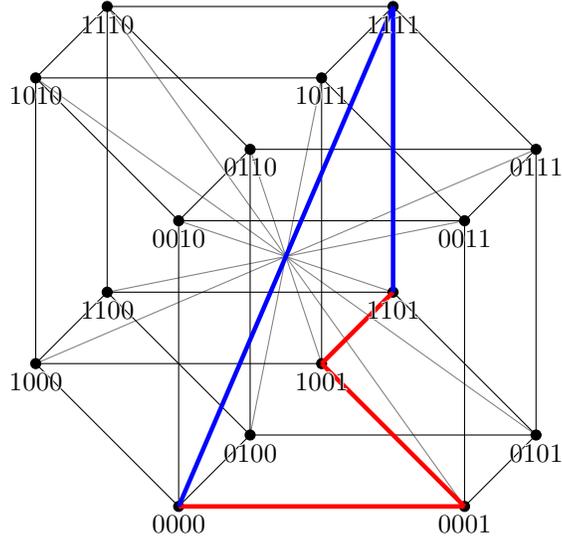
\begin{figure}[h!]
\centering
\begin{tikzpicture}[yscale=1.9, xscale=1.9, rotate=0]

\coordinate (0000) at (1,1);
\coordinate (0001) at (3,1);
\coordinate (0010) at (1,3);
\coordinate (0011) at (3,3);
\coordinate (0100) at (1.5,1.5);
\coordinate (0101) at (3.5,1.5);VVv
\coordinate (0110) at (1.5,3.5);
\coordinate (0111) at (3.5,3.5);

\coordinate (1000) at (0,2);
\coordinate (1001) at (2,2);
\coordinate (1010) at (0,4);
\coordinate (1011) at (2,4);
\coordinate (1100) at (0.5,2.5);
\coordinate (1101) at (2.5,2.5);
\coordinate (1110) at (0.5,4.5);
\coordinate (1111) at (2.5,4.5);

\draw[gray]
(0000) -- (1111)
(0010) -- (1101)
(0100) -- (1011)
(0110) -- (1001)
(1000) -- (0111)
(1010) -- (0101)
(1100) -- (0011)
(1110) -- (0001)
;

\draw
(0000) -- (0001)
(0010) -- (0011)
(0100) -- (0101)
(0110) -- (0111)
(1000) -- (1001)
(1010) -- (1011)
(1100) -- (1101)
(1110) -- (1111)
;

\draw
(0000) -- (0010)
(0001) -- (0011)
(0100) -- (0110)
(0101) -- (0111)
(1000) -- (1010)
(1001) -- (1011)
(1100) -- (1110)
(1101) -- (1111)
;

\draw
(0000) -- (0100)
(0001) -- (0101)
(0010) -- (0110)
(0011) -- (0111)
(1000) -- (1100)
(1001) -- (1101)
(1010) -- (1110)
(1011) -- (1111)
;

\draw
(0000) -- (1000)
(0001) -- (1001)
(0010) -- (1010)
(0011) -- (1011)
(0100) -- (1100)
(0101) -- (1101)
(0110) -- (1110)
(0111) -- (1111)
;

\filldraw[black] (0000)  circle (1pt);
\filldraw[black] (0001)  circle (1pt);
\filldraw[black] (0010)  circle (1pt);
\filldraw[black] (0011)  circle (1pt);
\filldraw[black] (0100)  circle (1pt);
\filldraw[black] (0101)  circle (1pt);
\filldraw[black] (0110)  circle (1pt);
\filldraw[black] (0111)  circle (1pt);

\filldraw[black] (1000)  circle (1pt);
\filldraw[black] (1001)  circle (1pt);
\filldraw[black] (1010)  circle (1pt);
\filldraw[black] (1011)  circle (1pt);
\filldraw[black] (1100)  circle (1pt);
\filldraw[black] (1101)  circle (1pt);
\filldraw[black] (1110)  circle (1pt);
\filldraw[black] (1111)  circle (1pt);

	\draw[line width=0.6mm, red]
	(0000) -- (0001)
	(0001) -- (1001)
	(1001) -- (1101)
	;
	\draw[line width=0.6mm, blue]
	(0000) -- (1111)
	(1111) -- (1101)
	;
\draw (0000) node[below] {\contour{white}{0000}};
\draw (0001) node[below] {\contour{white}{0001}};
\draw (0010) node[below] {\contour{white}{0010}};
\draw (0011) node[below] {\contour{white}{0011}};
\draw (0100) node[below] {\contour{white}{0100}};
\draw (0101) node[below] {\contour{white}{0101}};
\draw (0110) node[below] {\contour{white}{0110}};
\draw (0111) node[below] {\contour{white}{0111}};

\draw (1000) node[below] {\contour{white}{1000}};
\draw (1001) node[below] {\contour{white}{1001}};
\draw (1010) node[below] {\contour{white}{1010}};
\draw (1011) node[below] {\contour{white}{1011}};
\draw (1100) node[below] {\contour{white}{1100}};
\draw (1101) node[below] {\contour{white}{1101}};
\draw (1110) node[below] {\contour{white}{1110}};
\draw (1111) node[below] {\contour{white}{1111}};

\end{tikzpicture}
\caption{Cayley graph associated to the phase rotation metric on $\Ff_2^4$. Indicated are shortest paths from \((0,0,0,0)\) to \((1,1,0,1)\) in the phase rotation metric (blue) and in the Hamming metric (red).
}
\label{fig: phase cayley}
\end{figure}

\end{example}

\begin{example}
    \myfont{Rank Metric.} Consider \(V = \mathbb{F}_q^{n \times m}\), the vector space of \(n \times m\) matrices. The rank metric of a matrix \(A\) is defined as:
    \[
    \wt_{\operatorname{Rank}}(A) \coloneqq \operatorname{rank}(A).
    \]
    Any matrix of rank \(r\) can be written as the sum of \(r\) rank-1 matrices, thus the rank metric is a projective metric with spanning family
    \[
    \mathcal{F} \coloneqq \{A \in V \mid A \text{ has rank equal to } 1\}.
    \]
\end{example}

\begin{example}
    \myfont{Sum-Rank Metric.} Let \(V = \bigoplus_{i=1}^k \mathbb{F}_q^{n_i \times m_i}\). The sum-rank metric is defined as:
    \[
    \wt_{\operatorname{Sum-rank}}(A) \coloneqq \sum_{i=1}^k \wt_{\operatorname{Rank}}(A_i),
    \]
    where \(A = (A_1, \ldots, A_k)\) and each \(A_i \in \mathbb{F}_q^{n_i \times m_i}\). This metric is projective with spanning family \(\cF = \{A = (A_1, \ldots, A_k) \in V \mid A \text{ has sum-rank weight equal to } 1\}\).
\end{example}

\begin{example}
\myfont{Row Metric} and \myfont{Column Metric}. For \(V = \mathbb{F}_q^{n \times m}\), the row metric assigns weight based on the minimum number of rows needed to cover all nonzero entries of a matrix. Specifically, the row weight is defined as:
    \[
    \wt_{\operatorname{Row}}(A) = \min \{|R| \mid R \subseteq [n] \text{ such that if } A_{ij} \neq 0, \text{ then } i \in R \}.
    \]
   
    The row metric is a projective metric with spanning family consisting of matrices having all zero entries except for those in a single row. The column metric is defined as \(\wt_{\operatorname{Column}}(A) \coloneqq \wt_{\operatorname{Row}}(A^T)\).
\end{example}

\begin{example}
    \myfont{Cover Metric (or Term-Rank Metric \cite{gabidulin1971class}).} For \(V = \mathbb{F}_q^{n \times m}\), the cover metric assigns weight based on the minimum number of rows or columns needed to cover all non-zero entries of a matrix. Specifically, the cover weight is defined as:
    \[
    \wt_{\operatorname{Cover}}(A) = \min \{|R| + |C| \mid R \subseteq [n], C \subseteq [m], \text{ such that if } A_{ij} \neq 0, \text{ then } i \in R \text{ or } j \in C \}.
    \]
   
    The cover metric is a projective metric with the spanning family:
    \[
    \mathcal{F} = \{A \in V \mid \wt_{\operatorname{Cover}}(A) = 1\},
    \]
    which consists of matrices having all zero entries except for those in a single row or a single column. The cover metric is an example of a non-trivial combinatorial metric, where the sets $I_j$ in Example \ref{example: combinatorial metrics} are sets of coordinates in single rows or columns.
\end{example}

\newcommand{\outertimes}{\pmb{\otimes}}
\begin{example}\label{example: tensor rank} \myfont{Tensor Rank Metric}. (See \cite{ byrne2021tensor,roth1996tensor}). 
Let \( \mathcal{T} \) be a $d$-tensor in the space \( \mathbb{F}_q^{n_1 \times n_2 \times \cdots \times n_d} \). The \textit{tensor rank weight} of \( \mathcal{T} \), denoted as \(\wt_{\operatorname{Rank}}(\mathcal{T}) \), is defined as the minimum number $r$ of rank-1 tensors needed to express \( \mathcal{T} \) as their sum:
\[
\mathcal{T} = \sum_{k=1}^r u^{(1,k)} \outertimes u^{(2,k)} \outertimes \cdots \outertimes u^{(d,k)},
\]
where \( u^{(h,k)} \in \mathbb{F}_q^{n_k} \) for \( h = 1, 2, \ldots, d \), and \( \outertimes \) denotes the tensor (outer) product given by 
\[(u^{(1,k)} \outertimes u^{(2,k)} \outertimes \cdots \outertimes u^{(d,k)})_{i_1,\ldots,i_d} = \prod_{h=1}^du_{i_h}^{(h,k)}.\]
Clearly the tensor rank metric is a projective metric having as spanning family the set of rank-1 tensors. The tensor rank metric for $d=2$, i.e. for 2-tensors, coindices with the rank metric for matrices.
\end{example}

We now define the following operations  on projective metrics:
\begin{definition}[\textbf{Union of Weights}]
Let \(\cF\subset V\) and \(\cG \subset V\) define the projective weights \(\wt_{\cF}\) and \(\wt_{\cG}\). The union of their weights is defined as:
\[
\wt_{\cF} \cup \wt_{\cG} := \wt_{\cF \cup \cG}.
\]
which is a projective metric on \(V\)
\end{definition}

A particular case of the union of projective weights is given as follows. Let \(V_1\) and \(V_2\) be two vector spaces over \(\mathbb{F}_q\), with two projective metrics given by \(\mathcal{F}_1 \subset V_1\) and \(\mathcal{F}_2 \subset V_2\). Define \(V \coloneqq V_1 \oplus V_2\). By considering the natural inclusions \(i_1: V_1 \to V\) and \(i_2: V_2 \to V\), we can define the weights \(\wt_{i_1(\mathcal{F}_1)}\) and \(\wt_{i_2(\mathcal{F}_2)}\) on \(V\). 

We then define the \myfont{disjoint union} or \myfont{sum} of the projective weights \(\wt_{\mathcal{F}_1}\) and \(\wt_{\mathcal{F}_2}\) as:
\begin{equation}
    \label{def: disjoint union of projective weights}
\wt_{\mathcal{F}_1} \sqcup \wt_{\mathcal{F}_2} \coloneqq \wt_{i_1(\mathcal{F}_1)} \cup \wt_{i_2(\mathcal{F}_2)}.
\end{equation}
This construction provides a useful method to create simple projective weights in various dimensions by combining weights from different vector spaces.
An other useful operation on projective weights is given by the following:

\begin{definition}[\textbf{Tensor Product of Weights}]
Let \(\cF \subset \Ff_q^{N_1}\) and \(\cG \subset \Ff_q^{N_2}\) be two families defining projective weights $\wt_{\cF}$ and $\wt_{\cG}$. The tensor product of their weights is defined as:
\[
\wt_{\cF} \outertimes \wt_{\cG} := \wt_{\cF \outertimes \cG},
\]
where
\[
\cF \outertimes \cG := \{\langle f \outertimes  g \rangle \mid \langle f \rangle \in \cF, \langle g \rangle \in \cG\}.
\]
For the tensor product we regularly use the \emph{outer} tensor product $\outertimes_{\operatorname{out}}$, for two row vectors \(v \in \Ff_q^{N_1}\), \(w \in \Ff_q^{N_2}\) given by \(v \outertimes_{\operatorname{out}} w \coloneqq v^\top w \in \Ff_q^{N_1 \times N_2}\), or the \emph{Kronecker} tensor product $\outertimes_{\operatorname{Kron}}$, for two matrices \(A \in \bF_q^{n_1\times m_1}\), \(B \in \bF_q^{n_2 \times m_2}\) given by  
\[
A \outertimes_{\operatorname{Kron}} B :=
\begin{bmatrix}
A_{11} B & A_{12} B & \cdots & A_{1m_1} B \\\\
A_{21} B  & A_{22} B & \cdots & A_{2m_1} B \\\\
\vdots & \vdots & \ddots & \vdots \\\\
A_{n_11} B & A_{n_12} B & \cdots & A_{n_1m_1} B
\end{bmatrix} \in \bF_q^{n_1n_2 \times m_1m_2}.
\]

\end{definition}

We observe that the more ``complex" metrics introduced previously can be constructed from simpler ones using the following operations:
\begin{proposition}
The following relations hold:

\begin{enumerate}
    \item \textbf{Rank Weight:} 
    Consider the discrete weights \(\wt^n_{\operatorname{Dis}}\) and \(\wt^m_{\operatorname{Dis}}\) on \(\mathbb{F}_q^n\) and \(\mathbb{F}_q^m\), respectively, and the rank weight \(\wt_{\operatorname{Rank}}\) on \(\mathbb{F}_q^{n \times m}\). Then:
    \[
    \wt^n_{\operatorname{Dis}} \outertimes_{\operatorname{out}} \wt^m_{\operatorname{Dis}} = \wt_{\operatorname{Rank}}.
    \]

    \item \textbf{Sum-Rank Weight:}
    Let \(V_1 = \bF_q^l\) with Hamming weight \(\wtH\) and \(V_2 =  \mathbb{F}_q^{n \times m}\) with rank weight \(\wt_{\operatorname{Rank}}\). The Sum-Rank weight \(\wt_{\operatorname{Sum-rank}}\) on \(V = \bF_q^{n\times lm}\) satisfies:
    \[
    \wtH \outertimes_{\operatorname{Kron}} \wt_{\operatorname{Rank}} = \wt_{\operatorname{Sum-rank}}.
    \]

    \item \textbf{Row Weight:}
    Let \(\wtH\) be the Hamming weight on \(\mathbb{F}_q^n\) and \(\wt_{\operatorname{Dis}}\) the discrete weight on \(\mathbb{F}_q^m\). The Row weight on \(\bF_q^{n\times m}\) satisfies:
    \[
  \wtH \outertimes_{\operatorname{out}}   \wt_{\operatorname{Dis}} =  \wt_{\operatorname{Row}}.
    \]

    \item \textbf{Column Weight:}
    Let \(\wt_{\operatorname{Dis}}\) be the discrete weight on \(\mathbb{F}_q^n\) and \(\wtH\) the Hamming weight on \(\mathbb{F}_q^m\). The Column weight on \(\bF_q^{n\times m}\) satisfies:
    \[
     \wt_{\operatorname{Dis}} \outertimes_{\operatorname{out}}  \wtH = \wt_{\operatorname{Column}}.
    \]

    \item \textbf{Cover Weight:}
    The cover weight, combining row and column weights, satisfies:
    \[
    \wt_{\operatorname{Row}} \cup \wt_{\operatorname{Column}} = \wt_{\operatorname{Cover}}.
    \]
\end{enumerate}
\end{proposition}

\begin{proof}
We provide the proof for the relation \(\wt_{\text{Dis}} \outertimes \wt_{\text{Dis}} = \wt_{\text{Rank}}\). Let \(\mathcal{F}_n = \Gr_1(\mathbb{F}_q^n)\) and \(\mathcal{F}_m = \Gr_1(\mathbb{F}_q^m)\) be the spanning families inducing the discrete metrics on \(\mathbb{F}_q^n\) and \(\mathbb{F}_q^m\), respectively. Let \(\mathcal{F} = \{A \in \mathbb{F}_q^{n \times m} \mid \operatorname{rank}(A) = 1\}\) be the spanning family inducing the rank metric. 

For all \(\spn{v} \in \mathcal{F}_n\) and \(\spn{w} \in \mathcal{F}_m\), the tensor product \(v \outertimes w \in \mathbb{F}_q^{n \times m}\) is a rank-1 matrix. Thus, \(\mathcal{F}_n \outertimes \mathcal{F}_m \subseteq \mathcal{F}\). Conversely, any rank-1 matrix \(A\) can be written as \(v \outertimes w\) for some \(v \in \mathbb{F}_q^n\) and \(w \in \mathbb{F}_q^m\). Therefore, \(\mathcal{F}_n \outertimes \mathcal{F}_m = \mathcal{F}\), which implies \(\wt_{\operatorname{Dis}} \outertimes \wt_{\operatorname{Dis}} = \wt_{\operatorname{Rank}}\).

The proofs for the other relations follow similarly by demonstrating that the operations on the spanning families on the LHS correspond to the resulting spanning family for the metric on the RHS of the equation.
\end{proof}

\newpage
\section{Projective metrics as quotients}\label{sec: Projective metrics as quotients}
In this section we go in-depth into how projective metrics can be seen as quotient weights, and we will see that this induces many of their properties. The corresponding family of maps inducing projective weights were introduced in \cite{gabidulin1997metrics} as \textit{parent functions}. In the following, let $V$ be an $N$-dimensional vector space over $\Ff_q$ and $\F = \{\spn{f_1},\ldots,\spn{f_\N}\} \subset \Gr_1(V)$ a projective point family. Recall that we use the shorthand notation $[\N]$ for the integer set $\{1,\ldots,\N\}$.

\begin{definition}(Parent functions and Parent codes of $\cF$).\label{def: parent function and parent code}
A \myfont{parent function} of $\cF$ is a linear map $\varphi: \Ff_q^{\N} \to V$ such that 
\[
\left\{ \spn{\varphi( e_i )}  \mid i \in [\N] \right\} = \F
\]
where  $e_i$ is the $i$-th standard basis vector of $\Ff_q^\N$.
I.e., there is a permutation $\sigma$ of $[\N]$ and scalars $\lambda_i \in \Ff_q\setminus \{0\}$ such that for all $i \in [\N]$
\[ \varphi( e_i )  = \lambda_i f_{\sigma(i)}.
\]
A \myfont{parent code} of $\cF$ is a code \(\pc\subset \bF_q^{\N}\) such that $\pc= \ker(\varphi)$ where \(\varphi\) is a parent function of $\cF$.
\end{definition}

The main importance of parent functions is their connecting role between projective metrics and the Hamming metric. Namely, a projective weight function is equal to the quotient weight induced by a corresponding parent function, which follows by direct calculation.
Moreover, a parent code has minimum Hamming distance $\geq 3$, a direct result from the fact that the vectors in a projective point family are pairwise linear dependent.

\begin{proposition}\label{prop: proj iff quotient 1}
Consider the space $\Ff_q^\N$ endowed with the Hamming metric. If $\varphi: \Ff_q^\N \to V$ is a parent function of $\F$, then
\[
\wt_\F = \Qwt{\varphi}, 
\]
and the parent code $\pc = \ker(\varphi)$ has minimum Hamming distance
\[
d_{\operatorname{H}}(\pc) := \min_{\substack{x,y \in \scalebox{0.8}{\pc} \\ x \neq y}} \{d_{\operatorname{H}}(x,y)\} \geq 3. 
\]
\end{proposition}


As often happens with projective metrics, a notion for metrics corresponds to an `equivalent' notion for (parent) codes.  
For example, as we will see, knowing the sphere sizes or weight distribution of a projective metric is equivalent to knowing the coset weight distribution of a parent code; a  code is perfect for a projective metric if its inverse image under a parent function is a perfect code containing a parent code; decoding algorithms for a parent code are equivalent to an algorithm that calculates the projective weight, and so on.

\begin{center}

\renewcommand{\arraystretch}{2.0}
\begin{tabular}{|c|c|c|}
\hline
Notion for metrics & Corresponding notion for parent codes & Statement \\
\hline
\hline
Weight distribution & Coset leader weight distribution & Proposition \ref{prop: bijection parent codes and projective weights}\\
\hline
Weight calculation & Decoding to parent codeword & Proposition \ref{prop: decoding weight algorithms eq}\\
\hline
$\F$-perfect code & Hamming-perfect code containing parent code & Proposition \ref{prop: F-perfect codes} \\
\hline

\(\cF\)-isometry & Hamming isometry fixing parent code & Corollary \ref{cor: group isom stab} \\
\hline

\end{tabular}
\end{center}

\newpage
In Proposition \ref{prop: bijection parent codes and projective weights} we make explicit the bijection between subspaces of \(\bF_q^{\N}\) (corresponding to parent codes) and projective metrics.
In Corollary \ref{cor: reduction to parent functions} we see that any quotient map inducing the projective weight can be \onequote\onequote reduced" to a parent function. 
In Subsection \ref{subsec: Isometries}, Theorem \ref{thm: isometry isomorphism} provides a bijection between projective isometries and certain Hamming isometries. As a corollary (Corollary \ref{cor: group isom stab}) of this result we obtain the the isometry group from \(V\) onto itself is group isomorphic to the stabilizer of the parent code respect to the Hamming isometries in \(\bF_q^{\N}\).
Lastly, in Subsection \ref{subsec: decoding algorithm} we show that if a decoding algorithm is known for the parent code than we can use it to calculate the projective weight.

\subsection{Parent codes and parent functions}\label{subsec: parent codes and parent function}
Observe that a function  $\varphi: \Ff_q^{\N} \to V$ is a parent function of a projective weight if and only if \(\left\{ \varphi( e_i )  \mid i \in [\N] \right\}\) are pairwise linearly independent. Fixing a basis in \(V\), we get a correspondence between parent functions from \(\bF_q^{\N}\) to \(V\) and matrices of size \(\N \times N\) with pairwise linearly independent rows.
 Since any non-trivial projective point set has more than one parent function, and thus parent code, it is natural to ask what relationship is there between different parent codes associated to the same projective point set. 
 
 As briefly noted in \cite{gabidulin1998metrics} we see that they are Hamming equivalent: the set of parent codes of a projective metrics is equal to the Hamming equivalence class of any of the parent codes. We say that \( \C_1,\C_2 \subset V\) are \myfont{Hamming equivalent}, $\C_1 \equiv_{\operatorname{H}} \C_2$, if there exists $L:V \to  V$ such that $L(\C_1) = L(\C_2)$ and $L$ is a linear isometry respect to the Hamming metric. It is a known fact that if $L$ is a linear isometry respect to the Hamming metric then it can be written as the product of a permutation matrix $P$ and an invertible, diagonal matrix $\Lambda$. 

\begin{observation}
\label{obs: the parent codes of a spanning family F, are all Hamming equivalent}
   All parent codes of a projective weight \(\wt_{\cF}\) are Hamming equivalent.
\end{observation}
\begin{proof}
    Let $\varphi, \psi$ be two parent functions of $\cF= \{\langle f_1\rangle,\ldots,\langle f_{\N}\rangle\}$, and assume without loss of generality that $\varphi(e_i) = f_i$ for all $i \in [\N]$. Then there exists a permutation $\sigma \in S_{\N}$ (the symmetric group), and scalars $\lambda_i \in \Ff_q\setminus \{0\}$ for \( i \in [\N]\) such that \( \psi(e_i) = \lambda_{i} f_{\sigma(i)} =  \varphi(\lambda_{i} e_{\sigma(i)})\). Thus $\psi = \varphi \circ \Lambda \circ P$, where $P$ and $\Lambda$ are the linear applications associated respectively to the permutation of coordinates induced by $\sigma$ and the scaling of the coordinates by the $\lambda_i$'s. Thus, $\ker (\varphi) = \Lambda P  \ker (\psi)$ and so  \(\ker(\varphi) \equiv_{\operatorname{H}} \ker(\psi) \).
\end{proof}

Viceversa, another natural question to ask is: in what relationship are two projective metrics such that the corresponding parents codes coincide? It turns out that this happens precisely when two projective metrics are linearly isometric or isomorphic (see also Proposition \ref{prop projective metrics equivalence}).
We now show that there is a natural bijection between the following sets of equivalence classes:
\begin{itemize}
\item  The set $\overline{ Pr}_{\N}(V)$ of isomorphism classes of projective metrics $\overline{\wt}$ on $V$ induced by projective point families of cardinality $\N$;
    \item The set $\overline{ \Gr}_{\N-N}(\Ff_q^{\N})_{d_{\operatorname{H}}\geq 3}$ of Hamming-equivalence classes of $(\N - N)$-dimensional subspaces \(\overline\pc\)  of $\Ff_q^{\N}$ with  minimum Hamming distance \(d_{\operatorname{H}} \geq 3\).
\end{itemize}
%

This was already noted \cite{gabidulin1998metrics}; here we formalize the result.

\begin{restatable}[]{proposition}{parentcodeEqProject}
    \label{prop: bijection parent codes and projective weights}
    There exists a natural bijection between isomorphism classes of projective metrics and Hamming-equivalence classes of subspaces:
    \begin{align*}
    \Psi: \overline{Pr}_{\N}(V) \ &\to \ \overline{\Gr}_{\N-N}(\Ff_q^{\N})_{d_{\operatorname{H}}\geq 3} \\
    \overline{\wt}_{\cF} \ &\mapsto \ \overline{\pc}_{\cF},
\end{align*}
    where $\overline{\pc}_{\cF}$ is the Hamming equivalence class of parent codes of $\cF$. The inverse of this bijection is given by
    \begin{align*}
    \Psi^{-1}: \overline{\Gr}_{\N-N}(\Ff_q^{\N})_{d_{\operatorname{H}}\geq 3} \ &\to \ \overline{Pr}_{\N}(V) \qquad \text{} \\
    \overline{\pc} \qquad \ &\mapsto \ \overline{\wt}_{\operatorname{quot},\varphi_{\scalebox{0.65}{\pc}}},
\end{align*}
with $\varphi_{\scalebox{0.8}{\pc}} : \Ff_q^{\N} \to V$ a linear map with kernel $\pc$.
\end{restatable}

\begin{proof}
See Appendix \ref{appendix: Equivalence of Projective Metrics}.
\end{proof}


In conclusion, isomorphic projective metrics are obtained from quotient maps with equivalent kernels with minimum Hamming distance at least 3. These kernels, the parent codes, constitute the linear dependencies in the corresponding projective point family. Namely, if $\cF= \{\langle f_1\rangle,\ldots,\langle f_{\N}\rangle\}$ is a family and $\varphi: \Ff_q^\N \to V$ a parent function with $\varphi(e_i) = f_i$ and parent code $\pc$, then a codeword $(c_1,\ldots,c_\N) = \sum_{i=1}^\N c_i e_i \in \pc$ implies $\sum_{i=1}^\N c_i f_i = 0$.

\begin{example}[\textbf{Phase rotation metric}]
Let $N \geq 2$ and consider the projective point family $\F = \{e_1,\ldots,e_N,-\sum_{i=1}^N e_i\} \subset \Ff_q^N$. Then the linear map $ \varphi : \Ff_q^{N+1}  \to \Ff_q^N : 
    e_i  \mapsto e_i \text{ for } i \in [N] \text{ and }
    e_{N+1}  \mapsto -\sum_{i=1}^N e_i
    $
represented by the matrix
\[
\begin{pmatrix}
    1 & 0 & \cdots & 0 \\
    0 & 1 & \cdots & 0 \\
    \vdots & \vdots & \ddots & \vdots \\
    0 & 0 & \cdots & 1 \\
    -1 & -1 & \cdots & -1 \\
\end{pmatrix}
\]
is a parent function for $\F$  with parent code equal to the repetition code $\pc = \spn{(1 \, 1 \, \ldots \, 1)} \subset \Ff_q^{N+1}$.
\end{example}

    
Finally, we consider quotients of linear maps whose kernels have minimum Hamming distance less than 3. Interestingly, this does not yield any other types of metrics.
We show that for any map \(\xi\) coming from the Hamming metric inducing a quotient weight, there exists a parent function \(\varphi\) inducing the same quotient weight, and with $d_{\operatorname{H}}(\ker(\varphi))\geq 3$; that is, any quotient weight of the Hamming metric is a projective metric.

In the following, given two finite vector spaces $\Ff_q^{\N_1}$ and $\Ff_q^{\N_2}$ with standard bases $\{e_1,\ldots,e_{\N_1}\}$ and $\{e'_1,\ldots,e'_{\N_2}\}$ respectively, a \emph{weakly row monomial map} from $\Ff_q^{\N_1}$ to $\Ff_q^{\N_2}$ is a linear map $\psi$ such that for every $i  \in [\N_1]$ 
\[
\psi(e_i) = \lambda_i e'_{j_i}
\]
for some (possibly zero) $\lambda_i \in \Ff_q$ and $j_i \in [\N_2]$. Such a map can be represented by a \emph{weakly row monomial matrix}, i.e. a matrix $M \in \Ff_q^{\N_1 \times \N_2}$ whose rows each contain at most one non-zero element.

Our interest in these maps comes from the following factorization property of matrices:

\begin{proposition}\label{prop: matrix factorization row monomial}
Let $M \in \Ff_q^{\N_1 \times \Vdim}$ be any matrix. Then  
\[
M = R \cdot P,
\]
with $R \in \Ff_q^{\N_1 \times \N_2}$ a weakly row monomial matrix and $P \in \Ff_q^{\N_2 \times \Vdim}$ a matrix whose rows are all pairwise linear independent (and thus non-zero), for some $\N_2 \leq \N_1$.
   Equivalently, any map \(\xi:\bF_q^{\N_1} \to \bF_q^{\Vdim}\) is equal to the decomposition \(\xi=\varphi \circ \psi\) for some \(\psi\) weakly row monomial map from \(\bF_q^{\N_1} \to \bF_q^{\N_2}\) and a parent function \(\varphi\) from \(\bF_q^{\N_2}\to \bF_q^{\Vdim}\).
\end{proposition}

\begin{lemma}\label{lem: quot weight is Hamming}
Consider the finite vector spaces $\Ff_q^{\N_1}$ and $\Ff_q^{\N_2}$ with standard bases $\{e_1,\ldots,e_{\N_1}\}$ and $\{e'_1,\ldots,e'_{\N_2}\}$ respectively and $\N_2 \leq \N_1$, both spaces endowed with the Hamming weight $\wtH$.
If $\psi : \Ff_q^{\N_1} \twoheadrightarrow \Ff_q^{\N_2}$ is a surjective weakly row monomial map, then
\[
(\Ff_q^{\N_2},\wt_{\operatorname{quot},\psi}) = (\Ff_q^{\N_2}, \wtH).
\]

\end{lemma}
\begin{proof}
By surjectivity, for each basis vector $e'_j$ there is some $e_i$ and $\lambda \neq 0$ such that $\psi(e_i) = \lambda e'_j$.
 As permutation and scaling of the basis vectors preserves the Hamming metric, let us assume without loss of generality that for all $i \in [\N_2]$ $\psi(e_i) = e'_i$ and $\N_1 > \N_2$. Now, let $w = \sum_{i=1}^{\N_2} w_i e'_i \in \Ff_q^{\N_2}$ be any vector and  $v \in \psi^{-1}(w)$ with $\wtH(v)$ minimal. We can write $v = \Tilde{w} + r$ with $\Tilde{w} = \sum_{i=1}^{\N_2} w_i e_i  \in \psi^{-1}(w)$ for some $r = \sum_{i=1}^{\N_2} r_i e_i \in \ker(\psi)$. We claim that $\wtH(\sum_{i=1}^{\N_2} r_i e_i) \leq \wtH(\sum_{j=\N_2+1}^{\N_1} r_j e_j)$, which implies that 
 \begin{align*}
 \wtH(v) &= \wtH\left(\Tilde{w} + \sum_{i=1}^{\N_2} r_i e_i\right)+\wtH\left(\sum_{j=\N_2+1}^{\N_1} r_j e_j\right)\\ 
 &\geq \wtH\left(\Tilde{w}\right) - \wtH\left(\sum_{i=1}^{\N_2} r_i e_i\right) + \wtH\left(\sum_{j=\N_2+1}^{N_1} r_j e_j\right) \\
 &\geq \wtH\left(\Tilde{w}\right) = \wtH(w)
 \end{align*}
 with equalities for $r = 0$, and hence $\wt_{\operatorname{quot},\psi}(w) = \wtH(\Tilde{w}) = \wtH(w)$ as desired.

 To prove our claim above, note that $\psi(r) = \sum_{i=1}^{\N_2} r_i e'_i + \sum_{j=\N_2+1}^{\N_1} r_j \psi(e_j) = 0$. By linear independence of the $e'_i$ 's, for each $i \in [\N_2]$ with $r_i \neq 0$ there must be at least one $j \in \{\N_2+1,\ldots,\N_1\}$ such that $r_j \neq 0$ and $\psi(e_j) = e'_i$. So $\wtH(\sum_{i=1}^{\N_2} r_i e_i) = \left|\{i \in [\N_2] \mid r_i \neq 0 \}\right| \leq \left|\{j \in \{\N_2+1,\ldots,\N_1\} \mid r_j \neq 0 \}\right| = \wtH(\sum_{j=\N_2+1}^{\N_1} r_j e_j)$ which concludes the proof.
\end{proof}

\begin{corollary}\label{cor: reduction to parent functions}

Let $V$ be a finite vector space over $\Ff_q$ and  $\xi: \Ff_q^{\N_1} \twoheadrightarrow V$  any surjective linear map. Then there exists an integer \(\N_2 \leq \N_1\) and a surjective linear map $\varphi: \bF_q^{\N_2} \twoheadrightarrow V$ such that, when $\Ff_q^{\N_1}$ and $\bF_q^{\N_2}$ are endowed with the Hamming weight $\wtH$:
\begin{itemize}
    \item $ \wt_{\operatorname{quot},\xi} =  \wt_{\operatorname{quot},\varphi}$;
    \item $\varphi$ is a parent function, i.e. it can be represented by a matrix with pairwise linear independent rows and $d_{\operatorname{H}}(\ker(\varphi))\geq 3$. 
\end{itemize}

\end{corollary}

\begin{proof}
Let $M$ be the matrix representing $\xi$. By applying Proposition \ref{prop: matrix factorization row monomial} to $M$, we can decompose $\xi$ as $\varphi \circ \psi$, with $\psi: \Ff_q^{\N_1} \twoheadrightarrow \bF_q^{\N_2}$ a surjective weakly row monomial map and  $\varphi: \bF_q^{\N_2} \twoheadrightarrow V$  a surjective linear map represented by a matrix with pairwise linear independent rows, that is, \(\varphi\) is a parent function with $d_{\operatorname{H}}(\ker(\varphi))\geq 3$, and 
\[
\begin{tikzcd}
{\Ff_q^{\N_1}} \arrow[r, "\psi", two heads] \arrow[rd, "\xi"',  two heads] & {\bF_q^{\N_2}} \arrow[d, "\varphi",  two heads] \\
    & {\text{}\ V \ .}      
\end{tikzcd}
\]
Next, by Lemma \ref{lem: quot weight is Hamming} we know that $(\bF_q^{\N_2},\wt_{\operatorname{quot},\psi}) = (\bF_q^{\N_2}, \wtH)$, and with Corollary \ref{cor: composition of quotients} we conclude that $(V, \wt_{\operatorname{quot},\xi}) = (V, \wt_{\operatorname{quot},\varphi})$. 
\end{proof}

\subsection{Projective isometries}\label{subsec: Isometries}

In this subsection we consider linear isometries (see subsection \ref{subsec: weighted vector spaces and contractions}). We refer to a linear isometry between two projective metrics, as a \myemph{projective isometry}. We give a strong relation between projective isometries and corresponding sets of Hamming Isometries. In particular we show that the set of projective isometries from \(V\) to itself is group isomorphic to the stabilizer subgroup of the parent code respect to the Hamming isometries.  Let \(\varphi_V: \mathbb{F}_q^{\N} \to V\) and \(\varphi_W: \mathbb{F}_q^{\N} \to W\) be two parent maps defining projective weight functions on vector spaces \(V, W\), and let \(\pc_V := \ker(\varphi_V)\) and \(\pc_W := \ker(\varphi_W)\) denote their respective parent codes and \(\cF_V\), \(\cF_W\) their spanning families  with \(\cF_V= \{f_i\}_{i=1}^{\N}\) and \(\cF_W=\{f'_i\}_{i=1}^{\N}\). Consider the subsets \(\cC,\cC'\subset \bF_q^{\N}\) and \(\cD_V \subset V\), \(\cD_W \subset W\).

We define the following sets of isometries:
\begin{enumerate}

    \item \(\isom_{\mathcal{F}}(V, W) := \left\{ \isomfun \in \GL(V, W) \mid \isomfun \text{ is an isometry between } (V, \wt_{\cF_V}) \text{ and } (W, \wt_{\cF_W})\right\}\) is the set of projective isometries between projective metric spaces. If \(\cF_V\) and \(\cF_W\) are both basis of \(V\) and \(W\) respectively that we denote the set of Hamming isometries \(\isom_{\mathcal{F}}(V, W)\) by \(\isom_{\operatorname{H}}(V, W)\).

    \item  We denote the set of projective isometries sending \(\cD_V\) onto \(\cD_W\) as 
    \[\Fixisom{\cD_V}{\cD_W}{V}{W} := \left\{ \isomfun \in \isom_{\cF}(V, W) \mid \isomfun(\mathcal{D}_V) = \mathcal{D}_W \right\}.\] 
    If \(\cF_V,\cF_W\) are both bases we denote this set by \(\Hixisom{V}{W}{\mathcal{D}_V}{\mathcal{D}_W}\)
    
    \item  The group of projective isometries from \(V\) to itself are referred to projective automorphisms and denoted by in \(\aut_{\cF}(V)\coloneqq \isom_{\cF}(V,V)\) and \(\Fisom{\cC}{\cC'}{V} := \left\{ F \in \aut_{\cF}(V) \mid F(\cC) = \cC' \right\}\) is the set of projective automorphisms mapping \(\cC\) to \(\cC'\). If \(\cC = \cC'\) this is the \emph{stabilizer} subgroup of \(\cC\) and is denoted by \(\stab_{\cF}(\cC)\). (Again, if the \(\cF\) induces the Hamming metric, these sets are denoted by \(\aut_{H}(V), \Hisom{\cC}{\cC'}{V}\) and \(\stab_{\operatorname{H}}(\cC)\) respectively),

\end{enumerate}

The following gives a natural way to \onequote lift' a projective isometry to an Hamming isometry:
\begin{observation}\label{obs: lifting of projective isometry}
    The linear map \(\isomfun:V\to W\) is an isometry between \((V,\wt_{\cF_V})\to(W,\wt_{\cF_W}) \)
    if and only if there exists $\sigma \in S_{\N}$ such that $ \isomfun(f_i)= \lambda_i f'_{\sigma(i)}$ for all $i = 1,\ldots, \N$. Furthermore, the linear map defined by \(\lift{L}: \bF_q^{\N}\to \bF_q^{\N}\) such that \(\lift{L}(e_i)=\lambda_i e_{\sigma_i}\) is the unique Hamming isometry which makes the following diagram commute.
\[\begin{tikzcd}[ampersand replacement=\&]
	V \& W \\
	{\Ff_q^{\N}} \& {\Ff_q^{\N}}
	\arrow["\lift{\isomfun}", from=2-1, to=2-2]
	\arrow["\varphi_V", from=2-2, to=1-2]
	\arrow["\varphi_W"', from=2-1, to=1-1]
	\arrow["{\isomfun}"', from=1-1, to=1-2]
\end{tikzcd}\]
\end{observation}

The fact that the projective isometry group is group isomorphic to the Hamming stabilizer of the parent code comes from the following theorem.

\begin{theorem}\label{thm: isometry isomorphism}
Let \(\Phi: \Hisom{\pc_V}{\pc_W}{\bF_q^{\N}} \to \isom_{\cF}(V,W)\), such that  \(\Phi(\lift{\isomfun}) := \isomfun\), where \(\isomfun\) is the only 
    projective isometry such that the following diagram commutes:
\begin{equation}\label{eq: diagram comm}
\begin{tikzcd}[ampersand replacement=\&]
	V \& W \\
	{\Ff_q^{\N}} \& {\Ff_q^{\N}}
	\arrow["\lift{\isomfun}", from=2-1, to=2-2]
	\arrow["\varphi_V", from=2-2, to=1-2]
	\arrow["\varphi_W"', from=2-1, to=1-1]
	\arrow["{\isomfun}"', from=1-1, to=1-2]
\end{tikzcd}
\end{equation}
Then \(\Phi\) is a bijection and behaves well under composition, that is if \(\lift{\isomFun} \in  \Hisom{\pc_V}{\pc_{W'}}{\bF_q^{\N}}  \) and \(\lift{\isomfun} \in  \Hisom{\pc_W}{\pc_V}{\bF_q^{\N}}  \), then \[\Phi(\lift{\isomFun} \circ \lift{\isomfun} ) = \Phi(\lift{\isomFun}) \circ \Phi(\lift{\isomfun}). \]
Furthermore for any two linear codes \(\cD_V \subset V\) and \(\cD_W \subset W\), the following restriction of \(\Phi\) is also a bijection:  
\[\Phi': \Hisom{\pc_V}{\pc_W}{\bF_q^{\N}} \cap \Hisom{\varphi^{-1}(\cD_V)}{\varphi^{-1}(\cD_W)}{\bF_q^{\N}} \to \Fixisom{\cD_V}{\cD_W}{V}{W},\] where \(\Phi' \coloneqq \Phi|_{ \Hisom{\pc_V}{\pc_W}{\bF_q^{\N}} \cap \Hisom{\varphi^{-1}(\cD_V)}{\varphi^{-1}(\cD_W)}{\bF_q^{\N}}}\). 

\end{theorem}

\begin{proof}
Since \(\lift{\isomfun} \in  \Hisom{\pc_V}{\pc_W}{\bF_q^{\N}}\), we have \(\ker(\varphi_W \circ \isomfun) = \pc_V = \ker(\varphi_V)\). By the Isometry Property \ref{prop: Isometry Property}, there exists a unique isometry 
\(\isomfun:(V, \Qwt{\varphi}) \to (W,  \Qwt{\varphi_W})\), such that the diagram above commutes.

We then define \(\Phi(\lift{\isomfun}) \coloneqq \isomfun\). To show that \(\Phi\) behaves well under composition: for any \(\lift{\isomfun}, \lift{\isomFun}\), the following composition diagram shows that 
\(\Phi(\lift{\isomfun} \circ \lift{\isomFun}) = \Phi(\lift{\isomfun}) \circ \Phi(\lift{\isomFun})\), by uniqueness of the projective isometry that makes the diagram commute.

\[
\begin{tikzcd}[ampersand replacement=\&]
	W \& V \& W' \\
	{\Ff_q^{\N}} \& {\Ff_q^{\N}} \& {\Ff_q^{\N}}
	\arrow["{\Phi(\lift{\isomFun})}", from=1-1, to=1-2]
	\arrow["{\Phi(\lift{\isomfun})}", from=1-2, to=1-3]
	\arrow["\varphi_V", from=2-2, to=1-2]
	\arrow["\varphi_W", from=2-1, to=1-1]
	\arrow["\lift{\isomFun}", from=2-1, to=2-2]
	\arrow["\lift{\isomfun}", from=2-2, to=2-3]
	\arrow["\varphi_{W'}", from=2-3, to=1-3]
	\arrow["{\Phi(\lift{\isomfun}) \circ \Phi(\lift{\isomFun})}", bend left=40, from=1-1, to=1-3]
\end{tikzcd}
\]

To show injectivity, since \(\Phi\) behaves well under composition it is sufficient to show that if \(\Phi(\lift{\isomfun}) = Id\) then \(\lift{\isomfun}=Id\). If \(\lift{\isomfun}(e_i)=e_j\), then \(\varphi_V(e_i-e_j) = f_i - \varphi_V(\lift{\isomfun}(e_i)) = f_i - \isomfun(\varphi_V(e_i)) = 0\). We conclude that \(e_i - e_{j}\in \pc_V = \ker(\varphi_V)\), so \(j = i\) since the elements in \(\cF_V\) are pairwise linearly independent, and thus \(\lift{\isomfun} = Id\).

To show surjectivity: for \(\isomfun \in \isom_{\cF}(V,W)\), define \(\lift{\isomfun}\) as in Observation \ref{obs: lifting of projective isometry}. Then, since the diagram \eqref{eq: diagram comm} commutes, \(\lift{\isomfun} \in \Hisom{\pc_V}{\pc_W}{\bF_q^N}\), and \(\Phi(\lift{\isomfun}) = \isomfun\). Hence, \(\Phi\) is a bijection.

Given \(\cD_V,\cD_W\) subcodes of \(V\) and \(W\) respectively, we now consider the restriction \(\Phi'\) of \(\Phi\) on \( \Hisom{\pc_V}{\pc_W}{\bF_q^{\N}} \cap \Hisom{\varphi_V^{-1}(\cD_V)}{\varphi_W^{-1}(\cD_W)}{\bF_q^{\N}}\). We check that \(\isomfun(\cD_V) = \cD_W\): since \(\lift{\isomfun} \in \Hisom{\pc_V}{\pc_W}{\bF_q^{\N}}\), we have:
\[
\isomfun(\cD_V) = \varphi_W(\lift{\isomfun}(\varphi_V^{-1}(\cD_V))) = \varphi_W(\varphi_W^{-1}(\cD_W)) = \cD_W.
\]

\(\Phi'\) is injective because \(\Phi\) is. To show surjectivity, we need to show that \(\lift{\isomfun}\) as defined in Observation \ref{obs: lifting of projective isometry} is in \(  \Hisom{\varphi^{-1}(\cD_V)}{\varphi_W^{-1}(\cD_W)}{\bF_q^n}\) and \(\lift{\isomfun} \in \Hisom{\pc_V}{\pc_W}{\bF_q^{\N}}\), but this is true since diagram \eqref{eq: diagram comm} commutes and \(\isomfun\) is in \(\Fixisom{\cD_V}{\cD_W}{V}{W}\).

\end{proof}

Finally, if \(V=W\) and \(\cD = \cD_V=\cD_W\), then since \(\Hisom{\varphi^{-1}(\cD)}{\varphi^{-1}(\cD)}{\bF_q^{\N}} = \stab_{\operatorname{H}}(\varphi^{-1}(\cD))\) and \(\Fisom{\cD}{\cD}{V}= \stab_{\cF}(D)\).
We have:
\begin{corollary}\label{cor: group isom stab}
   \( \stab_{\operatorname{H}}(\varphi^{-1}(\cD))\cong  \stab_{\cF}(D) \) as groups. In particular for \(\cD=(0)\) we have
   \[
   \stab_{\operatorname{H}}(\pc)\cong \stab_{\cF}((0)) = \isom_{\cF}(V).
   \]
\end{corollary}

\subsection{Decoding of the parent code and projective weight algorithms.  }\label{subsec: decoding algorithm}
For general families $\F$, it is computationally hard (NP-hard) to calculate the $\F$-weight of a vector. For example, in the context of tensor codes (see Example \ref{example: tensor rank} and \cite{byrne2021tensor, roth1996tensor}), calculating the tensor rank of a $d$-tensor for $d \geq 3$ is known to be NP-hard. However in this subsection we show that if a fast decoding algorithm for the parent code is known, it can be used to easily calculate the  weight in the corresponding projective metric.

\begin{definition}(Decoding algorithm)
    Let \(\pc \subset \bF_q^{\N}\) be a linear code. A minimum distance decoder (for $\pc$ in the Hamming metric) is an algorithm $f$ that for any input $y \in \bF_q^{\N}$ outputs a vector $f(y) \in \pc$ satisfying \(d_{\operatorname{H}}(f(y),y)= \min_{c' \in \pc}d_{\operatorname{H}}(c',y)\).
\end{definition}

\begin{definition}\label{def: coset leader weight}
Let $C$ be a subspace of a weighted vector space $(X,\wt_X)$ and $y \in X$. The \myfont{coset leader weight} of the coset $y + C$ is given by
\[
\wt_X(y+C) := \min\left\{\wt_X(x) \mid x \in y + C\right\}.
\]
An element $x \in y +C$ attaining this minimum is called a \myfont{coset leader}.
\end{definition}


\begin{observation}\label{obs: proj weight cose weight}
 Let $\cF \subset \Gr_1(V)$ be a spanning family with corresponding parent function $\varphi: \Ff_q^\N \to V$ and parent code $\pc \in \Ff_q^{\N}$, with $\Ff_q^{\N}$ endowed with the Hamming weight $\wtH$.
 Then an element $x \in y+\pc$ is a coset leader of $y+\pc \in  {\Ff^{\N}_q}/{\pc} $ if and only if $\wtH(x) = \wt_{\cF}(\varphi(x))$. This follows directly from the interpretation of projective metrics as quotient weights:  $\wt_\F(\varphi(x)) = \wt_\F(\varphi(y)) = \Qwt{\varphi}(\varphi(y)) = \min\{\wtH(x') \mid x' \in y+\pc\} = \wtH(y+\pc)$.
\end{observation}
\medskip

This shows a direct relation between coset leader weights and projective weights. Together with minimum distance decoding, the problems of calculating these weights are equivalent, as descibed and often used in the works of Gabidulin and coauthors \cite{gabidulin1997metrics,gabidulin1998metrics,gabidulin2003codes}


\begin{proposition}\label{prop: decoding weight algorithms eq}
 Let $\cF = \{f_1,f_2,\ldots\} \subset \Gr_1(V)$ be a spanning family with corresponding parent function $\varphi: \Ff_q^\N \to V$ and parent code $\pc \in \Ff_q^{\N}$. Then the following problems are polynomial-time (in $\N$) reducible to each-other:
 \begin{enumerate}
     \item \textbf{Minimum distance decoding}: given $y \in \Ff_q^\N $, find $c \in \pc$  satisfying \(d_{\operatorname{H}}(c,y) = \min_{c' \in \pc}d_{\operatorname{H}}(c',y)\);
     \item \textbf{Coset leader finding}: in a given coset $y + \pc \subset \Ff_q^\N$, find a coset leader $x$ with respect to the Hamming weight.
     \item \textbf{Projective weight calculation and representation}: given $v \in V$, find a linear combination $v = \alpha_1 f_{i_1} + \alpha_2 f_{i_2} + \ldots $ of length $\wt_\F(v)$ with $\alpha_i \in \Ff_q$, $f_i \in \F$.
 \end{enumerate}
\end{proposition}

 \begin{proof}
 
 $1 \Longleftrightarrow 2$. This equivalence is immediate under the correspondence $x = y-c$.

$2 \Longrightarrow 3$. Given a coset $y + \pc$, let $v := \varphi(y) = \varphi(y + \pc)$. If $v = \alpha_1 f_{i_1} + \alpha_2 f_{i_2} + \ldots $ is a minimum linear combination of length $\wt_\F(v)$, then $ x = \sum_{j = 1}^{\wt_\F(v)} \alpha_j e_{i_j}$ is a coset leader in $\varphi^{-1}(v) = y + \pc$, where $\{e_1,\ldots,e_\N\}$ is the standard basis of $\Ff_q^\N$.

$3 \, \Longrightarrow 2$. Conversely, given $v \in V$, let $y + \pc := \varphi^{-1}(v)$. If $x = \sum_{j = 1}^{\N} \alpha_j e_{i}$ is a coset leader, then $v = \varphi(x) = \sum_{j = 1}^{\N} \alpha_j f_{i}$ is a minimum linear combination with $\wt_\F(v)$ non-zero terms.

For both directions we used the above Observation \ref{obs: proj weight cose weight} that the coset leader weight and the projective weight are equal.
 \end{proof}

 This proof also shows how to directly obtain a projective weight calculation algorithm from a minimum distance decoder $f$: given $v \in V$, let $y \in \varphi^{-1}(v)$. Then $\wt_\F(v) = \wtH(y - f(y))$.



\newpage

\section{Embeddings of scale-translation-invariant metrics into projective metrics}\label{section: embedding in projective metrics}

Throughout this section, let $(V, \wt_V)$ be an $n$-dimensional vector space over $\Ff_q$ with \textit{finitely-valued} scale-invariant weight $\wt_V$. We will establish a bijection between the sets
\begin{equation}\label{eq: bijection subspaces and projective metrics}
\left\{\substack{\text{Hamming metric subspaces}\\ \text{with $(V,\wt_V)$ as quotient}}\right\} \,\overset{\cong}{\longleftrightarrow} \,
\left\{\substack{\text{Projective metric spaces}\\ \text{containing $(V,\wt_V)$ as subspace}}\right\} 
\end{equation}
(up to isomorphism/equivalence), and show that these sets are non-empty. As a consequence, any finitely-valued scale-invariant weighted space $(V,\wt_V)$ can be embedded in a projective metric space.\\
\\
A \emph{Hamming metric subspace $C$ with $(V,\wt_V)$ as quotient} consists of the following data:\\
\\
\begin{minipage}{0.6\textwidth}
\begin{enumerate}
    \item Integers $r \geq s \geq n$;
    \item The space $\Ff_q^r$ endowed with the Hamming metric;
    \item A subspace
    $C \subseteq \Ff_q^r$ of dimension $s$, endowed with the Hamming metric on $\Ff_q^r$ restricted to $C$;
    \item The inclusion map $\rho: C \hookrightarrow \Ff_q^r$;
    \item A surjective linear map $\varphi : C \twoheadrightarrow V$ such that\\ $\wt_V = \wt_{\operatorname{quot},\varphi}$.
\end{enumerate}
\end{minipage}
\begin{minipage}{0.3\textwidth}
\text{}\\
\\
\begin{tikzcd}
& (\Ff_q^r, \wtH)    \\
\text{}\qquad(C,  \wtH|_{C}) \arrow[ru, "\rho", hook] \arrow[rd, "\varphi"', two heads] &                 \\
    & {\scalebox{1.3}{$\substack{(V,\wt_V) \\ = (V,\wt_{\operatorname{quot},\varphi}) }$}} 
\end{tikzcd}
\end{minipage}

\text{}\\
\\
We will refer to this data collectively by the tuple \(\HSubspaceQuotient\).

\text{}\\
Dually, a \emph{projective metric space $W$ with $(V,\wt_V)$ as subspace} consists of the following data:\\
\\
\begin{minipage}{0.6\textwidth}
\begin{enumerate}

    \item Integers $r \geq t \geq n$;
    \item The space $\Ff_q^r$ endowed with the Hamming metric;
     \item A surjective map $\psi : \Ff_q^r \twoheadrightarrow W$, with $W$ a space of dimension $t$ endowed with the quotient weight $\wt_W := \wt_{\operatorname{quot},\psi}$; 
    \item An injective linear map $\iota: V \hookrightarrow W$ that constitutes an embedding $(V,\wt_V ) \hookrightarrow (W,\wt_W )$, i.e. $\wt_V(x) = \wt_W(\iota(x))$ for all $x \in V$.
\end{enumerate}
\end{minipage}
\begin{minipage}{0.3\textwidth}
\text{}\\
\\
\begin{tikzcd}
{(\Ff_q^r,\wtH)} \arrow[rd, "\psi", two heads]                         &             \\
& {(W,\wt_{\operatorname{quot},\psi})} \\
{\scalebox{1.3}{$\substack{\text{}\quad(V,\wt_V) \\ \text{}\quad \cong (\iota(V),\wt_W|_{\iota(V)}) }$}} \arrow[ru, "\iota"', hook] &        
\end{tikzcd}
\end{minipage}

\text{}\\
\\
We will refer to this data collectively by the tuple \(\PSubspace\).

\subsection{Existence}

First we show that there exists a Hamming metric space $\HSubspaceQuotient$ with $(V,\wt_V)$ as quotient, i.e. the set on the LHS of equation \eqref{eq: bijection subspaces and projective metrics} is non-empty.   

Consider a set $\{v_1,v_2,\ldots,v_N\}\subset V$ of representatives of all 1-dimensional spaces in $\Gr_1(V)$ of size $N = \frac{q^n-1}{q-1}$. A \myfont{free weighted space on $V$} is defined as a space $\Ff_q^N$ with standard basis $\{e_1,e_2,\ldots,e_N\}$ endowed with the \myfont{free $V$-weight} given by
\[
\wt_{\operatorname{free},V}(x) := \sum_{\substack{ i \in \{1,\ldots,N\} \\ x^{(i)} \neq 0}}\wt_V(v_i)
\]
for any $x = (x^{(1)},x^{(2)},\ldots,x^{(N)}) = \sum_{i=1}^N x^{(i)} e_i \in \Ff_q^N$.

Note that free weighted spaces on $V$ heavily depend on the choice of representatives, but are all linearly isometric.

\begin{example}
 Let $V$ be endowed with the discrete metric. Then the free $V$-weight $\wt_{\operatorname{free},V}$ on $\Ff_q^N$ with respect to the standard basis is equal to the Hamming weight.  
\end{example}

Now consider a free weighted space $(\Ff_q^N, \wt_{\operatorname{free},V})$ with basis $\{e_i\}_{i = 1}^N$ and the surjective linear map
\begin{align}\label{eq: varphi 1}
\varphi: \Ff_q^{N} &\twoheadrightarrow V :\,  (x^{(1)},x^{(2)},\ldots,x^{(N)}) \mapsto \sum_{i = 1}^N x^{(i)} v_i
\end{align}
represented by the matrix
\[
M_{\varphi} = \left(
\begin{array}{c}
 - \, v_1 \, - \\
  \hline
  - \, v_2\, -\\
  \hline
   \vdots\\
  \hline
  - \,v_N \,-\\
\end{array}
\right).
\]
\begin{lemma}\label{lemma: quotient of Hamming subspace}
The weight $\wt_V$  is equal to the quotient weight $\wt_{\operatorname{quot},\varphi}$ on $V$ induced by $\varphi$.
\end{lemma}

\begin{proof}
Let $v \in V$.  By definition of free weighted spaces there is a basis vector $e_j$ such that $\varphi(\lambda e_j) = v$ for some $j$ and $\lambda \in \Ff_q$. Hence $\wt_{\operatorname{quot},\varphi}(v) \leq \wt_{\operatorname{free},V}(\lambda e_j) = \wt_V(v)$. On the other hand, if $x = (x^{(1)}, x^{(2)}, \ldots,x^{(N)}) \in \varphi^{-1}(v)$, then $v = \sum_{i=1}^N x^{(i)} v_i$ and so
\[
\wt_{\operatorname{free},V}(x) = \sum_{\substack{ i \in \{1,\ldots,N\} \\ x^{(i)} \neq 0}}\wt_V(v_i) = \sum_{\substack{ i = 1 }}^N \wt_V(x^{(i)} v_i) \geq \wt_V\left( \sum_{i = 1}^N x^{(i)} v_i \right) = \wt_V(v)
\]
by the triangle inequality. From the definition of the quotient metric we thus obtain $\wt_{\operatorname{quot},\varphi}(v) \geq \wt_V(v)$.
\end{proof}

As $\wt_{V}$ is finitely valued, we introduce

\[
t_i := \wt_V(v_i)
\]
\[
\mathds{1}_{t} := (\overbrace{1 \, 1 \cdots 1 }^t)  \qquad \mathbb{O}_{t} := (\overbrace{0 \, 0 \cdots 0 }^{t} )
\]

for $t \in \Nn$. With these vectors we define the injective linear map
\begin{align}\label{eq: rho 1}
\rho: \Ff_q^N &\to \Ff_q^{\text{}^{\scalebox{0.75}{$\sum t_i$}}} : \, (x_1,x_2,\ldots,x_N) \mapsto (x_1 \mathds{1}_{t_1} \mid x_2 \mathds{1}_{t_2} \mid \cdots \mid x_N \mathds{1}_{t_N} ).
\end{align}
This map is represented by the matrix
\begin{align*}
M_{\rho}   
&= 
\left(
\begin{array}{c|c|c|c}
  e_1^\top \outertimes \mathds{1}_{t_1}  &  e_2^\top \outertimes \mathds{1}_{t_2} & \cdots &  e_N^\top \outertimes \mathds{1}_{t_N}\\
\end{array}
\right)\\
&= 
\left(
\begin{array}{c|c|c|c}
  \mathds{1}_{t_1}  & \mathbb{O}_{t_2} & \cdots & \mathbb{O}_{t_N}\\
  \hline
   \mathbb{O}_{t_1} &  \mathds{1}_{t_2}
   &  \cdots & \mathbb{O}_{t_N}  \\ 
   \hline
   \vdots & \vdots & \ddots & \vdots\\
   \hline
    \mathbb{O}_{t_1} &  \mathbb{O}_{t_2}  & \cdots & \mathds{1}_{t_N}\\
\end{array}
\right) \in \Ff_q^{N \times \sum t_i}
\end{align*}
Note that the Hamming weight $\wtH$ of the vector $\mathds{1}_{t}$ is equal to $t$, so for $x = (x^{(1)},x^{(2)},\ldots,x^{(N)}) \in \Ff_q^N$ we have
\[
\wt_{\operatorname{free},V}(x) = \sum_{\substack{ i \in \{1,\ldots,N\} \\ x^{(i)} \neq 0}}\wt_V(v_i) =  \sum_{\substack{ i = 1 }}^N \wtH(x^{(i)} \mathds{1}_{t_i}) = \wtH(\rho(x))
\]
and thus the following lemma holds.
\begin{lemma}
Suppose $\Ff_q^{\text{}^{\scalebox{0.75}{$\sum t_i$}}}$ is endowed with the Hamming weight $\wtH$. Then $\rho: (\Ff_q^N, \wt_{\operatorname{free},V}) \to (\Ff_q^{\text{}^{\scalebox{0.75}{$\sum t_i$}}}, \wtH)$ is an embedding.
\end{lemma}

\begin{corollary}\label{cor: existence hamming subspace}
The tuple $(\varphi, \Ff_q^N,\rho,\Ff_q^{\text{}^{\scalebox{0.75}{$\sum t_i$}}})$, with $N = 
 \frac{q^n-1}{q-1}$ and $\varphi$,$\rho$ defined by \eqref{eq: varphi 1},\eqref{eq: rho 1}, is a Hamming metric subspace with $(V,\wt_V)$ as quotient:
\[
\begin{tikzcd}
& (\Ff_q^{\text{}^{\scalebox{0.75}{$\sum t_i$}}}, \wtH)    \\
\text{}\qquad(\Ff_q^N,  \wt_{\operatorname{free},V}) \arrow[ru, "\rho", hook] \arrow[rd, "\varphi"', two heads] &                 \\
    & {\scalebox{1.3}{$\substack{(V,\wt_V) \\ = (V,\wt_{\operatorname{quot},\varphi}). }$}} 
\end{tikzcd}
\]
\end{corollary}

%

\subsection{Bijection between Hamming subspaces and projective metric spaces}

Next, we demonstrate the natural bijection between  equivalence classes of Hammming metric subspaces with $(V,\wt_V)$ as quotient, and equivalence classes of projective metric subspaces with $(V,\wt_V)$ as subspace.
\begin{lemma}\label{lem: contraction composed embedding}
    Let \(\rho: (X,\wt_X) \to (Y,\wt_Y)\) be an embedding and \(\proj: (Y,\wt_Y) \to (Z,\wt_{\operatorname{quot},\proj})\) be a surjective linear map such that \(\ker(\proj) \subset \operatorname{Im}(\rho)\). Then, the quotient weight satisfies:
    \[
    \Qwt{\proj \circ \rho}|_{\proj(\rho(X))} = \Qwt{\proj}|_{\proj(\rho(X))}.
    \]
\end{lemma}

\begin{proof}
Since \(\ker(\proj) \subset \operatorname{Im}(\rho)\), it follows that 
\(
\rho(\ker(\proj \circ \rho)) = \ker(\proj).
\)
Let \( z \in \proj(\rho(X)) \), and let \( x \in X \) be such that \( z = \proj(\rho(x)) \). Then, by the definition of the quotient weight:
\[
\wt_{\proj}(z) = \min_{y \in \rho(x) + \ker(\proj)} \wt_Y(y) =
 \min_{x' \in x + \ker(\proj \circ \rho)} \wt_X(x') 
= \wt_{\proj \circ \rho}(z) 
\]
which completes the proof.
\end{proof}
The following simple lemma implies that the restriction of a quotient weight is a quotient weight.
\begin{lemma} \label{lem: restriction of quotient is quotient}
    Let \(\varphi:X \to W\) be a surjective linear map and consider the induced quotient weight \(\Qwt{\varphi}\) on $W$. If \(\HSubspace \subset X\) is a linear subspace such that \(\ker(\varphi) \subset \HSubspace\), and \(V := \varphi(C)\), then \(\Qwt{\varphi}|_V= \Qwt{\varphi|_{\HSubspace}}\).
\end{lemma}
\begin{theorem}\label{thm: bijection Hamm subspace proj metric}
    Let \((V, \wt_V)\) be an \(n\)-dimensional vector space over \(\mathbb{F}_q\) equipped with a \textit{finitely-valued}, scale-invariant weight \(\wt_V\). Then, there exists a natural bijection between:
    \[
    \left\{\substack{\text{Hamming metric subspaces}\\ \text{with $(V,\wt_V)$ as quotient}}\right\} \,\overset{\cong}{\longleftrightarrow} \,
    \left\{\substack{\text{Projective metric spaces}\\ \text{containing $(V,\wt_V)$ as a subspace}}\right\},
    \]
    where
    \begin{enumerate}
        \item The first set is endowed with the equivalence relation: \( \HSubspaceQuotientNum{1} \cong \HSubspaceQuotientNum{2} \) if \(r_1 =r_2=r\) and there exists a linear Hamming isometry \(L \in \GL(\bF_q^r)\) such that \(L(\HSubspace_1)=\HSubspace_2\) and \(\varphi_2 \circ L|_{\HSubspace_1} = \varphi_1 \). 
    
        \item The second set is endowed with the equivalence relation: \(\PSubspaceNum{1}\cong \PSubspaceNum{2}\) if \(r_1 =r_2=r\) and there exists a linear isometry respect to the projective metrics \(T:W_1 \to W_2\) such that \(T \circ \iota_1 = \iota_2\).
    \end{enumerate}   
\end{theorem}

\begin{proof}
    Let \(\HSubspaceQuotient\) be a Hamming metric subspace with \((V, \wt_V)\) as its quotient. Define the quotient map
    \[
    \proj: \mathbb{F}_q^r \twoheadrightarrow W \coloneqq \faktor{\mathbb{F}_q^r}{\rho(\ker (\varphi))}
    \]
    where \(\varphi: \HSubspace \twoheadrightarrow V\) satisfies \(\wt_V = \wt_{\operatorname{quot},\varphi}\), and \(\rho: \HSubspace \hookrightarrow \mathbb{F}_q^r\) is the natural inclusion (an embedding). Since \((W, \wt_{\operatorname{quot}, \proj})\) is a quotient weight, it is also a projective weight by Corollary~\ref{cor: reduction to parent functions}. Since \(\ker(\varphi) = \rho(\ker(\varphi)) = \ker(\proj \circ \rho)\), the Universal Property (Theorem~\ref{thm: universal property}) guarantees the existence of a unique contraction
    \[
    \iota: (V, \Qwt{V}) \to (W, \Qwt{\proj \circ \rho})
    \]
    such that \( \iota \circ \varphi = \proj \circ \rho \) and furthermore \(\Qwt{\iota} = \Qwt{\psi \circ \rho}\) on ${\iota(V)}$, and \( \ker(\iota) = \varphi(\ker{\proj \circ \rho}) = 0 \), implying that \( \iota \) is injective.
\[
\begin{tikzcd}
& (\Ff_q^r, \wtH)   \arrow[rd, "\psi", two heads] &\\
(C,  \wtH|_{C}) \arrow[ru, "\rho", hook] \arrow[rd, "\varphi"', two heads]  \arrow[rr, "\psi \circ \rho \ = \ \iota \circ \varphi "]  &    & (W,\wt_{\operatorname{quot},\psi})           \\
    &(V,\wt_{\operatorname{quot},\varphi})  \arrow[ru, "\exists ! \,\iota"', hook, dashed]&
\end{tikzcd}
\]

    Since \(\Qwt{\iota} = \Qwt{\psi \circ \rho}\) and by Lemma~\ref{lem: contraction composed embedding}, since \(\ker(\proj) = \rho(\ker(\varphi))\), we have:
    \[
    \wt_{\operatorname{quot}, \iota }|_{\iota(V)} = \wt_{\operatorname{quot},\proj \circ \rho }|_{\iota(V)} = \wt_{\operatorname{quot},\proj }|_{\iota(V)}.
    \]
    
 Thus, \( \iota: (V, \wt_V) \to (W, \Qwt{\proj}) \) is an embedding in a projective metric space.

    We now define the mapping:
    \[
    \Psi:  \left\{\substack{\text{Hamming metric subspaces}\\ \text{with $(V,\wt_V)$ as quotient}}\right\} \, \to
    \left\{\substack{\text{Projective metric spaces}\\ \text{containing $(V,\wt_V)$ as a subspace}}\right\}
    \]
    by setting \( \Psi(\HSubspaceQuotient)) := \PSubspace \).  Note that $\dim(W) = r - \dim(\ker(\varphi)) = r-(\dim(C) - \dim(V))) = r - s + n$.
    We will now show that the map 
    \[
    \overline{\Psi}(\eqclass{\HSubspaceQuotient})\coloneqq\eqclass{\Psi(\HSubspaceQuotient)}
    \]
    on the equivalence classes $\eqclass{ \, \cdot \, }$ is well-defined and a bijection. 

\paragraph{\(\overline{\Psi}\) is well-defined:}
We must verify that the image of a Hamming metric subspace \(\HSubspaceQuotient\) under \(\Psi\) is well-defined up to linear isomorphism of projective metric spaces. 

Suppose two Hamming metric subspaces \(\HSubspaceQuotientNum{1}\) and \(\HSubspaceQuotientNum{2}\) are equivalent under a Hamming isometry \(L \in \GL(\mathbb{F}_q^r)\) (with \(r = r_1 = r_2\)) such that \(L(\HSubspace_1) = \HSubspace_2\) and \(\varphi \circ L|_{\cC_1} = \varphi_2\). Let \(\PSubspaceNum{1}\) and \(\PSubspaceNum{2}\) be the respective images of \(\HSubspaceQuotientNum{1}\) and \(\HSubspaceQuotientNum{2}\) under \(\Psi\).

Since \(L(\HSubspace_1) = \HSubspace_2\), from the definition of \(\psi_i\), from Theorem \ref{thm: isometry isomorphism}, there exist a unique projective isometry \(T \in \Fixisom{\iota_1(V)}{\iota_2(V)}{W_1}{W_2}\) such that:
\[
T \circ \psi_1 = \psi_2 \circ L,
\]
 which implies \(T \circ \iota_1 = \iota_2\).

Hence, \(\Psi(\HSubspaceQuotientNum{1})\) and \(\Psi(\HSubspaceQuotientNum{2})\) are isomorphic projective metric spaces, showing that \(\Psi\) is well-defined on equivalence classes.


\paragraph{Inverse of \(\overline{\Psi}\):}
We now describe the inverse map \(\Phi\) of \(\Psi\). Given an embedding \(\mathcal{P} = \PSubspace\) of \((V, \wt_V)\) into a projective metric space, we define:
\[
\Phi(\mathcal{P}) := \mathcal{H} := \HSubspaceQuotient,
\]
where:
\begin{itemize}
    \item \(\mathbb{F}_q^r := \psi^{-1}(W)\),
    \item \(\HSubspace := \psi^{-1}(\iota(V))\), 
    \item \(\rho : \HSubspace \hookrightarrow \mathbb{F}_q^r\) the natural inclusion map, and
    \item \(\varphi := \iota^{-1} \circ  \psi|_{\HSubspace}\).
\end{itemize}
\[
\begin{tikzcd}
& (\Ff_q^r, \wtH)   \arrow[rd, "\psi", two heads] &\\
(C,  \wtH|_{C}) \arrow[ru, "\rho", hook] \arrow[rd, "\varphi \,=\, \iota^{-1} \circ \,   \psi|_{C}"', two heads, dashed]  \arrow[rr, "\psi \circ \rho \, = \, \psi|_{\HSubspace} \,=\,  \iota \circ \varphi "]  &    & (W,\wt_{\operatorname{quot},\psi})           \\
    &(V,\wt_{\operatorname{quot},\varphi})  \arrow[ru, "\iota"', hook]&
\end{tikzcd}
\]

Using the same arguments as in the proof of \(\Psi\)'s well-definedness and surjectivity, it follows that:
\[
\overline{\Psi} \circ \overline{\Phi} = \overline{\operatorname{Id}} \quad \text{and} \quad \overline{\Phi} \circ \overline{\Psi} = \overline{\operatorname{Id}},
\]
showing that \(\Phi\) is indeed a two-sided inverse to \(\Psi\). It remains to show that \(\Phi\) is well-defined on equivalence classes of projective metric embeddings.

Let \(\PSubspaceNum{1}\) and \(\PSubspaceNum{2}\) be two projective metric embeddings of \((V, \wt_V)\) that are equivalent under a linear isometry \(T: W_1 \to W_2\) satisfying \(T \circ \iota_1 = \iota_2\). We show that \(\HSubspaceQuotientNum{1}= \Phi\left(\PSubspaceNum{1}\right)\) and \(\HSubspaceQuotientNum{2}= \Phi\left(\PSubspaceNum{2}\right)\) are equivalent.
Since \(T: W_1 \to W_2\) is a projective isometry such that \(T(\iota_1(V))=\iota_2(V)\), by Theorem \ref{thm: isometry isomorphism}, there exists a unique Hamming Isometry \(L: \bF_q^r \to \bF_q^r\) such that \(L(\cC_1)=L(\psi_1^{-1}(\iota_1(V)))=\psi_2^{-1}(\iota_2(V))= \cC_2\) and
\begin{equation}
    \label{eq: ToP}
T \circ \psi_1 = \psi_2 \circ L.
\end{equation}

From the commutativity of the diagram

\[\begin{tikzcd}[ampersand replacement=\&]
	{\HSubspace_1} \&\& V \&\& V \&\& {\HSubspace_1} \\
	\& {} \&\& {} \&\& {} \\
	{\bF_q^r} \&\& {W_1} \&\& {W_2} \&\& {\bF_q^r}
	\arrow["{\varphi_1}", from=1-1, to=1-3]
	\arrow["{\rho_1}"', from=1-1, to=3-1]
	\arrow["{Id}", from=1-3, to=1-5]
	\arrow["{\iota_1}"{description}, from=1-3, to=3-3]
	\arrow["{\iota_2}"{description}, from=1-5, to=3-5]
	\arrow["{\varphi_2}"', from=1-7, to=1-5]
	\arrow["{\rho_2}", from=1-7, to=3-7]
	\arrow["{\psi_1}"', from=3-1, to=3-3]
	\arrow["L"', curve={height=30pt}, from=3-1, to=3-7]
	\arrow["T"', from=3-3, to=3-5]
	\arrow["{\psi_2}", from=3-7, to=3-5]
\end{tikzcd}\]
it follows that:
\[
\varphi_2 \circ L|_{\HSubspace_1} = \varphi_1.
\]

This shows that \(\Phi(\PSubspaceNum{1})\) and \(\Phi(\PSubspaceNum{2})\) are Hamming equivalent, hence \(\Phi\) is well-defined on equivalence classes.

We conclude that \(\Phi\) is a well-defined inverse to \(\Psi\), completing the bijection.

\end{proof}

Finally since we have seen in Lemma \ref{lemma: quotient of Hamming subspace} that there exists an Hamming metric subspace with \((V, \wt_V)\) as a quotient it follows that:
\scaleinvariantembedding*

\newpage
\subsection{The embedding frontier}
We can now study the fundamental question:\\
\\
\textit{\mbox{``Given a scale-invariant weighted space ${(V,\wt_V)}$ of dimension $n$, what are the  possible values $(a,b)$}\\ such that ${(V,\wt_V)}$ is embeddable into  a projective  metric space $(W, \wt_\F)$ of dimension $n+a$ with  $\wt_\F$ equal to the quotient weight induced by a surjective map $\psi: (\Ff_q^{n+a+b},\wtH) \twoheadrightarrow W$?"}

\[
\begin{tikzcd}
{(\Ff_q^{n+a+b},\wtH)} \arrow[rd, "\psi", two heads]                         &             \\
& {(W,\wt_{\operatorname{quot},\psi}) \cong \Ff_q^{n+a}} \\
{\scalebox{1.0}{${(V,\wt_V)}$}} \arrow[ru, "\iota"', hook] &        
\end{tikzcd}
\]
If a pair $(a,b) \in \Nn^2$ satisfies these conditions, then $(V,\wt_V)$ is \myfont{$\bm{(a,b)}$-Hamming-embeddable} or \myfont{$\bm{(a,b)}$-embeddable} for short. 
Note that if $(V,\wt_V)$ is \myfont{$\bm{(a,b)}$-embeddable}, then it is also $(a+1,b)$-embeddable and $(a,b+1)$-embeddable by adding coordinates.
Hence, it suffices to find pairs minimal in this regard: we call a pair $(a,b) \in \Nn^2$  \myfont{(Pareto) optimal} for $(V,\wt_V)$ if
\begin{itemize}
    \item $(V,\wt_V)$ is $(a,b)$-embeddable;
    \item $(V,\wt_V)$ is not $(a',b')$-embeddable for any $(a',b')$ with  $a' < a$ and $b' \leq b$, or with  $a' \leq a$ and $b' < b$.
\end{itemize}

The finite set of all Pareto optimal pairs $(a,b)$ for $(V,\wt_V)$  is called the \textbf{(Pareto) embedding frontier} of $(V,\wt_V)$, and is an interesting geometric invariant of the weighted space. 

Note that if a pair $(a,b)$ is optimal, then the corresponding $(n+a)$-dimensional projective metric space $(W,\wt_\F)$ has a spanning family $\F$ of size $n+a+b$. Namely, if it were smaller, then we could construct a different map $\psi' : \Ff_q^{b'} \twoheadrightarrow W$ with $b' = |\F| - (n+a) < b$, contradicting optimality.

\begin{example}
Consider a scale-translation-invariant space $(V,\wt_V)$ over $\Ff_q$ of dimension $n$. Let $\{v_1,v_2,\ldots,v_N\}\subset V$ be a set of representatives of all 1-dimensional spaces in $\Gr_1(V)$ of size $N = \frac{q^n-1}{q-1}$, and let $t_i := \wt_V(v_i)$. Then Corollary \ref{cor: existence hamming subspace}  and Theorem \ref{thm: bijection Hamm subspace proj metric} show that $(V,\wt_V)$ is $\left(\sum_{i=1}^N (t_i - 1), \ N-n\right)$-embeddable, but often this is not Pareto optimal.

\end{example}

\begin{example}
Consider a scale-invariant weighted space $(V,\wt_V)$ over $\Ff_q$ of dimension $n$.
\begin{itemize}
    \item If the embedding frontier contains a pair of the form $(a,0)$ for some $a \in \Nn$, then $(V,\wt_V)$ is isomorphic to a subspace $(C,\wtH|_C)$ of the Hamming metric space $(\Ff_q^{n+a},\wtH)$.
    \item If the embedding frontier contains a pair of the form $(0,b)$ for some $b \in \Nn$, then $(V,\wt_V)$ is a projective metric space isomorphic to a quotient of the Hamming metric space $(\Ff_q^{n+b},\wtH)$, under a surjective map  $\psi: (\Ff_q^{n+b},\wtH) \twoheadrightarrow (W,\Qwt{\psi}) \cong  (V,\wt_V)$.
\end{itemize}

\end{example}

\newpage
\begin{example}
    Consider the vector space $V = \Ff_2^2 = \{0,v_1,v_2,v_3\}$ endowed with a weight $\wt_V$, and let $t_i := \wt_V(v_i)$ for $i = 1,2,3$. Let $D :=  \left\lceil \frac{t_1+t_2+t_3}{2} \right\rceil$. We can now consider the injective linear map
     \begin{align*}
            \iota : \Ff_2^2 \ &  \to \ \Ff_2^{D}\\
            v_2 \ &\mapsto \ (\overbrace{1 \, 1 \, 1 \cdots\cdots 1 \, 1 \, 1 }^{t_2} \overbrace{0 \, 0 \cdots 0 \, 0 }^{D-t_2}\hspace{-0mm})\\
            v_3 \ &\mapsto \    (\hspace{-0mm}\overbrace{0 \, 0 \cdots 0 \, 0 }^{D-t_3}  \overbrace{1 \, 1 \, 1 \cdots \cdots 1\, 1\, 1 }^{t_3}) \\
            v_1 \ &\mapsto \ (\hspace{-0mm}\overbrace{1 \, 1 \cdots 1 \, 1 }^{D-t_3}  \hspace{-0mm} 0 \cdots 0 \hspace{-0mm}\overbrace{1 \, 1 \cdots 1 \, 1 }^{D-t_2}\hspace{-0mm}).
        \end{align*}
    Note that the triangle inequality $t_1 \leq t_2 + t_3$ always ensures that $(D-t_3)+(D-t_2) \leq D$. The images under the map above have the following Hamming weights:
    \[
    \wtH(\iota(v_2)) = t_2, \qquad \wtH(\iota(v_3)) = t_3, \qquad \wtH(\iota(v_1)) = \begin{cases}
        t_1 & \text{ if } t_1+t_2+t_3 \text{ is even} \\
         t_1+1 & \text{ if } t_1+t_2+t_3 \text{ is odd} 
    \end{cases}.
    \]
    \begin{itemize}
        \item Case 1: If $t_1+t_2+t_3$ is even, then by endowing the space $\Ff_2^D$ with the Hamming weight $\wtH$, the map $\iota$ is an embedding. So $(\Ff_2^2, \wt_V)$ is $(D-2,\, 0)$-embeddable, and it is easy to check that $(D-2,\, 0)$ is a Pareto optimal pair.

        \item Case 2: If $t_1+t_2+t_3$ is odd, then we endow the space $\Ff_2^D$ with the projective weight $\wt_\F$, with spanning family $\F = \{e_1, e_2,\ldots,e_D, (100\cdots001)\}$ and $e_j$ the $j$-th standard basis vector of $\Ff_2^D$.
        Now $\wt_\F(\iota(v_i)) = t_i$ for $i=1,2,3$, so the map $\iota$ is an embedding. This means $(\Ff_2^2, \wt_V)$ is $(D-2,\, 1)$-embeddable, and again it is easy to check that this is a Pareto optimal pair.
    \end{itemize}

\noindent
In conclusion, $(\Ff_2^2, \wt_V)$ is $(a,b)$-embeddable with $a \in \Nn$ and $b \in \{0,1\}$ given by
\[
t_1 + t_2 + t_3 - 4  = 2a - b
\]
with this $(a,b)$ being Pareto optimal for $(\Ff_2^2, \wt_V)$.
\end{example}

\newpage

\section{Sphere sizes and matroids}\label{sec: Sphere sizes}

The study of sphere sizes in projective metrics is fundamental to understanding error correction capabilities and their relationships with classical coding bounds. Sphere sizes play a crucial role in determining the packing efficiency of a code and its ability to correct errors. 
A first relation between sphere sizes in projective metric and coset weight distributions of their parent codes has been found in \cite{gabidulin1998metrics}. 

One of the most well-known theoretical bounds in coding theory is the \emph{sphere-packing bound} \cite{shannon} or \emph{Hamming bound} (\cite{Hamming},\cite{MacWilliams}) providing a limit on the number of codewords a code can contain while still maintaining a certain minimum distance. Codes attaining the Hamming bound are called \emph{perfect codes}. In subsection \ref{subsectiond: perfect codes}, we discuss perfect codes in projective metrics, referencing the foundational work of Gabidulin and others on such constructions.

In this section, we give some basic results on sphere sizes and relate sphere sizes to the coset weight distribution of \emph{parent codes}: for any code \(\pc\) there exists a projective weight whose sequence of sphere sizes is equal to the coset weight distribution of \(\pc\) and viceversa.  Additionally in subsection \ref{subsec: matroids}, we show how a \emph{matroid} structure defined from the spanning family \(\cF\) univocally describes sphere sizes of the projective metric \(\wt_{\cF}\).

\begin{definition}
Given a projective weight \(\wt_\cF\) on \(V\), an integer \(t \in \bN\) and vector \(x \in V\), the \textbf{sphere} and \textbf{ball of} \textbf{radius} \(\bm{t}\) \textbf{centered at} \(\bm{x}\) are  respectively defined as
\begin{align*}
    \SF_t(x) &:= \{x + v \ : \  v \in V \ \text{ and } \ \wt_\cF(v) = t\}\\
    \BF_t(x) &:= \{x + v \ : \  v \in V \ \text{ and } \ \wt_\cF(v) \leq t\}.
\end{align*}
For spheres and balls in the Hamming metric in particular, we shall write \(\SH_t(x)\) and \(\BH_t(x)\) respectively.
\end{definition}


\subsection{Some observations}\label{subsec: some obs}

The following lemma says that the restriction of the parent function to the ball of radius \(r\) equal to half of the minimum distance of the parent code, behaves like an isometry:

\begin{lemma}\label{lem: parent code injective and weight preserving}
    Let \(\wt_\cF\) be a projective weight on \(V\) with \(\cF = \{\langle f_1 \rangle,\ldots, \langle f_\N \rangle\}\) and \(\varphi : \bF_q^\N \to V\) a parent function with parent code  \(\pc = \ker(\varphi)\). Let \(r := \left \lfloor \frac{\dH(\pc) - 1}{2} \right \rfloor\) and \(x \in V\). Then \(\varphi\) is injective on the ball \(\BH_r(x) \subset \bF_q^\N\), i.e.
    if \(y,z\) are two vectors in \(\BH_r(x) \) such that \(\varphi(y) = \varphi(z)\), then \(y = z\).
    Moreover, if \(y \in \BH_r(0)\), then  \(\wt_\cF(\varphi(y)) =\wtH(y)\).
\end{lemma}

\begin{proof}
Let  \(y,z \in \BH_r(x) \) such that \(\varphi(y) = \varphi(z)\). Then \(y - z \in \pc\), but \(\wt_\cF(y-z) = \wt_\cF((y-x)-(z-x)) \leq  \wt_\cF(y-x) +  \wt_\cF(z-x) \leq 2r = \dH(\pc) - 1\), so \(y = z\). 

To show the second statement, let \(y \in \BH_r(0)\) and suppose \(y = \sum_{i=1}^\N y_i e_i\).
Then we have \( \varphi(y) =  \sum_{i=1}^\N y_i\varphi(e_i) =\sum_{i=1}^\N y_i f_i\) and thus \(\wt_\cF(\varphi(y)) \leq \wtH(y)\).
 By definition of projective metrics there exists a \(z = (z_1, \ldots, z_\N) \in \bF_q^\N\) such that \( \varphi(y) = \sum_{i=1}^\N z_i f_i = \varphi(z)\) with \(\wtH(z) = \wt_\cF(\varphi(y))\). We obtain \(\wtH(z) = \wt_\cF(\varphi(y)) \leq \wtH(y) \leq r\), so by the first statement it follows that \(y = z\) and thus the inequalities above are in fact equalities.
\end{proof}

As a corollary we get that for small enough radii, the parent functions are bijections between spheres (balls) in the Hamming metric and spheres (balls) in projective metrics:

\begin{proposition}\label{prop: small spheres bijection}
Let \(\wt_\cF\) be a projective weight on \(V\)  and \(\varphi : \bF_q^\N \to V\) a parent function with parent code  \(\pc = \ker(\varphi)\). Let \(r := \left \lfloor \frac{\dH(\pc) - 1}{2} \right \rfloor\) and \(x \in \bF_q^\N\). Then for any \(t \leq r\), the function \(\varphi\) acts as a bijection between the sphere \(\SH_t(x) \subset \bF_q^\N\) and the sphere \(\SF_t(\varphi(x)) \subset V\).
\end{proposition}

\begin{proof}
Since \(\SH_t(x) = x + \SH_t(0)\), \ \(\SF_t(\varphi(x)) = \varphi(x) + \SF_t(0)\) and \(\varphi\) linear, we can assume w.l.o.g. that \(x = 0\).

Lemma \ref{lem: parent code injective and weight preserving} implies directly that \(\varphi(\SH_t(0)) \subseteq \SF_t(0)\) and that the restriction \(\varphi |_{\SH_t(0)}\) is injective.
Since by definition of projective metrics for any \(z \in \SF_t(0)\)  there exists a \(y =  \sum_{i=1}^\N y_i e_i \in \bF_q^\N \) such that \( \sum_{i=1}^\N y_i f_i = z\) and \(\wtH(y) = \wt_\cF(z) = t\), the restriction \(\varphi |_{\SH_t(0)}\) is also surjective. 
\end{proof}

\begin{corollary}[\cite{gabidulin1997metrics,gabidulin1998metrics}]
Let $y \in V$. Then for $t \leq  \left\lfloor \frac{\dH(\pc) - 1}{2} \right \rfloor$, the sphere size of radius $t$ is given by
\[
|\SF_{t}(y)| = |\SF_{t}(0)| =  (q-1)^{t}\binom{\N}{t}.
\]
\end{corollary}

 We now show how, knowing the ball sizes with the weights \(\wt_{\cF_1}\) and \(\wt_{\cF_2}\), we can obtain the ball sizes for their disjoint union \eqref{def: disjoint union of projective weights}: \(\wt_{\cF} \coloneqq \wt_{\cF_1} \sqcup \wt_{\cF_2}\).

\begin{proposition}
     
    Let \(R_1,R_2\) be the radiuses of the largest non empty spheres with the weights \(w_{\cF_1}\) and \(w_{\cF_2}\)
    that is \(R_i := \max \{r \in \bN \mid S_r^{\cF_i}(0) \neq \emptyset \}\) for i=1,2. Then 
    \begin{equation}
    \label{eq: splitting spheres}
        S_r^{\cF}(0) = \bigsqcup_{\substack{l_1 + l_2 = r \\ l_i \leq R_i} } (S_{l_1}(\cF_1) \oplus S_{l_2}(\cF_2))
    \end{equation}
\end{proposition}

\begin{proof}
We show that the union is disjoint. Let \(v \in (S_{l_1}^{\mathcal{F}_1} \oplus S_{l_2}^{\mathcal{F}_2}) \cap (S_{l'_1}^{\mathcal{F}_1} \oplus S_{l'_2}^{\mathcal{F}_2})\) with \(l_1 \neq l'_1\) and \(l_2 \neq l'_2\). There exist unique \(v_1 \in V_1\), \(v_2 \in V_2\) such that \(v = v_1 + v_2\), thus \(v_1 \in S_{l_1}^{\mathcal{F}_1} \cap S_{l'_1}^{\mathcal{F}_1} = \emptyset\) and \(v_2 \in S_{l_2}^{\mathcal{F}_2} \cap S_{l'_2}^{\mathcal{F}_2} = \emptyset\), which gives a contradiction, thus the union is disjoint. Since for all \(v_1 \in V_1\), \(v_2 \in V_2\) we have \(\wt_{\cF}(v_1+v_2) = w_{\mathcal{F}_1}(v_1) + w_{\mathcal{F}_2}(v_2)\), it follows that \(S_r^{\mathcal{F}}(0) = \bigsqcup_{\substack{l_1 + l_2 = r \\ l_i \leq R_i}} (S_{l_1}^{\mathcal{F}_1} \oplus S_{l_2}^{\mathcal{F}_2})\).

\end{proof}

Thus, given the sphere sizes of $\cF_1$ and $\cF_2$ we can have the sphere sizes of $\cF$. Taking the cardinality of the set in Equation \eqref{eq: splitting spheres}, we obtain the following:
\begin{corollary}
\begin{equation}
    |S_r^{\cF}|  = \sum_{l=0}^r |S_{r-l}^{\cF_1}| |S_{l}^{\cF_2}|
\end{equation}
    
\end{corollary}

We conclude this subsection with a connection between sphere sizes distributions and coset weight distribution.  The theory of coset leader weights (see definition \ref{def: coset leader weight} and \cite{jurrius2009extended}) revolves largely around linear subspaces or codes in the Hamming metric.  Given a linear subspace $C$ of $\Ff_q^{n}$ endowed with the Hamming metric, its \myfont{coset leader weight distribution} $\alpha_C$ is the sequence $(\alpha_0,\alpha_1,\ldots,\alpha_n)$ of length $n+1$, where $\alpha_i$ is the number of cosets of $C$ that are of coset leader weight $i$. It was already noted by Gabidulin in \cite{gabidulin1998metrics}  that the sphere size distribution of a projective metric is equal to the coset weight distribution of its parent code. In general, given any linear subspace \(\cC\) there exists a projective metric such that its coset weight distribution is equal to the sphere size distribution of the projective metric. This is a direct consequence of the interpretation of projective metrics as quotient weights (Proposition \ref{prop: bijection parent codes and projective weights}) (in case $d_{\operatorname{H}}(C) \geq 3$) and Corollary \ref{cor: reduction to parent functions}  (in case $d_{\operatorname{H}}(C) \leq 2$).

\begin{proposition}
Let $\cC$ be a linear subspace of $\Ff_q^{\N}$ endowed with the Hamming metric. Then there exists a finite vector space $V$ weighted with a projecive metric $\wt_{\cF}$ such that the coset leader weight distribution $\alpha_C$ is equal to the sequence of sphere sizes
$(S_0^\cF,S_1^\cF,\ldots,S_\N^\cF)$ of $V$. 
\end{proposition}


\subsection{Perfect codes}
\label{subsectiond: perfect codes}
A code is \emph{perfect} if there exists an integer \(t\) such that the union of all balls of radius \(t\) centred at the codewords partitions the ambient space.

\begin{definition}[Perfect codes]
    Given a distance function \(d(\cdot,\cdot)\) on \(V\) and a code \(\cC \subseteq V\), we say that \(\cC\) is \textbf{perfect} with respect to \(d(\cdot,\cdot)\) if there exists a \(t \in \bN\) such that:

    \begin{equation}
        \bigsqcup_{c \in \cC}B_t(c) = V
    \end{equation}
    where \(B_t(c) := \{x \in V : d(x,c) \leq t\}\) are disjoint balls of radius \(t\) centered around codewords.
\end{definition}

\begin{observation}
If $d(\cdot,\cdot)$ is a convex metric (e.g. a projective metric), \(\cC \subseteq V\) is perfect with balls of radius \(t\) and \(|\cC| \geq 2\), then \(t = \frac{d(\cC)-1}{2}\) with \(d(\cC) := \min_{\substack{x,y \in {\cC} \\ x \neq y}} \left(d(x,y)\right)\). This follows from the fact that for any \(c \in \cC\) there exists a different  \(c' \in \cC\) such that \(B_{t+1}(c) \cap B_{t}(c')\) is nonempty, containing some vector \(v\). Then \(d(c,c') \leq d(c,v) + d(v,c') = 2t+1\). On the other hand, for two distinct codewords \(x,x' \in \cC\) the distance \(d(x,x')\) cannot be less then \(2t+1\), otherwise, by convexity, the balls \(B_{t}(x)\)  and \(B_{t}(x')\) wouldn't be disjoint.


\end{observation}

\medskip
For clarity we often refer to perfect codes respect to the Hamming metric as \myfont{Hamming perfect} and with respect to  a projective metric with spanning family  \(\cF\) as \myfont{\(\bm{\cF}\)-perfect}.\\

Perfect codes for projective metrics were first investigated by Gabidulin and Simonis in~\cite{gabidulin1998metrics} and \cite{gabidulin1998perfect}.  
Part of their analysis relies on the following statement {\cite[Proposition 2]{gabidulin1998metrics}}:\\

\begin{remark}
    For every non-zero linear code \(0\neq\cD\subset V\) we have
  \begin{equation}\label{eq:wrong-eq}
    \dH\!\bigl(\varphi^{-1}(\cD)\bigr) = \min\{\dF(\cD),\dH(\pc)\}.
  \end{equation}
\noindent
In particular, when $\dH(\pc) \geq \max_{x \in \bF_q^\N}\{\dH(x,\pc)\} = \max_{y \in V}\{\wt_\F(y)\}$, then 
\begin{equation}\label{eq:wrong-eq}
    \dH\!\bigl(\varphi^{-1}(\cD)\bigr) = \dF(\cD).
  \end{equation}
\end{remark}


As a direct consequence {\cite[Corollary 1]{gabidulin1998metrics}}, there is a bijection between optimal codes in the Hamming metric space $\Ff_q^\N$ containing $\pc$ and optimal codes in the projective metric space $(V,\wt_\F)$. Here, optimality is with respect to cardinality and minimum distance. In the rest of this subsection we extend this result and describe a correspondence between perfect linear codes in projective metrics and perfect linear codes in the Hamming metric.



\begin{proposition}\label{prop: F-perfect codes}
   Given a projective weight \(\wt_\cF\) on \(V\), a parent function \(\varphi : \bF_q^\N \to V\) with parent code \(\pc\). Suppose that $\dH(\pc) \geq \max_{y \in V}\{\wt_\F(y)\}$. Then a linear code \(\hat\cC\) containing the parent code \(\pc \subset \hat\cC\) is Hamming perfect if and only if \(\cC \coloneqq \varphi(\hat\cC)\) if \(\cF\)-perfect. In such case we have that \(\dH(\hat\cC) = \dF(\cC)\).
\end{proposition}

\begin{proof}


     
     Let $t = \left\lfloor \frac{\dH(\hat \cC)-1}{2}\right\rfloor$. We assume that \(\hat \cC\) is Hamming perfect, so 
      \(
        \bF_q^\N = \bigsqcup_{x \in \hat \cC} B_t^H(x).
    \)
    Let \(y \in \hat{\cC}\) be any codeword. Since \(y + \pc \subseteq \hat \cC\) we have that \( \dH(\hat \cC) \leq \dH(y + \pc) = \dH(\pc)\). Hence we can use Proposition \ref{prop: small spheres bijection} to find \(\varphi(B_t^H(x)) = B_t^{\cF}(\varphi(x))\) for any \(x \in \hat \cC\) and thus
 \begin{equation*}
        V = \varphi(\bF_q^\N) = \bigcup_{x \in \hat \cC} \varphi(B_t^H(x))= \bigcup_{v \in \cC}B_t^{\cF}(v).
    \end{equation*}
Also note that if \(v,w \in \cC\) are codewords with \(d_\cF(v,w) = d_\cF(\cC)\), then  
\[
\dH(\hat \cC) = \min_{\substack{x,y \in \hat{\cC} \\ x \neq y}} \left(\dH(x,y)\right) \leq \min_{\substack{x \in \varphi^{-1}(v) \\ y  \in \varphi^{-1}(w)}} \left(\dH(x,y)\right) = d_\cF(v,w) = d_\cF(\cC).
\] 
This implies that the balls \(B_t^{\cF}(v)\) for \(v \in \cC\) are all disjoint and thus
 \begin{equation*}
        V = \bigsqcup_{v \in \cC}B_t^{\cF}(v).
    \end{equation*}
As a consequence, \(\dH(\hat{\cC}) = 2t + 1 = d_\cF(\cC)\). 

We now assume that \(\cC\) is \(\cF\)-perfect, let \(t = \left\lfloor \frac{\dH(\hat{\cC}) - 1}{2} \right\rfloor \leq \dH(\pc)\). We already know that the Hamming balls of radius \(t\) centered at the codewords of \(\hat{\cC}\) are disjoint; it remains to show that their union covers the entire space \(\mathbb{F}_q^{\N}\).

Let \(x \in \mathbb{F}_q^{\N}\). Since \(\cC\) is \(\cF\)-perfect, there exists \(y \in \cC\) such that \(\mathrm{d}_{\cF}(\varphi(x), y) \leq t\). This implies that
\[
\dH(x, \varphi^{-1}(y)) \leq t,
\]
i.e., there exists \(x' \in \hat{\cC}\) such that \(x \in B_t^{\operatorname{H}}(x')\). Therefore, \(\hat{\cC}\) is Hamming-perfect.
  
\end{proof}

\subsection{Matroids and Projective Metrics}\label{subsec: matroids}
The notion of equivalence between projective metrics based on weights is quite restrictive. 
Given two spanning families $\mathcal{F}$ and $\mathcal{G}$ in an $N$-dimensional space $V$, a projective equivalence is univocally determined by assigning $N+1$ elements of $\mathcal{F}$ to the same number of elements of $\mathcal{G}$. 
However, this may not induce an actual equivalence if the resulting linear isomorphism does not map $\mathcal{F}$ to $\mathcal{G}$.
Furthermore, many very \onequote\onequote similar'' projective metrics, sharing important characteristics such as ball sizes, etc., are not equivalent.
For this reason, in this subsection, we explore how projective metrics behave when we consider their natural matroid structure. With this structure, all equivalent projective metrics are also equivalent as matroids. However, the opposite implication is not true; therefore, matroid equivalence of projective metrics is a less strict equivalence. 
%
For an introduction to matroid theory, refer to \cite{oxley2003matroid}.

\begin{definition}[Matroid associated to $\cF$]
Given a projective metric defined by $\cF \subset V$. Let $A_{\cF} \in \bF^{N \times \N}$ be the matrix that has as columns the vectors in $\cF$. The \emph{matroid associated to } $\cF$ is $M(\cF) \coloneqq M(A_{\cF})$.
\end{definition}

\begin{remark}
\begin{enumerate}
    \item  Changing the order of the columns of $A_{\cF}$ gives equivalent matroids.
    \item \(
            \cF \equiv\cG \implies M(\cF) \equiv M(\cG) 
            \)
    \item For every cycle \(c \subset \cF\) exist $a_1,\ldots,a_{|c|} \in \bF\withoutzero$ such that $\sum_{i = 1}^{|c|}a_ic_i = 0$. 
\end{enumerate}
\end{remark}

For a general field \(\bF\), equivalence as projective weights and equivalence as matroids are not equivalent. This is because if we have a matroid equivalence \(\varphi \) between \(\cF\) and \(\cF'\), we may define the linear isomorphism \(\Phi\) mapping basis \(\cB \subset \cF\) to the basis \(\varphi(\cB)\), but this does not guarantee that \(\Phi\) will map \(\cF\) to \(\cF'\). We must note that this is indeed the case for \(\bF = \bF_2\), to see this, for every \(f \in \cF \withoutzero\) consider the smallest \( I \subset \cB\) such that \(c :=I \cup \{f\}\) is linearly dependent. Then \(c\) is a cycle respect to \(M(\cF)\), and \(f = \sum_{g \in I}g\). Then \(\varphi(c)\) is also a cycle respect to \(M(\cF')\) and thus \(\varphi(f) = \sum_{g \in I}\varphi(g) = \sum_{g \in I}\Phi(g) = \Phi(\sum_{g \in I}g) = \Phi(f)\). Thus \(\Phi\) maps \(\cF\) to \(\cF'\) and is thus a projective weights equivalence. In summary:

\begin{remark}
     For \(\bF = \bF_2\) matroid equivalence and projective equivalence are equivalent.
\end{remark}

To understand how the sizes of the balls relate to the matroid equivalence classes, we consider the inclusion-exclusion principle, which allows us to compute the cardinality of the intersections of various subspaces generated by subsets of $\cF$. In particular, we say that a hyperplane in \(V\) is generated by vectors in \(\cF\) whenever it is the span of \(N-1\) linearly independent vectors in \(\cF\). Although the matroid structure suffices to describe the intersection between any two such subspaces (see Appendix \ref{appendix: matroid subspacex intersection}), it does not provide sufficient information to describe these intersections in general. Indeed, there exist matroid-equivalent projective metrics that have different ball sizes (see \cite{jurrius2009extended}). However, the matroid associated with an extension of $\cF$ contains enough information to fully describe the intersections between subspaces generated by \(\cF\). In this subsection we see that when the matroid associated to the extend families of two projective metrics are equivalent, then the sphere sizes of the two metrics also coincide. Following the ideas of \cite{jurrius2009extended} on derived codes, we consider an extension of spanning families that contains intersections of hyperplanes generated by vectors in \(\cF\).

\begin{definition}[Extended Projective Metric]
Let $\cF \subset V$ be a spanning family. The \emph{extended projective metric} of $\cF$ is the set $\overline{\cF}$ containing all vectors $v \in V$ such that $\langle v \rangle$ is the intersection of hyperplanes generated by vectors in $\cF$. If $\cF = \overline{\cF}$, we say that $\cF$ is closed under extension.
\end{definition}

\begin{example}
Let \(V\) be a vector space of dimension \(N\) over a finite field \(\mathbb{F}\). Let \(\cB \subset V\) be a basis of \(V\), and define
\[
w \coloneqq \sum_{v \in \cB} v.
\]
Consider the projective metric on \(V\) induced by the spanning family \(\cF \coloneqq \cB \cup \{w\}\). This metric corresponds to the Phase Rotation metric (\ref{example: Phase Rotation Metric}) in the space \(\mathbb{F}^N\), where vectors are expressed in the basis \(\cB\).

\medskip

The extended family \(\overline{\cF}\) of \(\cF\) is
\[
\overline{\cF} \coloneqq \Bigl\{\, w - \sum_{g \in I} g \;\Bigm|\; I \subset \cB \Bigr\},
\]
i.e., the set of all vectors in \(V\) whose coordinates (in the basis \(\cB\)) take at most two values, one of which is zero.
\medskip
For a proof of this statement, see Appendix \ref{appendix: example matroid equivalence}.
\end{example}
The following proposition motivates the name \emph{extension} of \(\cF\).
\begin{proposition}
   \( \cF \subset \overline{\cF}\)
\end{proposition}
\begin{proof}
    This is true since for all \(f \in \cF\) there exists a basis \(\cB \subset \cF\) such that \(f \in \cB \subset \cF\), and \[\spn{f} = \cap_{\substack{f \in I \subset \cB\\ |I|=N-1}}\spn{I}.\]
\end{proof}

By definition, any one-dimensional intersection of hyperplanes generated by a vector in \(\cF\), is generated by a vector in \(\overline{\cF}\), the following proposition shows that any intersection of such hyperplanes can be generated by vectors in \(\overline{\cF}\).

\begin{proposition}
\label{prop: extension intersection}
    If \(\cF\) be a spanning family with extension \(\overline{\cF}\).
   If \(W \coloneqq \cap_{j =1}^{N-w}H_j\) where \(H_j\) are hyperplanes generated as vectors in \(\cF\), and \(\dim{W}=w\), there exist \(f_1,\ldots,f_w \in \overline{\cF}\) such that \[
    W= \spn{f_1,\ldots, f_w}
    \]
\end{proposition}
\begin{proof}
    Let \(h_1,\ldots h_{N-w} \in V\) be the vectors orthogonal to the corresponding hyperplane \(H_j\). Add \(h_N,\ldots, h_{N-w+1}\), vectors orthogonal to  hyperplanes generated by vectors in \(\cF\) to obtain a basis of \(V\) (this is always possible because \(\cF\) is a spanning family), and \(H_1, \ldots,H_N\) be the corresponding hyperplanes. Then the vectors \(\spn{f_i} \coloneqq \cap_{\substack{j \neq N-w+i}}H_j \), for \(i = 1, \ldots, w\) are linearly independent, generate \(W\) and are in \(\overline{\cF}\).
\end{proof}

The following proposition shows that the matroid associated with the extended family of \(\wt_{\mathcal{F}}\) fully describes the sphere sizes of \(\mathcal{F}\). It can be interpreted as the matroid on \(\mathcal{F}\), to which specific vectors have been added to describe the intersections of hyperplanes generated by \(\mathcal{F}\).

\begin{proposition}

    Let  \(\cF\), \(\cG\) be two spanning families with extensions $\overline{\cF}$ and $\overline{\cG}$.
    If there is a matroid isomorphism $\miso: \overline{\cF} \to \overline{\cG}$ such that $\miso(\cF) = \cG$, then the metrics $\wt_\cF$, $\wt_\cG$ have equal sphere sizes.
\end{proposition}
\begin{proof}
Let \( M_t(\mathcal{F}) \subset M(\mathcal{F}) \) denote the family of subsets of \(\mathcal{F}\) of cardinality \( t \) that are linearly independent. The ball of radius \( t \) is given by:
\[
B_t^{\mathcal{F}} = \bigcup_{J \in M_t(\mathcal{F})} \spn{J}.
\]
By the inclusion-exclusion principle, the size \( |B_t^{\mathcal{F}}| \) can be expressed in terms of the cardinalities of the intersections \( \left| \bigcap_{J \in \mathcal{J}} \spn{J} \right| \) for all \( \mathcal{J} \subset M_t(\mathcal{F}) \). Since \( \miso(M_t(\mathcal{F})) = M_t(\miso(\mathcal{F})) \) and each \( \spn{J} \) can also be seen as the intersection of hyperplanes generated by subsets of \( \mathcal{F} \). Therefore, \( \left| \bigcap_{J \in \mathcal{J}} \spn{J} \right| \) can be viewed as the intersection of \( k \) hyperplanes generated by subsets of \( \mathcal{F} \).

Thus, we need only to show that for all hyperplanes defined by \( I_1, \ldots, I_k \in M(\mathcal{F}) \), we have:
\[
\left| \bigcap_{i=1}^k \spn{I_i} \right| = \left| \bigcap_{i=1}^k \spn{\miso(I_i)} \right|.
\]
Since \( \miso^{-1} \) is also a matroid isomorphism, it suffices to show that:
\[
\left| \bigcap_{i=1}^k \spn{I_i} \right| \leq \left| \bigcap_{i=1}^k \spn{\miso(I_i)} \right|.
\]

By Proposition \ref{prop: extension intersection}  there exist \(f_1, \ldots, f_w \in \overline{\cF}\) linearly independent vectors such that \(\spn{f_1,\ldots,f_w}= \cap_{i=1}^k \spn{I_i}\), then \(\miso{(f_1)},
\ldots, \miso{(f_w)}\) are still linearly independent, we only need to show that \(\spn{\miso{((f_1))},\ldots,\miso{(f_w)}} \subset \cap_{i=1}^k \spn{\miso(I_i)} \). 
We show that \(\miso(f_i) \in \spn{\miso(I_j)} \) for all \(j=1,\ldots,k\) and for all \(i=1,\ldots,w\).
Let \(I_j \in M(\cF)\) generate \(H_j\). 
If \(f_i \in I_j\), then \(\miso(f_i) \in \miso(I_j) \subset \spn{\miso(I_j)}\). 
Otherwise, if \(f_i \notin I_j\), let \(c_j\) be the cycle such that \({f_i} \in c_j \subset \{f_i\} \cup I_j\).
Since \(\miso(c_j)\) is also a cycle, \(\miso(f_i) \in \miso(c_j) \subset \miso(I_j) \cup \{\miso(f_i)\} \), and thus \(\miso(f_i) \in \spn{\miso(I_j)}\). 
\end{proof}

The following provides a trivial counterexample to the natural question: If two projective matroids have extended families that are matroid equivalent, are they necessarily linearly equivalent? Thus matroid equivalence of the extended families of two projective metrics is a strictly weaker condition than linear equivalence.

\begin{example}
    Consider \(V = \bF_q^2\). Any projective metric is trivially closed under extension. Hence, any two projective metrics that are matroid equivalent also have matroid equivalent extended families. However, not all such projective metrics are linearly equivalent. For instance, consider the projective metrics defined by the families \(\cF \coloneqq \{(1,0), (0,1), (1,1), (1,2)\}\) and \(\cG \coloneqq \{(1,0), (0,1), (1,1), (1,3)\}\) for \(q = 7\). These projective metrics are matroid equivalent but not linearly equivalent.
\end{example}

A natural question is then if knowing the structure of the extended matroid is sufficient to know the ball sizes of a projective metric. The following result shows that this is the case.
\begin{proposition}
    Let \(\overline{\mathcal{F}}\) be the extended family of \(\mathcal{F}\). Then, we have:
    \begin{equation}
        |B^{\mathcal{F}}_t| = \sum_{\emptyset \neq H \subseteq M^t_{\mathcal{F}}}(-1)^{|H|+1}q^{|\overline{\mathcal{F}}_H|}
    \end{equation}
    where \(M^t_{\mathcal{F}} \coloneqq \{I \in M_{\mathcal{F}} \mid |I|=t\}\) and \(\overline{\mathcal{F}}_H\) is a maximal independent set in \(M_{\overline{\mathcal{F}}}\) such that for all \(f \in \overline{\mathcal{F}}_H\) and \(I \in H\), the set \(\{f\} \cup I\) forms a cycle respect to \(M_{\overline{\mathcal{F}}}\).
\end{proposition}

\begin{proof}
    By the inclusion-exclusion principle, it suffices to show that 
    \[
    |\cap_{I \in H}\spn{I}| = q^{|\overline{\mathcal{F}}_H|}.
    \]
    Specifically, we demonstrate that:
    \[
    \cap_{I \in H}\spn{I} = \spn{\overline{\mathcal{F}}_H}.
    \]
    Since for all \(f \in \overline{\mathcal{F}}_H\) and \(I \in H\), the set \(\{f\} \cup H\) forms a cycle, it follows that \(f \in \spn{I}\) for every \(I \in H\). Therefore, \(f \in \cap_{I \in H}\spn{I}\), implying \(\spn{\overline{\mathcal{F}}_H} \subset \cap_{I \in H}\spn{I}\).
    
    By Proposition \ref{prop: extension intersection}, there exist elements \(f_1, \ldots, f_w \in M_{\overline{\mathcal{F}}}\) such that:
    \[
    \spn{f_1, \ldots, f_w} = \cap_{I \in H}\spn{I}.
    \]
    
    We proceed by contradiction. Assume that \(\spn{\overline{\mathcal{F}}_H} \subsetneq \cap_{I \in H}\spn{I}\). Then, there exists an element \(f_k\) in \(\{f_1, \dots, f_w\}\) such that \(f_k \notin \spn{\overline{\mathcal{F}}_H}\). However, since \(f_k \in \spn{I}\) for all \(I \in H\), adding \(f_k\) to \(\overline{\mathcal{F}}_H\) would contradict its maximality, thus leading to a contradiction.
\end{proof}

The following proposition can be useful to construct high dimensional examples of spanning family extensions.
\begin{proposition}\label{prop: extension-direct-sum}
Let \(\cF \subset V_1\) and \(\cG \subset V_2\) be projective point families in finite-dimensional vector spaces \(V_1\) and \(V_2\), respectively. Let \(\overline{\cF}_1 \subset V_1\) and \(\overline{\cF}_2 \subset V_2\) denote their respective extensions. Then the extension of the direct sum family satisfies:
\[
\overline{\cF \oplus \cG} = \overline{\cF}_1 \oplus \overline{\cF}_2.
\]
\end{proposition}

\begin{proof}
The extension \(\overline{\cF \oplus \cG}\) is defined as the minimal set of additional projective points needed to fully describe all hyperplane intersections among the elements of \(\cF \oplus \cG\). Since \(\cF \subset V_1\) and \(\cG \subset V_2\) span disjoint coordinate spaces, any linear dependency among elements of \(\cF \oplus \cG\) must occur entirely within either \(V_1\) or \(V_2\).

Consequently, the circuits (i.e., minimal dependent sets) of the matroid \(M_{\cF \oplus \cG}\) decompose as the disjoint union of the circuits in \(M_{\cF}\) and \(M_{\cG}\). Thus, in order to consider all intersections in the direct sum, it suffices to include exactly the extensions required for each component independently.

Hence, the extended family of \(\cF \oplus \cG\) is exactly the direct sum of the extensions:
\[
\overline{\cF \oplus \cG} = \overline{\cF}_1 \oplus \overline{\cF}_2. \qedhere
\]
\end{proof}

\newpage

{
\section{Singleton bound}\label{sec: Singleton Bound}

{
A generalization of the Singleton bound for translation invariant metrics, such as projective metrics, can be obtained by introducing the concept of anticodes.

\begin{definition}
    Let \(V\) be a vector space over \(\bF\). Let \(t \in \bN\) be an integer and \(d:V\times V \to \bR\) a distance function. A \(t\)-anticode is a non-empty subset \(\cA \subset V\) such that \(d(x,y) \leq t\) for all \(x,y \in \cA \). We say that \(\cA\) is linear if it is a subspace of \(V\). 
\end{definition}

\begin{lemma}[Code-Anticode Bound]\label{lem:Linear Code-Anticode Bound}
    Let \(d\) be a translation-invariant distance function and \(t \in \bN\). If \(\cC \subset V\) is a linear code with minimum distance \(d(\cC) \geq t\), then
    \begin{equation}
        |\cC| \leq \frac{|V|}{max\{|\cA| \mid \cA \text{ is a (t-1) linear anticode}\}} \label{eq:Linear Code-Anticode Bound}
    \end{equation}
\end{lemma}

In general, the value of the denominator in \eqref{eq:Linear Code-Anticode Bound}, is hard to compute since it might require analyzing almost all subspaces of $V$. For this reason we consider the following relaxed bound for projective metrics.

\begin{lemma}[Relaxed Code-Anticode Bound for Projective Metrics]\label{lem:Relaxed Code-Anticode Bound for Projective Metrics}
  Let \(d_{\cF}\) be a projective metric and \(t \in \bN\). If \(\cC \subset V\) is  a linear code with minimum distance \(d_\F(\cC) \geq t\), then
    \begin{equation}
        |\cC| \leq \frac{|V|}{max\{|\cA| \mid \cA \text{ is a (t-1) linear anticode with basis in }\cF \}} \label{eq:Relaxed Code-Anticode Bound for Projective Metrics}
    \end{equation}
\end{lemma}

The denominator of the bound \eqref{eq:Relaxed Code-Anticode Bound for Projective Metrics} is equal to \(q^{\mu_{\cF}(t-1)}\), where \(\mu_{\cF}\) is defined as:\[\mu_{\cF}(t)\coloneqq max\{dim(\cA) \mid \cA \text{ is a  linear t-anticode with basis in }\cF \}\] is easier to compute and we show some examples where a closed-form expression can be found. First we consider the following equivalent definition of \(\mu_{\cF}\)

\begin{observation}
Let $t \in [N]$.  $\mu_{\cF}(t)$ coincides with the maximum cardinality of a subset $\G \subseteq \cF$ satisfying
\begin{enumerate}
    \item All $f_i \in \G$ are linear independent from each other over $\Ff$;
    \item All $v \in \langle \G \rangle$ have $\wt_\cF(v) \leq t$.
\end{enumerate}
\end{observation}
\begin{proof}
    Since \(d_{\cF}(x,y) = d_{\cF}(0,x-y) = \wt_{\cF}(x-y)\). A linear code \(\cA\) is a \(t\)-anticode if and only if \(\cA \subset B^{\cF}_t\). From this follows the equivalence of the two definitions.
    
\end{proof}
 
\(\mu_{\cF}(t)\) can be seen as the dimension of the larger subspace generated by elements in \(\cF\) contained in the ball of radius \(t\).
In particular bound \eqref{eq:Linear Code-Anticode Bound} and bound \eqref{eq:Relaxed Code-Anticode Bound for Projective Metrics} do not coincide exactly when the largest subspace \(\cA\) contained in the ball of radius \(t\), is not generated by elements in \(\cF\). The justification on why this does not \onequote\onequote generally" happen is that we are always guaranteed to have a subspace of dimension \(t\) contained in the ball of size \(t\) generated by elements in \(\cF\) (we only need to take the span of \(t\) linearly indipendent vectors in \(\cF\)). But this is not the case for subspaces of dimension \(t\) or larger which are not generated by any subset of elements in \(\cF\). Thus is would seem that subspaces generated by elements in \(\cF\) make good candidates for being the larger subspaces contained in the ball of size \(t\). We observe that in various classical metrics, such as Hamming metric and Rank metric these two bounds coincide. Although it must be noted that in general the two bounds are not equal.
\begin{example}
We give an example where bounds \eqref{eq:Linear Code-Anticode Bound} \eqref{eq:Relaxed Code-Anticode Bound for Projective Metrics} differ. In particular we construct a projective metric in which exists a \(3\) dimensional \(2\)-anticode which is not generated by elements in \(\cF\) but there doesn't exist any linear \(3\) dimensional \(2\)-anticode generated by elements of \(\cF\).
  Consider the spanning family \(\cF = \{f_1,\ldots, f_{14}\} \subset \bF_2^{10}\) where \(f_i = e_i\) for \(i = 1,\ldots, 10\) and

  \[
f_{11} = \begin{pmatrix}
1 \\
1 \\
1 \\
1 \\
0 \\
0 \\
1 \\
0 \\
0 \\
0 \\
\end{pmatrix}, \quad
f_{12} = \begin{pmatrix}
1 \\
1 \\
0 \\
0 \\
1 \\
1 \\
0 \\
1 \\
0 \\
0 \\
\end{pmatrix}, \quad
f_{13} = \begin{pmatrix}
0 \\
0 \\
1 \\
1 \\
1 \\
1 \\
0 \\
0 \\
1 \\
0 \\
\end{pmatrix}, \quad
f_{14} = \begin{pmatrix}
1 \\
1 \\
1 \\
1 \\
1 \\
1 \\
0 \\
0 \\
0 \\
1 \\
\end{pmatrix}
\]

Consider then the set of linearly independent vectors \(\cG \coloneqq \{v_1 = e_1+e_2, v_2 = e_3+e_4, v_3 = e_5+e_6\} \not\subset \cF\). It can be easily seen that \(\langle \cG \rangle \subset B^{\cF}_2\). However, the span of any subset of three elements of \(\cF\) is not contained in \(B^{\cF}_2\). This follows from the fact that the minimum distance of the parent code \(\pc\) is equal to \(\dH(\pc) = 6\). Let \(u_1, u_2, u_3 \in \mathcal{F}\). Consider \(v := u_1 + u_2 + u_3\); if this has \(\mathcal{F}\)-weight two, then there exist \(u_4, u_5 \in \mathcal{F}\) such that \(v = u_4 + u_5\). At least one of \(u_4\) and \(u_5\) is different from \(u_i, \; i \leq 3\). Thus the minimum distance $\mathcal{C} = ker(F)$ is smaller than $5$ and this gives a contradiction. The minimum distance of the parent code was computed algorithmically. The code can be found in \cite{projective_git}.

\hspace{1cm}

\textbf{General construction:}, let \(\cG < V \coloneqq \bF_q^N\) be a subspace dimension \(n_g \geq 3\) with minimum Hamming weight of \(\cG \) strictly larger than \(3\). Then consider the vector space \(W \coloneqq V \times \bF_q^{\cG}\). Consider the spanning family \(\cF \coloneqq \{e_1,\ldots,e_N,e_g, f_g \coloneqq e_g + g\}_{g \in \cG}\). We can easily see that \(\cG \subset \cB^{\cF}_2\). Next we check that no subset of independent vectors of \(\cF' \subset \cF\) of cardinality \(n_g\) has span \(\langle \cF'\rangle \subset \cB_2^{\cF}\). If such \(\cF'\) existed, than let \(f_1,f_2,f_3 \in \cF'\) be distinct elements of \(\cF'\), since \(f_1 + f_2 + f_3 \in \cB^{\cF}_2\), this implies that \(d(\pc_{\cF}) \leq 5\). 
But \(d(\pc_{\cF})\) is the cardinality of the smallest subset of linearly dependent vectors of \(\cF\). 
So let \(\sum_{g \in \cG} \left(\lambda_g f_g + \nu_ge_g\right) + \sum_{i=1}^N\mu_ie_i =0\) be the corresponding non trivial linear combination. From the definition of \(f_g\) it follows that \(\lambda_g = \nu_g\) for all \(g \in \cG\) and that \(\sum_{i=1}^N\mu_ie_i \in \cG\). Thus we have that \(d(\pc_{\cG})= \min_{g \in \cG\withoutzero}\left(\wtH(g) + 2\wt_{\cG}(g)\right) = d(\cG) + 2 > 5\) by hypothesis on the minimum distance of \(\cG\). This gives a contraddiction. Thus we have a subspace of dimension \(n_g\) of maximum weight \(2\) but no subspace generated by elements of \(\cF\) of dimension larger than 2 and with maximum weight \(2\). 
\end{example}

Let \(\cF\) be a spanning family on \(V = \bF_q^N\).

\begin{definition}
Let $t \in [N]$. We define $\mu_{\cF}(t)$ as the maximum cardinality of a subset $\G \subseteq \cF$ satisfying
\begin{enumerate}
    \item All $f_i \in \G$ are linear independent from each other over $\Ff$;
    \item All $v \in \langle \G \rangle$ have $\wt_\cF(v) \leq t$.
\end{enumerate}
\end{definition}

\textbf{Properties:}\\
Note that for  $t \in [N]$, we have (easy to check)
\begin{enumerate}
    
\item $\mu_{\cF}(t) \geq t$
\item (Strict growth of \(\mu_{\cF}(\cdot)\)) If $\mu_{\cF}(t) < N$ then $\mu_{\cF}(t+1) > \mu_{\cF}(t)$ \label{eq: strict growth of mu}
\item $\mu_{\cF}(t) \geq \mu_{\cF'}(t)$ for all \(\cF' \subset \cF\)
\item $\mu_{\cF}(t) = N$ iff $t \geq \max\{\wt_{\cF}(v) \,|\, v \in V\}$

 \begin{proof}
     (\ref{eq: strict growth of mu}) Consider $\cG \subset \cF$ such that \(|\cG|= \mu_{\cF}(t)\) and \(\Span(\cG) \subset B_{t}^{\cF}\), since \(\Span{\cG} \neq V\), there exists $f \in \cF$ such that $\cG' := \cG \cup \{f\}$ is linearly independent and $\Span \cG' = \Span \cG + \langle f \rangle \subset \cB_{t+1}^{\cF}$ thus $\mu_{\cF}(t+1) \geq \mu_{\cF}(t) + 1$.
 \end{proof}
\end{enumerate}

\begin{example}
Let $e_1,\ldots,e_N$ be the canonical basis of \(V\). Then the spanning family $\cF = \{e_1,\ldots,e_N\}$ will define a Hamming metric $d_\cF$ on $V$, and  for  all $t \in [N]$, we have 
\[
\mu_{\cF}(t) = t.
\]
\begin{proof}

We know that \(\mu_\cF(t) \geq t\) for Property 1. Further more for every \(\cG \subset \cF\) of cardinality \(t+1\), the element \(x = \sum_{g \in \cG} g \in \Span(\cG)\) has \(\wt_{\cF}(x) = t+1\). Thus \(\mu_\cF(t) < t + 1\) and this concludes.
\end{proof}
\end{example}

\begin{remark}
Let $V = \Ff^N_q$ with standard basis vectors $e_1,\ldots,e_N$. Then the spanning family $\cF = \{e_1,\ldots,e_N, \bm{1}=\sum_{i=1}^N e_i\}$ will define the phase rotation metric $d_\cF$ on $V$.
For all $x$ in $\bF_q^{N}$ we have

\begin{equation} 
\label{eq: phase weight}
\wt_{\cF}(x) = \min_{c \in \bF_q}(w_H(x), w_H(x-c\mathbf{1})+1)
\end{equation}

\begin{equation}
\label{eq: phase weight upper bound}
        \wt_{\cF}(x) \leq \lceil N - \frac{N}{q}\rceil 
\end{equation}

\end{remark}
\begin{proof}
    \eqref{eq: phase weight upper bound} For all $x \in V$ exists $c \in \bF_q$ such that $|\{i \in \{1,\ldots,N\}\mid x_i = c\} |\geq \lceil \frac{N}{q} \rceil$. Thus by\ \eqref{eq: phase weight} we have 
    \[
    \wt_{\cF}(x) \leq w_H(x-c\bfone ) +1 = N - |\{i \in \{1,\ldots,N\}\mid x_i = c\} | +1 \leq N -\lceil \frac{N}{q}\rceil+1
    \] which is equal to $\lceil N - \frac{N}{q}\rceil$ when $q \nmid N$. Otherwise if \(q \mid N\) it can be seen that the element in $V$ with the largest weight is $x$ having $ \frac{N}{q} $ coordinates equal to $c \in \bF$ for each $c$ in $\bF$ and thus $\wt_{\cF}(x) \leq N - \frac{N}{q}= \lceil N - \frac{N}{q}\rceil$.
\end{proof}

\begin{example}
Let \(\cF:=\{e_1,\ldots, e_N, \mathbf{1}\}\) be the spanning family inducing the phase rotation metric, then for all $t \in [N]$ we have 
\[
\mu_{\cF}(t) = \begin{cases}
t & \text{ if } t < \lceil N - \frac{N}{q}\rceil\\
N & \text{ otherwise }   \\
\end{cases}
\]

\begin{proof}
By Equation \eqref{eq: phase weight upper bound} for all $t \geq \lceil N - \frac{N}{q} \rceil $ we have $B_t^{\cF}= V$  and thus \(\mu_\cF(t) = N\).

\par    We now consider the case \(t < \lceil N - \frac{N}{q} \rceil =: \bar n\). From equation \eqref{eq: strict growth of mu} follows that if we have \(\mu_{\cF}(\bar n-1) = \bar n-1 \) then $\mu_{\cF}(t) = t$  for all $t \leq \bar n-1$. We do this by showing that for all $\cG \subset \cF$ linearly independent of cardinality $|\cG| = \bar n $ then \(\Span(\cG) \nsubseteq B^{\cF}_{\bar n-1}\).
    If \(\cG = \{e_1,e_2,\ldots,e_{\bar n}\}\), we construct $x \in \Span(\cG)$ such that $x \notin B^{\cF}_{\bar n-1}$. Let $k := \lceil \frac{\bar n}{q - 1} \rceil$ and \(\bF_q \withoutzero =\{p_1,\ldots,p_{q-1}\} \). we define $a_{kj + i} := p_j$ for all $i = 0,\ldots , k-1$ and $j = 1,\ldots,q-1$. We define \( x := \sum_{i =1}^{\bar n}a_i e_i \in \Span(\cG)\). Then $\wt_{\cF}(x) = \min_{c \in \bF_q}\{w_H(x), w_{H}(x-c\mathbf{1})+1\}$. If $\wt_{\cF}(x) = w_H(x) = \bar n$ then $x \notin B^{\cF}_{\bar n-1}$ and this concludes. By definition of $x$ we have that \(\min_{c \in \bF_q}\{ w_H{x-c\mathbf{1}}+1\} \geq N-k+1\), we want to show that $N-k+1 \geq \bar n $ since this would imply \( \min_{c \in \bF_q}\{ w_H{x-c\mathbf{1}}+1\} \geq \bar n \) and thus $\wt_{\cF}(x) = w_H(x) = \bar n$. This is true since:
    \[ \bar n + k = \lceil N - \frac{N}{q}\rceil + \lceil \frac{\bar n}{q-1}\rceil \leq N - \lfloor \frac{N}{q}\rfloor + \lceil \frac{\lceil N \frac{q-1}{q}\rceil}{q-1}\rceil = N - \lfloor \frac{N}{q}\rfloor + \lceil\frac{N}{q} \rceil = N \leq N + 1
    \]
    Where the second equality comes from the nested division propriety of the ceil function.
     We need to consider all the other possible linearly indipendent \(\cG' \subset \cF\) of cardinality $\bar n$.  <
     If $\cG'$ is comprised of vectors of the canonical base $\cB \subset \cF$ then there is a linear isomorphism \(L\) such that \(L(\cF) = \cF\) and \(L(\cG) = \cG'\). In particular \(L\) is a projective isometry, that is it preserves the \(\cF\)weight:\(\wt_\cF(L(x)) = \wt_\cF(x)\) for all \(x\) in \(V\). Thus $\Span(\cG') \subset B^{\cF}_{\bar n}$ if and only if $\cG \subset B^{\cF}_{\bar n}$ which we have seen is false. It only rests the case  $\cG'= \{e_1,\ldots,e_{\bar n-1},\bf1\}$. The same argument can applied with the linear isomorphism defined by $F(e_i) = e_i$ and \(F(e_{\bar n} = \bf1)\). Since again \(L(\cF)=\cF\) and thus \(L\) is a projective isometry.

\end{proof}

\end{example}

\begin{example}
Let $V = \Ff^{m \times n}$ be the space of $m \times n$ matrices. Then the spanning family $\cF = \{\text{rank 1 matrices in }V\}$ will define a rank metric $d_\cF$ on $V$, and  for  all $t \in [N]$ with $N = m n$ we have 
\[
\mu_{\cF}(t) = \begin{cases}
t \cdot \max\{m,n\} & \text{ if } t \leq \min\{m,n\}\\
N & \text{ if } t \geq \min\{m,n\}\\
\end{cases}
\]
This is because the largest code supported by only $t$ columns or only $t$ rows will have dimension $tm$ or $tn$ respectively.
\end{example}

\begin{theorem}[\textbf{General Singleton-type bound}] \label{thm: gen singleton}
Let $V$ be an $N$-dimensional vector space over $\Ff_q$.
Let $\C \subseteq V$ be a code and let $d = d_\cF(\C)$. Then
\[
|\C| \leq q^{N - \mu_{\cF}(d-1)} \leq q^{N-d+1}
\]
\end{theorem}

\begin{proof}
Write $\mu := \mu_{\cF}(d-1)$ and let $\G =\{g_1, \ldots, g_\mu\} \subseteq \cF$ be a subset of size $\mu$ satisfying
\begin{enumerate}
    \item All $g_i \in \G$ are linear independent from each other over $\Ff$;
    \item All $v \in \langle \G \rangle$ have $\wt_\cF(v) \leq d-1$.
\end{enumerate}
We can extend $G$ to a basis $\mathcal{B}$ for $V$:
\[
\mathcal{B} =\{g_1, \ldots, g_\mu, g_{\mu+1}, \ldots, g_{N}\}
\]
by picking vectors $g_{\mu+1}, \ldots, g_{N}$ that are  linearly independent from each other and from $\mathcal{G}$. We claim that the projection 
\[
\operatorname{Proj}: \C \to \Ff_q^{N - \mu}
\]
of a codeword $\sum_{i=1}c_i g_i \in \C$ onto its last $N-\mu$ coordinates $(c_{\mu+1}, c_{\mu + 2},  \ldots, c_{N}$)
is injective, thereby proving the theorem.

Let $a,b \in \C$ and suppose $\operatorname{Proj}(a) = \operatorname{Proj}(b)$. Then $\operatorname{Proj}(a-b) = 0$, so $a - b \in \langle \G \rangle$. But this means that
\[
d_\cF(a,b) = \wt_\cF(a - b) \leq d - 1 < d,
\]
so $a = b$.\end{proof}

\begin{example}
Let $V = \Ff^{m \times n}$ be the space of $m \times n$ matrices. Then the spanning family $\cF = \{\text{rank 1 matrices in }V\}$ will define a rank metric $d_\cF$ on $V$. The bound from Theorem \ref{thm: gen singleton} implies that any code $\C$ with $d_\cF(\C) = d$ will satisfy
\[
|\C| \leq q^{mn - (d-1)\max\{m,n\}} = q^{\max\{m,n\}(\min\{m,n\} - d + 1)}.
\]
This coincides with the usual Singleton-like bound for the rank-metric.
\end{example}

\section{Conclusion}\label{sec:conclusion}

In this work, we developed a comprehensive framework for \emph{projective metrics} or, equivalently, for integral convex scale–translation–invariant metrics on finite–dimensional vector spaces over finite fields.  
Starting from the notion of \emph{quotient weights}, we have shown several elementary tools and used them to streamline many proofs and existing constructions.  This provided the technical backbone for all subsequent sections.

\medskip\noindent
The main contributions can be summarized as follows.
\begin{itemize}
  \item We established new fundamental properties of projective metrics and showed that they coincide with the class of integral, convex, scale-translation-invariant metrics.  
  \item We exhibited an explicit bijection between isomorphism classes of projective metrics on a space~$V$ and Hamming–equivalence classes of linear subspaces (``parent codes’’) of~$\Ff_q^N$, thereby characterizing their full isometry groups.  
  \item We proved that \emph{every} scale–invariant metric embeds isometrically into a suitable projective metric, revealing projective metrics as natural ambient spaces for the study of general translation– and scale–invariant codes.  
  \item We analyzed the behavior of sphere sizes under metric constructions such as disjoint union and, via the notion of \emph{extended spanning families}, showed that the associated matroid completely determines the volume of a ball.  
  \item We derived a Singleton–type bound for projective metrics by adapting the code–anticode technique to the present setting.
\end{itemize}

Beyond these theoretical advances, we supplied numerous examples, counterexamples, and a \textsc{SageMath} script that automates many of the computations introduced in the paper.

\medskip\noindent
\textbf{Future directions.}  

\medskip\noindent
We hope that the unifying perspective offered and the tools developed in this paper will serve as a foundation for the study of projective metrics, further bridging the research on different metrics, and more broadly between coding theory, graph theory, matroid theory, and finite geometry.

%% file: Sections/appendix.tex
\newpage
\renewcommand{\thesubsection}{\Alph{subsection}}
\renewcommand{\thetheorem}{\Alph{subsection}.\arabic{theorem}}
\renewcommand{\thedefinition}{\Alph{subsection}.\arabic{definition}}
\renewcommand{\theexample}{\Alph{subsection}.\arabic{example}}
\renewcommand{\theproposition}{\Alph{subsection}.\arabic{proposition}}
\renewcommand{\theobservation}{\Alph{subsection}.\arabic{observation}}

\counterwithin{proposition}{subsection}
\counterwithin{definition}{subsection}

\begin{appendices}

\subsection{Quotient weights}
\label{appendix: quotient weights}
\quotientweightproperties*

{\itshape
\begin{enumerate}
    \item \textbf{(Maximality)} The quotient weight \(\wt_{\operatorname{quot},\varphi}\) is the largest weight on \(Y\) making \(\varphi\) a contraction: for any other weight \(\wt_Y\) such that \(\varphi: (X,\wt_X) \to (Y,\wt_Y)\) is a contraction, the identity map \(\operatorname{id}: (Y, \wt_{\operatorname{quot},\varphi}) \to (Y, \wt_Y)\) is a contraction.

    \item \textbf{(Universal Property)} If \(f: (X, \wt_X) \to (Z, \wt_Z)\) is a contraction with \(\ker(\varphi) \subseteq \ker(f)\), then there exists a unique contraction \(g: (Y, \wt_{\operatorname{quot},\varphi}) \to (Z, \wt_Z)\) such that \(f = g \circ \varphi\). Moreover,
    \[
    \ker(g) = \varphi(\ker(f)) \quad \text{and} \quad \wt_{\operatorname{quot},f} = \wt_{\operatorname{quot},g}.
    \]
    This property is summarized in the following commuting diagram:
    \[
    \begin{tikzcd}
    (X, \wt_X) \arrow[r, "\varphi", two heads] \arrow[rd, "f"'] & (Y, \wt_{\operatorname{quot},\varphi}) \arrow[d, dashed, "\exists! \, g"] \\
    & (Z, \wt_Z)
    \end{tikzcd}
    \]
    \item (\textbf{Isometry Property}) If \(\ker(\varphi) = \ker(f)\) and \(f\) is surjective, then \(g: (Y, \Qwt{\varphi}) \to (Y, \Qwt{f})\) as defined in (2) is an isometry. 
\end{enumerate}
}

\begin{proof}
Since statement (1) is a special case of (2) with \(Z = Y\) and \(f = \varphi\), it suffices to prove (2).

By the universal property of vector space quotients, there exists a unique linear map \(g: Y \to Z\) such that \(f = g \circ \varphi\). It remains to show that \(g\) is a contraction with respect to \(\wt_{\operatorname{quot},\varphi}\).

Let \(y \in Y\). By definition of the quotient weight, there exists \(x \in \varphi^{-1}(y)\) such that
\[
\wt_{\operatorname{quot},\varphi}(y) = \wt_X(x).
\]
Then, using the fact that \(f = g \circ \varphi\),
\[
\wt_Z(g(y)) = \wt_Z(f(x)) \leq \wt_X(x) = \wt_{\operatorname{quot},\varphi}(y),
\]
which shows that \(g\) is a contraction.

Now we prove that \(\ker(g) = \varphi(\ker(f))\). Since \(f = g \circ \varphi\), if \(x \in \ker(f)\), then \(g(\varphi(x)) = f(x) = 0\), so \(\varphi(x) \in \ker(g)\). Thus, \(\varphi(\ker(f)) \subseteq \ker(g)\).

For the reverse inclusion, let \(y \in \ker(g)\), so \(g(y) = 0\). Since \(\varphi\) is surjective, there exists \(x \in X\) such that \(y = \varphi(x)\), and hence \(f(x) = g(\varphi(x)) = 0\), i.e., \(x \in \ker(f)\). Therefore, \(y = \varphi(x) \in \varphi(\ker(f))\), and so \(\ker(g) \subseteq \varphi(\ker(f))\). We conclude \(\ker(g) = \varphi(\ker(f))\).

Finally, we show that the quotient weights \(\wt_{\operatorname{quot},f}\) and \(\wt_{\operatorname{quot},g}\) are equal.

Let \(z \in Z\). If \(z \notin \operatorname{Im}(f) = \operatorname{Im}(g)\), then both \(\wt_{\operatorname{quot},f}(z)\) and \(\wt_{\operatorname{quot},g}(z)\) are defined as \(+\infty\), so equality holds.

Now suppose \(z \in \operatorname{Im}(f)\). Let \(x \in X\) be such that \(f(x) = z\), and let \(y = \varphi(x)\), so \(g(y) = z\). Then:
\[
\wt_{\operatorname{quot},f}(z) = \min_{x' \in f^{-1}(z)} \wt_X(x').
\]
and by definition of \(\Qwt{g}\) and \(\wt_{\operatorname{quot},\varphi}\), we have:
\[
\wt_{\operatorname{quot},g}(z) = \min_{y' \in g^{-1}(z)} \wt_{\operatorname{quot},\varphi}(y') = \min_{y' \in g^{-1}(z)} \min_{x' \in \varphi^{-1}(y')}\wt_{X}(x')=\min_{x' \in \varphi^{-1}\left(g^{-1}(z)\right)} \wt_X(x').
\]
Since \(f= g \circ \varphi\), we have \(f^{-1}(z) = (g \circ \varphi)^{-1}(z) = \varphi^{-1} \left( g^{-1}(z)\right)\), thus
\[
\wt_{\operatorname{quot},f}(z) = \wt_{\operatorname{quot},g}(z).
\]
This completes the proof for (2). We prove (3). Since \(f\) is surjective, also \(g\) is surjective. Since \(\ker(f) = \ker(\varphi)\), since \( \ker(g) = \varphi(\ker(f)) = 0\), \(g\) is a bijection. 
Lastly, \(g\) is an isometry since for every \(y \in Y\):
\[
\Qwt{\varphi}(y)= \min_{x' \in \varphi^{-1}(y)}\wt_X(x') = \min_{x' \in f^{-1}(g(y))}\wt_X(x) = \Qwt{f}(g(y)).
\]
Where the second equality follows from \(f = g \circ \varphi\) and \(\ker(f) = \ker(\varphi)\).
\end{proof}

\subsection{Linear equivalence and projective metrics}
\label{appendix: linear equivalence}
\propProjectiveMetricsEquivalence*

\begin{proof} Let $\cF= \{\spn{f_1},\ldots,\spn{f_\N}\}$.\\
``$\Leftarrow$'' \,
Suppose $\wt(\cdot)=\wt_{L(\cF)}(\cdot)$ for some $L \in \GL(V)$. Let $x \in V$. Then for any $I \subseteq [\N]$ we have $x \in \sum_{i \in I}\spn{f_i}$ if and only if $L(x) \in \sum_{i \in I}L(\spn{f_i})$. Therefore $\wt_\cF(x) = \wt_{L(\cF)}(L(x)) = \wt(L(x))$, and thus $\wt_\cF(\cdot) \simeq \wt(\cdot)$.\\
``$\Rightarrow$''\, Suppose $\wt_\cF(\cdot) \simeq \wt(\cdot)$, i.e. $\wt_\cF(\cdot) = \wt(L(\cdot))$ for some $L \in \GL(V)$.  Let $y \in V$. Then for any $I \subseteq [\N]$ we have $y \in \sum_{i \in I}L(\spn{f_i})$ if and only if $L^{-1}(y) \in \sum_{i \in I}\spn{f_i}$. Therefore $\wt_{L(\cF)}(y)  = \wt_{\cF}(L^{-1}(y)) = \wt(y)$, and thus $\wt_{L(\cF)}(\cdot)  = \wt(\cdot)$.
\end{proof}

\subsection{Weight functions generated by families of subspaces}
\label{appendix: F-metrics}
Gabidulin and Simonis first introduced the concept of weight functions generated by families of subspaces in \cite{gabidulin1998metrics}. In their work, they considered projective metrics as a special case of metrics generated by families of subspaces. In this section, we demonstrate that these two notions are, in fact, equivalent.


\hspace{1cm}

\begin{remark}
For any subspace $U \leq V$, we can consider $\Gr_k(U)$ as a subset of $\Gr_k(V)$, since any $k$-dimensional subspace of $U$ is also a $k$-dimensional subspace of $V$. This is needed for the following proposition.
\end{remark}

\propProjectivePointForm*

\begin{proof}
The fact that \(\ppf(\cF)\) spans the same subspace of \(\cF\) is obvious.
Let $x \in \langle F \mid F \in \cF \rangle$ be any vector and write $\cF= \{F_1,\ldots,F_\N\} \subseteq \mathcal{P}(V)$. Suppose $I \subseteq [\N]$ is  a set of minimum cardinality that satisfies
$x \in \sum_{i \in I}\spn{F_i}.$
This means that $x$ can be written as a sum $x = \sum_{i \in I}f_i$
for certain non-zero vectors $f_i \in \spn{F_i}$. This implies
$x \in \sum_{i \in I}\spn{f_i}.$ 
Since each 1-dimensional subspace $\spn{f_i}$ is an element of $\Gr_1(\spn{F_i})$, and therefore of $\ppf(\cF)$,  we conclude that $\wt_{\cF}(x) = |I| \geq \wt_{\ppf(\cF)}(x)$.

Now write $\ppf(\cF) = \{G_1,\ldots,G_{\N\bm{'}}\} \subseteq  \Gr_1(V)$.  Suppose $J \subseteq [\N\bm{'}]$ is  a set of minimum cardinality that satisfies
$x \in \sum_{j \in J}G_j$. By definition of $\ppf(\cF)$ we can choose for every $G_j$ a basic set $F_j \in \cF$ such that $G_j \subseteq \spn{F_j}$ (some $F_j$'s chosen this way may coincide, but this is not a problem). It follows that $x \in \sum_{j \in J}\spn{F_j}$, and thus $\wt_{\cF}(x) \leq |J| = \wt_{\ppf(\cF)}(x)$.
\end{proof}

\subsection{Equivalence of Projective Metrics}\label{appendix: Equivalence of Projective Metrics}

\parentcodeEqProject*

\begin{proof}

First, we need to show that \(\Psi\) is well-defined. Specifically, if two projective metrics \(\wt_{\cF}\) and \(\wt_{\mathcal{F}'}\) are linear equivalent, then their parent codes are Hamming equivalent as well. 
By Proposition \ref{prop projective metrics equivalence} there exists a linear isomorphism \(L: V \rightarrow V\) such that \(L(\mathcal{F}) = \mathcal{F}'\).
Let \(\varphi\) be a parent function of \(\mathcal{F}\) and let $\cB := \{e_1,\cdots, e_{\N}\}\subset \Ff_q^{\N} $ the canonical base.
Since \(L (\varphi(\mathcal{B})) = L (\mathcal{F}) = \mathcal{F}'\), \(L \circ \varphi\) is a parent function of \(\mathcal{F}'\). 
Thus, \(\overline{\pc_{\mathcal{F}'}} = \overline{\text{ker}(L \circ \varphi)} = \overline{\text{ker}(\varphi)} = \overline{\pc_{\mathcal{F}}}\), confirming that \(\Psi\) is well-defined.
\par
To demonstrate that \(\Psi\) is injective, we proceed by contradiction. 
Let \(\bar \wt_{\cF} \neq \bar \wt_{\cF}\) be such that \(\overline{\pc_{\mathcal{F}}} = \Psi(\bar\wt_{\cF}) = \Psi(\bar\wt_{\mathcal{F}'}) = \overline{\pc_{\mathcal{F}'}}\). This implies the existence of two parent functions \(\varphi_{\mathcal{F}}\) and \(\varphi_{\mathcal{F}'}\) with the same kernel \(\pc = \pc_{\mathcal{F}} = \pc_{\mathcal{F}'}\). 
The associated quotient maps of these parent functions,
\(\overline{\varphi_{\mathcal{F}}}, \overline{\varphi_{\mathcal{F}'}}:\sfrac{\mathbb{F}_q^{\mathbb{N}}}{\pc} \to V\) act as isomorphisms from \(\sfrac{\mathbb{F}_q^{\mathbb{N}}}{\pc}\) to \(V\). 
Consequently, the function \(L = \overline{\varphi_{\mathcal{F}}} \circ \overline{\varphi_{\mathcal{F}'}}^{-1}\) from \(V\) to \(V\) is also an isomorphism. Given that \(\overline{\varphi_{\mathcal{F}}}([e_i]) = f_i\) and \(\overline{\varphi_{\mathcal{F}'}}([e_i]) = f'_i\), it follows that \(L(\mathcal{F}) = \mathcal{F}'\). 
This means \(\wt_{\cF}\) and \(w_{\mathcal{F'}}\) are linear equivalent, which gives a contradiction.
\par
Lastly, we show that $\Psi$ is surjective. Let $\bar{\pc} \in \bar{Gr}_{\N-N}(\mathbb{F}_q^{\N})$. We search for $\mathcal{F}$ such that $\Psi(\bar\wt_{\cF}) = \bar{\pc}$. 
The vector space $\mathbb{F}_q^{\N} / \pc$ has dimension $N$, so there exists an isomorphism $T: \mathbb{F}_q^{\N} / \pc \to V$.
Let $\pi: \mathbb{F}_q^{\N} \to \mathbb{F}_q^{\N} / \pc$ be the projection onto $\mathbb{F}_q^{\N} / \pc$, then define $\varphi := T \circ \pi$. The map $\varphi$ is surjective, and $\ker(\varphi) = \pc$. %
Consider $\mathcal{F} := \varphi(\mathcal{B})$, where $\mathcal{B}$ is the standard basis of $\mathbb{F}_q^{\N}$. 
We need to show that $\mathcal{F}$ is a spanning family for $V$ of size $\N$ (in general it could be a spanning family of size smaller than \(\N\) if its vectors are not pairwise independent). 
This would imply that its parent code is $\bar{\pc}$ and thus $\bar\cC = \Psi(\bar\wt_{\cF})$, proving that $\Psi$ is surjective.
To confirm that $\mathcal{F}$ is a spanning family of size $\N$, we need to show that its vectors are pairwise independent. 
Consider any non-trivial linear combination of vectors in $\mathcal{F}$:
\[
0 = \sum_{f \in \mathcal{F}} \lambda_f f = \sum_{i=1}^{\N} \lambda_{\varphi(e_i)} \varphi(e_i) = \varphi\left(\sum_{i=1}^{\N} \lambda_{\varphi(e_i)} e_i\right).
\]
This implies that $\sum_{i=1}^{\N} \lambda_{\varphi(e_i)} e_i \in \pc$. Since $\pc$ has a minimum weight greater than $3$, we have:
\[
\#\{f \in \mathcal{F} \mid \lambda_f \neq 0\} \geq 3,
\]
which means the vectors in $\mathcal{F}$ are pairwise independent.
\end{proof}
\subsection{Correction, detection and convexity}\label{appendix: correction, detection and convexity}

In this section we will introduce the notions of \textit{correction-capability} $\tau$ and the \textit{equal-detection threshold} $\sigma_{\operatorname{eq}}$. 
Suppose that from two distinct vectors $v_1,v_2 \in V$ one is selected (say $v$) and transmitted over a channel. 
The numbers $\tau,\sigma_{\operatorname{eq}}$ then indicate for which radii $\rho \in \Nn$, after receiving a vector $x \in V$ with $d(v,x) = \rho$, a minimum-distance decoder will always decode $x$ correctly to $v$, or if it could encounter a tie with $d(v_1,x) = d(v_2,x) = \rho$.  

Secondly, two notions of \textit{normality} are introduced that capture if $\tau$ and $\sigma_{\operatorname{eq}}$ only depend on the distance $d(v_1,v_2)$ in a natural way. Finally it is shown that a metric being normal for both notions is equivalent to being convex, and give two equivalent definitions of convexity.
\convexity*
\begin{proof}
Assume that property (1) holds for \(d\). Let \( v_1, v_2 \in V \) and consider a sequence 
\[
S = \{x_0 = v_1, x_1, \ldots, x_s = v_2\} \subset V
\]
such that \( d(x_k, x_{k+1}) = 1 \) for all \( k = 0, \ldots, s-1 \), and \( d(v_1, v_2) = s-1 \).

We aim to show that for all \( i \in \{0,1,\ldots, d(v_1,v_2)\} \), we have:
\[
d(v_1, x_i) = i \quad \text{and} \quad d(x_i, v_2) = d(v_1, v_2) - i.
\]

For the boundary cases: When \( i = 0 \), we have \( x_0 = v_1 \), so \( d(v_1, x_0) = d(v_1, v_1) = 0 \) and \( d(v_2, x_0) = d(v_2, v_1) - 0 \), which holds trivially. When \( i = d(v_1, v_2) = s-1 \), we have \( x_s = v_2 \), so \( d(v_1, x_s) = s-1 \) and \( d(x_s, v_2) = 0 \), which also holds trivially.

Now, consider the case where \( 0 < i < s-1 \). We first show that \( d(v_1, x_i) \leq i \). By applying the triangle inequality along the sequence:
\[
d(v_1, x_i) \leq \sum_{j=0}^{i-1} d(x_j, x_{j+1}) = i.
\]

Now, assume for contradiction that \( d(v_1, x_i) = n' < i \). Then, there exists a shorter sequence:
\[
S' = \{x'_0 = v_1, x'_1, \ldots, x'_{n'} = x_i\}
\]
such that \( d(x'_j, x'_{j+1}) = 1 \) for all \( j \). However, if we form the sequence \( S'' = S' \cup \{x_k\}_{k > i} \), we would obtain a sequence shorter than \( S \) connecting \( v_1 \) and \( v_2 \), contradicting the minimality assumption of \( S \). Thus, our assumption that \( d(v_1, x_i) < i \) must be false, meaning \( d(v_1, x_i) = i \).

Similarly, applying the same reasoning to \( d(x_i, v_2) \), we conclude \( d(x_i, v_2) = s-1 - i \), completing the proof. Now assume that (2) holds. Given any two elements \( v, w \in V \), we want to show that the minimum number of steps of distance 1 to move from \( v \) to \( w \) corresponds to the metric \( d(v,w) \).

We proceed by induction on \( d(v,w) \). Base case: If \( d(v,w) = 1 \), then there is a direct path from \( v \) to \( w \), which satisfies condition (1) trivially. Assume that for all distances less than \( n \), there exists a sequence of elements satisfying (1). Consider \( d(v,w) = n \). By (2), there exists an element \( x \) such that:
  \[
  d(v,x) = 1 \quad \text{and} \quad d(x,w) = n - 1.
  \]
  By the induction hypothesis, there exist sequences \( S_1 = \{v, x_1\} \) and \( S_2 = \{x_1, x_{2}, \ldots, x_n = w\} \), both of which satisfy condition (1). Concatenating these two sequences, we obtain a sequence of length \( n+1 \) that satisfies (1) for \( d(v,w) = n \), completing the proof.
\end{proof}

Let $V$ be a finite vector space endowed with an integral metric $d(\cdot,\cdot)$. Recall that for any $v \in V$ and $\rho \in \Nn$, the ball and sphere with radius $\rho$ centered at $v$ are respectively given by 
\[
B_{\rho}^d(v) := \left\{x \in V \mid d(v,x) \leq \rho \right\}, \quad S_{\rho}^d(v) := \left\{x \in V \mid d(v,x) = \rho \right\}.
\]

\begin{definition}(\cite{silva2008rank})\label{def: correction-normal}
Given a finite vector space $V$ with an integral metric $d(\cdot,\cdot)$, the \myfont{(distance)-correction capability} $\tau(v_1,v_2)$ of a pair of distinct vectors $v_1,v_2 \in V$ is the integer $\tau$ satisfying one of these equivalent statements:
\begin{itemize}
\item $
\tau = \min_{x \in V} \max\{d(v_1,x),d(x,v_2)\}-1;$

    \item $\tau$ is the largest integer such that
\[
B_{\tau}^d(v_1) \cap B_{\tau}^d(v_2) = \emptyset.
\]

\end{itemize}
Moreover, $d(\cdot,\cdot)$ is said to be \myfont{correction-normal} if for all $v_1,v_2 \in V$:
\[
\tau(v_1,v_2) = \left\lfloor \frac{d(v_1,v_2)-1}{2} \right\rfloor.
\]
\end{definition}
{
\color{black}

Informally, if the correction capability is $\tau$, this means that $v_1$ and $v_2$ can `distinguish' any vector that is at most distance $\tau$ away from one of $v_1$, $v_2$.  If an integral metric $d(\cdot,\cdot)$ is correction-normal, it fully describes the error correction capability between pairs of vector in a natural way, as the triangle inequality ensures that  $\tau(v_1,v_2) \geq \left\lfloor \frac{d(v_1,v_2)-1}{2} \right\rfloor$ is always satisfied.

\begin{definition}\label{def: equal-detection threshold}
Given a finite vector space $V$ with integral metric $d(\cdot,\cdot)$, the \myfont{equal-detection threshold} $\sigma_{\operatorname{eq}}(v_1,v_2)$ of a pair of vectors $v_1,v_2 \in V$ is the unique integer $s \in \Nn$ that, if such $s$ exists, satisfies one of these equivalent statements:
\begin{itemize}
\item 
\[
s = d(B_{s}^d(v_1),v_2) =d(B_{s}^d(v_2),v_1);
\]

\item 
\[
s = \min_{\substack{x \in V \\ d(v_1,x) \leq s}} \{d(x,v_2)\} = \min_{\substack{y \in V \\ d(y,v_2) \leq s}} \{d(v_1,y)\};
\]  
\item 
\[
B_s^d(v_1) \cap B_{s-1}^d(v_2) = B_{s-1}^d(v_1) \cap B_s^d(v_2) = \emptyset  \quad \text{ and } \quad 
S_s^d(v_1) \cap S_s^d(v_2) \neq \emptyset.
\]
\end{itemize}

If such $s$ does not exist, we set $\sigma_{\operatorname{eq}}(v_1,v_2) = 0$ as convention.
The metric $d(\cdot,\cdot)$ is called \myfont{equal-detection-normal} if for all $v_1,v_2 \in V$:
\[
d(v_1,v_2) = 2 \sigma_{\operatorname{eq}}(v_1,v_2)
\]
whenever $\sigma_{\operatorname{eq}}(v_1,v_2) \neq 0$.
\end{definition}
Note that uniqueness of $s$ (if it exists) follows from the fact that $d(B_{s}^d(v_1),v_2) \geq d(B_{s'}^d(v_1),v_2)
$ whenever $s < s'$.


An integral metric $d(\cdot,\cdot)$ being equal-detection normal indicates that vectors of odd distance apart have zero equal-detection threshold, and that all non-zero equal-detection thresholds are fully described by $d(\cdot,\cdot)$ in a natural way.

\begin{example}
    Consider the vector space $V = \Ff_2^4$ endowed with the weight function
    \[
    \wt_V(v) := \begin{cases}
       \wt_{\operatorname{H}}(v) & \text{ if } v \neq (1111)\\
        3 & \text{ if } v = (1111)
    \end{cases}
    \]
    for every $v \in V$, with $\wt_{\operatorname{H}}$ the Hamming weight. Since the metric $d_V$ induced by $\wt_V$ is the ordinary Hamming metric with one modification, it is still correction-normal but not equal-detection-normal anymore. This is because $\sigma_{\operatorname{eq}}(0000, 1111) = 2$ while $d_V(0000, 1111) = 3$. 
\end{example}

\begin{example}
    Consider the vector space $W = \Ff_2^2$ endowed with the weight function
    \[
    \wt_W(v) := 2
       \wt_{\operatorname{H}}(v) 
    \]
    for every $v \in W$, with $\wt_{\operatorname{H}}$ the Hamming weight. Now the metric $d_W$ induced by $\wt_W$ is equal-detection-normal but not correction-normal because $\tau(00, 10) = 1$ and $\lfloor (d_W(00, 10)-1)/2 \rfloor = 0$. 
\end{example}

The implication from convex to correction-normal is also discussed and proven in \cite{silva2008rank}.

\convexcorrectdetect*
\begin{proof}
 
Let $d(\cdot,\cdot)$ be an integral metric on a finite vector space $V$. 

Suppose $d(\cdot,\cdot)$ is convex, and let $v_1,v_2 \in V$ be distinct. Let $t = \left\lfloor (d(v_1,v_2)-1)/2 \right\rfloor$. By convexity there is an $x \in V$ with $d(v_1,x) = t+1$ and $d(x,v_2) = d(v_1,v_2)-(t+1) \leq t+1$ and so $B_{t+1}^d(v_1) \cap B_{t+1}^d(v_2) \neq \emptyset$. Hence $\tau(v_1,v_2) = t$ implying that $d(\cdot,\cdot)$ is correction-normal.

Moreover, if $d(v_1,v_2)$ is odd, then for $s \geq t+1$ it holds that $B_{s}^d(v_1) \cap B_{s-1}^d(v_2) \neq \emptyset$ and for $s \leq t$ that $S_{s}^d(v_1) \cap S_{s}^d(v_2) = \emptyset$. So in the case that $d(v_1,v_2)$ is odd we have $\sigma_{\operatorname{eq}}(v_1,v_2) = 0$. If  $d(v_1,v_2)$ is even, then $t+1 =  d(v_1,v_2)/2$, and for $s = t+1$ we have $B_s^d(v_1) \cap B_{s-1}^d(v_2) = B_{s-1}^d(v_1) \cap B_s^d(v_2) = \emptyset$  and $S_s^d(v_1) \cap S_s^d(v_2) \neq \emptyset$. This proves equal-detection-normality.\\

Now suppose $d(\cdot,\cdot)$ is correction-normal and equal-detection-normal.  Suppose by strong induction that convexity holds for any $v_1,v_2 \in V$ with $d(v_1,v_2) \leq d$. Let $w_1,w_2 \in  V$ with $d(w_1,w_2) = d+1$. Since $d(\cdot,\cdot)$ is correction-normal, there exists 
without loss of generality an $x \in V$ with $d(w_1,x) \leq \lfloor\frac{d}{2}\rfloor + 1$ and  $d(x,w_2) = \lfloor\frac{d}{2}\rfloor + 1$. If there is such an $x$ with $d(w_1,x) = \lfloor\frac{d}{2}\rfloor$, then by the triangle inequality applied to \(d(w_1,w_2)=d+1\), $d$ is even and
\[
d(w_1,x) = \tfrac{d}{2} \quad \text{and} \quad d(x,w_2)=\tfrac{d}{2}+1.
\]
By strong induction, the convexity property applied to $(w_1,x)$ and $(x,w_2)$ implies the existence of sequences  $(w_1 = x_1, x_2,\ldots,x_{d/2} = x)$  and  $(x_{d/2}, x_{d/2+1},\ldots,x_{d+1}=w_2)$ with $d(x_i, x_{i+1}) = 1$ for each $i$. Concatenation of these sequences shows the convexity property for $(w_1,w_2)$.

If on the other hand  there is no $x \in V$ with $d(w_1,x) < \lfloor\frac{d}{2}\rfloor + 1$ and  $d(x,w_2) = \lfloor\frac{d}{2}\rfloor + 1$ (nor with $w_1$ and $w_2$ swapped), but there does exist $x'\in V$ with $d(w_1,x') = d(x',w_2) = \lfloor\frac{d}{2}\rfloor + 1$, 
\[
\sigma_{\operatorname{eq}}(w_1,w_2) = \lfloor\tfrac{d}{2}\rfloor + 1 \quad \text{ and } \quad d(w_1,w_2) = d+1 = 2(\lfloor\tfrac{d}{2}\rfloor + 1)
\]
by equal-detection normality and hence $d$ is odd. Again by strong induction, the convexity property applied to $(w_1,x)$ and $(x,w_2)$ implies the existence of sequences  $(w_1 = x_1, x_2,\ldots,x_{\lfloor d/2 \rfloor + 1} = x)$  and  $(x_{\lfloor d/2 \rfloor + 1}, x_{\lfloor d/2 \rfloor + 2},\ldots,x_{d+1}=w_2)$ with $d(x_i, x_{i+1}) = 1$ for each $i$, and concatenation of these sequences shows the convexity property for $(w_1,w_2)$.
\end{proof}
}



\subsection{Matroids of Spanning Family fully describe intersections of two subspaces}\label{appendix: matroid subspacex intersection}
\begin{observation}
    \begin{enumerate}
        \item Let \(F,G \in M(\cF)\) and $c \subset F\cup G$ be a cycle then exist $a_f \in \bF\withoutzero,\; b_{g} \in \bF\withoutzero$ for all $f \in c \cap F$ and $g \in c \cap G$ such that 
        \[
        v_c \coloneqq \sum_{f \in F \cap c}a_ff = \sum_{g \in c \cap G}b_{g}g \in \langle c \cap F \rangle \cap \langle c \cap G \rangle \subset  \langle  F \rangle \cap \langle G \rangle
        \]
        \item Furthermore if $c \subset F \cup G$ is a circuit of $M(\cF)$ then \(\dim( \langle c \cap F \rangle \cap \langle c \cap G \rangle))=1\).
        \item For every \( x \in \fgcap{F}{G} \setminus \langle F \cap G\rangle\) there exist a cycle \(c \subset F \cup G\) such that \( x \in \cfgcap{c}{F}{G}\).
    \end{enumerate}
\end{observation}
\begin{proof}
    (1) follows from the definition of cycle in $M(\cF)$. For (2), by (1) we have that $\dim(\cap F \rangle \cap \langle c \setminus \{\bar f\} \cap G \rangle) \geq 1$.
    We observe that for all $f \in c$, $c \setminus \{f\} \in M(\cF)$, thus \(\langle c \setminus \{f\} \cap F \rangle \cap \langle c \setminus \{f\} \cap G \rangle) = \{0\} \).
    Consider \(v,w \in \langle c \cap F \rangle \cap \langle c \cap G \rangle)\) non null. There exists $a_f,b_f \in \bF$ such that \(v = \sum_{f \in c \cap F} a_ff\) and \(w = \sum_{f \in c \cap F}b_ff \). Fix $\bar f$ such that $a_{\bar f} \neq 0$, then 
    \[
     w - b_{\bar f}a_{\bar f}^{-1}v \in \langle c \setminus {\bar f} \cap F \rangle \cap \langle c \setminus \{\bar f\} \cap G \rangle = \{0\}
    \]
    Thus \( w \in \Span(v)\) and $\dim(\cap F \rangle \cap \langle c \setminus \{\bar f\} \cap G \rangle) = 1$. 
\end{proof}

\begin{proposition}
    Let \(F,G \in M(\cF)\) such that \(\dim(\fgcap{F}{G})=d \geq 1\). If \(|F \cap G| = k\), then there exists \(c_1,\ldots,c_{d-k} \subset F \cup G\) circuits and $v_{c_i} \in \cfgcap{c_i}{F}{G}$ for \(i = 1,\ldots, d\) such that \(\langle v_{c_1},\ldots,v_{c_{d-k}}, F \cap G \rangle = \fgcap{F}{G}\). We refer to \(\{v_{c_i}\}_{i \in [d]}\) and to \(\{c_i\}_{i \in [d]}\)  as a \emph{bases of circuits}. \\ Further more we can chose a bases of circuits such that for all \(i = 1, \ldots, d-k\) exists \(e_i \in c_i \setminus \cup_{j \neq i}c_j\).
\end{proposition}

\begin{proof}
    Let \(C\) be the set of circuits contained in \(F \cup G\). For all \(c \in C\) pick \(v_c \in \cfgcap{c}{F}{G}\). We only need to prove that \(\langle v_c \mid c \in C \rangle + \langle F \cap G\rangle = \fgcap{F}{G}\). Clearly  \( \langle v_c \mid c \in C \rangle + \langle F \cap G\rangle \subset \fgcap{F}{G}\). Viceversa, let \( v \in \fgcap{F}{G} \setminus \langle F \cap G \rangle \) and \(\bar c \subset F \cup G\) be a cycle with minimum cardinality \(k = |\bar c|\) such that \( v \in  \cfgcap{\bar c}{F}{G}\). We prove that \(v \in \langle v_c \mid c \in C \rangle + \langle F \cap G\rangle\) by induction on \(k\). If \(k = 2\) then \(\bar c\) is a circuit therefore \(  v \in  \cfgcap{\bar c}{F}{G} \subset \langle v_c \mid c \in C \rangle \). For the induction step. Since \(v \notin \langle F \cap G \rangle\), there exists \(e \in \bar c \setminus (F \cap G)\) and  \(e \in c \subset \bar c\) such that \(c \) is a circuit. If \(c = \bar c\) we finished. Otherwise let \(a_f,b_f \neq 0\) such that \(v = \sum_{f \in \bar c \cap F}a_ff\) and \(v_c = \sum_{f \in c \cap F}b_ff\).  Then \( 0 \neq w \coloneqq v - a_eb_e^{-1}v_c \in \fgcap{F}{G} \neq 0\) because \( c \neq \bar c\) and then \(w \in \cfgcap{\bar c \setminus \{e\}}{F}{G} \neq \{0\}\). If \(w \in \langle F \cap G \rangle\) we finished. Otherwise, \(\bar c \setminus \{e\}\) is a cycle and \(|\bar c \setminus \{e\}|=k-1\). By induction we have \(w \in \langle v_c \mid c \in C \rangle + \langle F \cap G \rangle\). Thus \(v = w + a_eb_e^{-1}v_c \in \langle v_c \mid c \in C \rangle + \langle F \cap G\rangle\) and this conclude the first part of the proof. \hfill \break 
    
    We now want to prove that we can chose a bases of circuits such that for all \(i = 1, \ldots, d-k\) exists \(e_i \in c_i\) such that \(e_i \notin \cup_{j \neq i}c_j\). Let \(S \subset F \cup G\) be maximal among the elements of \(M(\cF)\) contained in \( F \cup G\). For all \(e \in F \cup G \setminus S\) let \(c_e\) be the only circuit contained in \(S \cup \{e\}\). Clearly \(\langle v_{c_e} \mid e \in F \cup G \setminus S \rangle + \langle F \cap G \rangle\subset \langle v_c \mid c \in C\rangle = \fgcap{F}{G}\). Viceversa consider \( v \in \fgcap{F}{G} \setminus \langle F \cap G \rangle\). There exists \(a_f,b_g \in \bF\) for all \(f \in F, g\in G\) and \(m \in \langle F \cap G \rangle \)such that
    \[
    v' := v - m = \sum_{s \in S \cap F}a_ss + \sum_{f \in F \setminus S}a_ff = \sum_{s \in S \cap G}b_ss + \sum_{g \in G \setminus S}a_gg.
    \]
    Since \(S \in M(\cF)\), then \(\cfgcap{S}{F}{G} = \{0\}\). If \(a_f = 0,b_g =0\) for all \(f \in F \setminus S, \, g \in G \setminus S\) then \(v \in \cfgcap{S}{F}{G} = \{0\}\) thus \( v = 0\). Otherwise we "eliminate" the non null coefficients in \(G\setminus S)\) of \(v\), since \(v_{c_g} \in \langle F \rangle\): \[
    v'' \coloneqq v - m - \sum_{\substack{g \in G \setminus S \\ b_g \neq 0}}b_g\pi_g(v_{c_g})^{-1}v_{c_g} \in \fgcap{G \cap S}{F}
    \]
    Where \(\pi_g(v_{c_e}) \in \bF \withoutzero\) are the coefficients such that \(v_{c_e} = \sum_{g \in G \cap c_e}\pi_g(c_{c_e})g\).Then we "eliminate" the non null coefficients in \(F \setminus S\) of \(v\), since for all \(f \in F \setminus S\) we have \(v_{c_e} \in \langle G \cap S \rangle\):
    \[
    v''' \coloneqq v'' - \sum_{\substack{f \in F \setminus S \\ a_f \neq 0}}a_f\pi_f(v_{c_f})^{-1}v_{c_f} \in \fgcap{G \cap S}{F\cap S} = \{0\}
    \]
    Thus \(v''' = 0\) and \[v =  \sum_{\substack{f \in F \setminus S \\ a_f \neq 0}}a_f\pi_f(v_{c_f})^{-1}v_{c_f} + \sum_{\substack{g \in G \setminus S \\ b_g \neq 0}}b_g\pi_g(v_{c_g})^{-1}v_{c_g} + m \in \langle v_{c_e} \mid F \cup G \setminus S \rangle + \langle F \cap G\rangle  \]
\end{proof}
\subsection{Extended Family of Phase Rotation Metric Example}

\begin{example}\label{appendix: example matroid equivalence}
Let \(V\) be a vector space of dimension \(N\) over a finite field \(\mathbb{F}\). Let \(\cB \subset V\) be a basis of \(V\), and define
\[
w \coloneqq \sum_{v \in \cB} v.
\]
Consider the projective metric on \(V\) induced by the spanning family \(\cF \coloneqq \cB \cup \{w\}\). This metric corresponds to the Phase Rotation metric (\ref{example: Phase Rotation Metric}) in the space \(\mathbb{F}^N\), where vectors are expressed in the basis \(\cB\).

\medskip

We claim that the extended family \(\overline{\cF}\) of \(\cF\) is
\[
\overline{\cF} \coloneqq \Bigl\{\, w - \sum_{g \in I} g \;\Bigm|\; I \subset \cB \Bigr\},
\]
i.e., the set of all vectors in \(V\) whose coordinates (in the basis \(\cB\)) take at most two values, one of which is zero. The proof proceeds by induction on \(N\). The base case (N=1) is trivial, since \(V\) is one-dimensional and \(\cB\) has only one vector.
For the inductive step, assume the statement holds for all dimensions less than \(N\). All the vectors in \(\overline{\cF}\) are such that \(\langle v \rangle = \bigcap_{i=1}^{N-1} \langle G_i \rangle\), where \(G_i \in M(\cF)\) and \(\lvert G_i \rvert = N-1\). We must show that 
\[
\spn{v} \;=\; \spn{\,w - \sum_{g \in I}g\,}
\quad
\text{for some } I \subset \cB.
\]

\medskip

\noindent
\textbf{Case 1:} Suppose \(w \in G_i\) for all \(i\). Then
\[
\langle v \rangle \;=\; \bigcap_{i=1}^{N-1} \spn{G_i} \;=\; \spn{w}.
\]

\noindent
\textbf{Case 2:} There exists some \(G_i\) with \(w \notin G_i\). That is \(G_i = \cB \setminus \{g'\}\) for some \(g' \in \cB\).
Set \(V' \coloneqq \langle G_i \rangle\). Since \(G_i\) is a set of \((N-1)\) linearly independent vectors, it forms a basis \(\cB' = G_i \subset V'\). Define 
\[
w' \coloneqq \sum_{g \in \cB'} g = w - g'
\quad\text{and}\quad
\cF' \coloneqq \cB' \cup \{w'\}.
\]
We aim to show that for all other \(j\), the intersection
\[
\langle G_j \rangle \,\cap\, \langle G_i \rangle
\]
is a hyperplane in \(V'\) generated by a subset of \(\cF'\). Then, by the induction hypothesis (applied to the space \(V'\) of dimension \((N-1)\)), we conclude that 
\(\langle v \rangle\) is spanned by a vector in \(\overline{\cF'}\subset \overline{\cF}\), of the form \(w' - \sum_{g \in I'} g = w - g' - \sum_{g \in I'}g\). Translating this back to the original family \(\cF\), we get the desired vector \(w - \sum_{g \in I} g\), where \(I = I' \cup \{g'\}\).

\medskip

\noindent
\textbf{Subcase 2a:} If 
\(\lvert G_i \cap G_j \rvert = N-2\), then
\[
H_{ij} \coloneqq \langle G_j \rangle \,\cap\, \langle G_i \rangle 
= \langle G_i \cap G_j \rangle,
\]
 is a hyperplane in \(V'\) generated by elements of \(G_i \subset \cF'\).

\medskip

\noindent
\textbf{Subcase 2b:} 
Otherwise, if \(\lvert G_i \cap G_j \rvert = N-3\), we have $w, g' \in G_j$.
Furthermore, since \(w' \in \spn{G_i}\) 
we have \(
w' 
= w \;-\; g'
\) is also in \(\spn{G_j}\). Thus,
\[
\langle G_j \rangle \,\cap\, \langle G_i \rangle 
= \langle G_i \cap G_j \rangle + \langle w' \rangle,
\]
which is also a hyperplane in \(V'\) generated by elements in \(\cF'\).

\medskip

By induction, every one-dimensional intersection \(\langle v \rangle\) is spanned by some vector of the form \(w - \sum_{g \in I} g\). Hence
\(\overline{\cF} = \bigl\{ w - \sum_{g \in I} g : I \subset \cB \bigr\}\), as claimed.
\end{example}


\end{appendices}

%% file: main.bbl
\begin{thebibliography}{10}

\bibitem{bitzer2024}
Sebastian Bitzer, Alberto Ravagnani, and Violetta Weger.
\newblock Weighted-hamming metric for parallel channels.
\newblock In {\em 2024 IEEE International Symposium on Information Theory (ISIT)}, pages 2616--2621, 2024.

\bibitem{byrne2021tensor}
Eimear Byrne and Giuseppe Cotardo.
\newblock Tensor codes and their invariants.
\newblock {\em arXiv preprint arXiv:2112.08100}, 2021.
\newblock \url{https://arxiv.org/pdf/2112.08100}.

\bibitem{byrne2019tensor}
Eimear Byrne, Alessandro Neri, Alberto Ravagnani, and John Sheekey.
\newblock Tensor representation of rank-metric codes.
\newblock {\em SIAM Journal on Applied Algebra and Geometry}, 3(4):614--643, 2019.
\newblock \url{https://epubs.siam.org/doi/pdf/10.1137/19M1253964}.

\bibitem{catterall2006public}
N~Catterall, Ernst~M Gabidulin, Bahram Honary, and Vitaly~A Obernikhin.
\newblock Public key cryptosystem based metrics associated with grs codes.
\newblock In {\em 2006 IEEE International Symposium on Information Theory}, pages 729--733. IEEE, 2006.
\newblock \url{https://ieeexplore.ieee.org/abstract/document/4036059}.

\bibitem{firer2021alternative}
Marcelo Firer.
\newblock Alternative metrics.
\newblock In {\em Concise Encyclopedia of Coding Theory}, pages 555--574. Chapman and Hall/CRC, 2021.
\newblock \url{http://arxiv.org/abs/1911.12396}.

\bibitem{projective_git}
[Author's First~Name] Frulcino.
\newblock Projective metrics.
\newblock \url{https://github.com/frulcino/projective_metrics}, 2024.
\newblock Accessed: 2024-06-28.

\bibitem{gabidulin1998hard}
E~Gabidulin and M~Bossert.
\newblock Hard and soft decision decoding of phase rotation invariant block codes.
\newblock In {\em 1998 International Zurich Seminar on Broadband Communications. Accessing, Transmission, Networking. Proceedings (Cat. No. 98TH8277)}, pages 249--251. IEEE, 1998.
\newblock \url{https://ieeexplore.ieee.org/abstract/document/670272}.

\bibitem{gabidulin1971class}
EM~Gabidulin.
\newblock A class of two-dimensional codes correcting lattice-pattern errors.
\newblock In {\em Proc. of the 2-nd International Symposium on Information Theory}, pages 44--47, 1971.

\bibitem{gabidulin2012brief}
Ernst Gabidulin.
\newblock A brief survey of metrics in coding theory.
\newblock {\em Mathematics of Distances and Applications}, 66:66--84, 2012.
\newblock \url{http://foibg.com/ibs_isc/ibs-25/ibs-25-p06.pdf}.

\bibitem{gabidulin2003codes}
Ernst~M Gabidulin and Vitaly~A Obernikhin.
\newblock Codes in the vandermonde f-metric and their application.
\newblock {\em Problems of Information Transmission}, 39(2):159--169, 2003.
\newblock \url{https://link.springer.com/article/10.1023/A:1025165820188}.

\bibitem{gabidulin1997metrics}
Ernst~M Gabidulin and Juriaan Simonis.
\newblock Metrics generated by a projective point set.
\newblock In {\em Proceedings of IEEE International Symposium on Information Theory}, page 248. IEEE, 1997.
\newblock \url{https://ieeexplore.ieee.org/document/613163}.

\bibitem{gabidulin1998metrics}
Ernst~M Gabidulin and Juriaan Simonis.
\newblock Metrics generated by families of subspaces.
\newblock {\em IEEE Transactions on Information Theory}, 44(3):1336--1341, 1998.
\newblock \url{https://ieeexplore.ieee.org/document/669429}.

\bibitem{gabidulin1998perfect}
Ernst~M Gabidulin and Juriaan Simonis.
\newblock Perfect codes for metrics generated by primitive 2-error-correcting binary bch codes.
\newblock In {\em Proceedings. 1998 IEEE International Symposium on Information Theory (Cat. No. 98CH36252)}, page~68. IEEE, 1998.
\newblock \url{https://ieeexplore.ieee.org/abstract/document/708651}.

\bibitem{Grandis2007}
M.~Grandis.
\newblock Categories, norms and weights.
\newblock {\em Journal of Homotopy and Related Structures}, 2(2):171--186, 2007.
\newblock \url{http://eudml.org/doc/224503}.

\bibitem{Hamming}
R.~W. Hamming.
\newblock Error detecting and error correcting codes.
\newblock {\em The Bell System Technical Journal}, 29(2):147--160, 1950.

\bibitem{helleseth1979weight}
Tor Helleseth.
\newblock The weight distribution of the coset leaders for some classes of codes with related parity-check matrices.
\newblock {\em Discrete Mathematics}, 28(2):161--171, 1979.
\newblock \url{https://www.sciencedirect.com/science/article/pii/0012365X79900931}.

\bibitem{jurrius2009extended}
Relinde Jurrius and Ruud Pellikaan.
\newblock The extended coset leader weight enumerator.
\newblock {\em WIC, Benelux}, 2009.
\newblock \url{http://www.win.tue.nl/~ruudp/paper/54.pdf}.

\bibitem{MacWilliams}
F.~Jessie MacWilliams and N.~J.~A. Sloane.
\newblock The theory of error-correcting codes.
\newblock 1977.

\bibitem{oxley2003matroid}
James Oxley.
\newblock What is a matroid.
\newblock {\em Cubo}, 5(179):780, 2003.

\bibitem{pinheiro2019combinatorial}
Jerry~Anderson Pinheiro, Roberto~Assis Machado, and Marcelo Firer.
\newblock Combinatorial metrics: Macwilliams-type identities, isometries and extension property.
\newblock {\em Designs, Codes and Cryptography}, 87:327--340, 2019.

\bibitem{roth1996tensor}
Ron~M Roth.
\newblock Tensor codes for the rank metric.
\newblock {\em IEEE Transactions on Information Theory}, 42(6):2146--2157, 1996.
\newblock \url{https://ieeexplore.ieee.org/iel1/18/12145/00556603.pdf}.

\bibitem{shannon}
C.~E. Shannon.
\newblock A mathematical theory of communication.
\newblock {\em Bell System Technical Journal}, 27(3):379--423, 1948.

\bibitem{silva2008rank}
Danilo Silva and Frank~R. Kschischang.
\newblock On metrics for error correction in network coding.
\newblock {\em IEEE Transactions on Information Theory}, 55(12):5479--5490, 2009.

\end{thebibliography}
